\documentclass[11pt,a4paper]{article}

\usepackage[utf8]{inputenc}
\usepackage[T1]{fontenc}
\usepackage[english]{babel} 
\usepackage{amsmath}
\usepackage{amsfonts}
\usepackage{amsthm}
\usepackage{amssymb}
\usepackage{makeidx}
\usepackage{graphicx}
\usepackage{lmodern}
\usepackage[left=2cm,right=2cm,top=2cm,bottom=2cm]{geometry}

\usepackage{color}
\usepackage{comment}
\usepackage[dvipsnames]{xcolor}
\usepackage{graphicx}
\usepackage{framed}
\usepackage{tikz}
\usepackage{calrsfs}
\DeclareMathAlphabet{\pazocal}{OMS}{zplm}{m}{n}
\usepackage{bm}

\usepackage{fourier-orns}
\usepackage{mathtools}
\usepackage{relsize}
\usepackage{dsfont}
\usepackage{comment}

\usepackage{hyperref} 
\hypersetup{
	linktocpage = true,
	colorlinks = true,
	linkcolor = {RoyalBlue},
	urlcolor = {RoyalBlue},
	citecolor = {Green}}

\newcommand{\B}{\mathbb{B}}
\newcommand{\R}{\mathbb{R}}

\newcommand{\Apazo}{\pazocal{A}}
\newcommand{\Epazo}{\pazocal{E}}
\newcommand{\Fpazo}{\pazocal{F}}

\newcommand{\Kpazo}{\pazocal{K}}

\newcommand{\Mpazo}{\pazocal{M}}
\newcommand{\Bpazo}{\pazocal{B}}
\newcommand{\Cpazo}{\pazocal{C}}
\newcommand{\Qpazo}{\pazocal{Q}}
\newcommand{\Hpazo}{\pazocal{H}}
\newcommand{\Npazo}{\pazocal{N}}
\newcommand{\Lpazo}{\pazocal{L}}

\newcommand{\Dpazo}{\pazocal{D}}
\newcommand{\Rpazo}{\pazocal{R}}
\newcommand{\Spazo}{\pazocal{S}}
\newcommand{\Opazo}{\pazocal{O}}
\newcommand{\Wpazo}{\pazocal{W}}

\newcommand{\Acal}{\mathcal{A}}
\newcommand{\Dcal}{\mathcal{D}}
\newcommand{\Ecal}{\mathcal{E}}

\newcommand{\Lcal}{\mathcal{L}}

\newcommand{\Pcal}{\mathcal{P}}

\newcommand{\Tcal}{\mathcal{T}}

\newcommand{\Id}{\textnormal{Id}}

\newcommand{\D}{\textnormal{D}}

\newcommand{\Tan}{\textnormal{Tan}}

\newcommand{\supp}{\textnormal{supp}}
\newcommand{\dom}{\textnormal{dom}}

\newcommand{\Lip}{\textnormal{Lip}}
\newcommand{\AC}{\textnormal{AC}}
\newcommand{\Graph}{\textnormal{Graph}}

\newcommand{\Div}{\textnormal{div}}

\newcommand{\dist}{\textnormal{dist}}
\newcommand{\textbn}[1]{\textnormal{\textbf{#1}}}
\newcommand{\co}{\overline{\textnormal{co}} \hspace{0.05cm}}

\newcommand{\dsf}{\textnormal{\textsf{d}}}

\newcommand{\Bgamma}{\boldsymbol{\gamma}}

\newcommand{\Bnu}{\boldsymbol{\nu}}

\newcommand{\Bmu}{\boldsymbol{\mu}}

\newcommand{\BB}[1]{\textcolor{blue}{#1}}

\newcommand{\INTDom}[3]{\int_{#2} #1 \textnormal{d} #3}
\newcommand{\INTSeg}[4]{\int_{#3}^{#4} #1 \textnormal{d} #2}
\newcommand{\NormL}[3]{\parallel \hspace{-0.1cm} #1 \hspace{-0.1cm} \parallel _ {L^{#2}(#3)}}
\newcommand{\NormLp}[3]{\parallel \hspace{-0.1cm} #1 \hspace{-0.1cm} \parallel _ {\pazocal{L}^{#2}(#3)}}
\newcommand{\NormC}[3]{\left\| #1  \right\| _ {C^{#2}(#3)}}
\newcommand{\Norm}[1]{\parallel \hspace{-0.1cm} #1 \hspace{-0.1cm} \parallel}

\newcommand{\tto}{\rightrightarrows}

\newtheorem{Def}{Definition}[section]
\newtheorem{thm}[Def]{Theorem}
\newtheorem{prop}[Def]{Proposition}
\newtheorem{rmk}[Def]{Remark}
\newtheorem{lem}[Def]{Lemma}

\numberwithin{equation}{section}
\renewcommand{\epsilon}{\varepsilon}

\newtheorem{taggedhypx}{Hypotheses}
\newenvironment{taggedhyp}[1]
    {\renewcommand\thetaggedhypx{#1}\taggedhypx}
    {\endtaggedhypx}

\title{On the Viability and Invariance of Proper Sets under \\ Continuity Inclusions in Wasserstein Spaces}

\author{Benoît Bonnet-Weill\footnote{CNRS, LAAS, 7 avenue du colonel Roche, F-31400 Toulouse, France. \textit{E-mail}: \texttt{benoit.bonnet@laas.fr} (Corresponding author)} \,and Hélène Frankowska\footnote{CNRS,  IMJ-PRG,  UMR  7586,  Sorbonne  Université,  4  place  Jussieu,  75252  Paris,  France. \hfill \hspace{3.5cm} \textit{E-mail}: \texttt{helene.frankowska@imj-prg.fr}}}

\begin{document}

\maketitle

\begin{abstract}
In this article, we derive conditions for the existence of solutions to state-constrained continuity inclusions in Wasserstein spaces whose right-hand sides may be discontinuous in time. These latter are based on a fine investigation of the infinitesimal behaviour of the underlying reachable sets, through which we show that up to a negligible set of times, every admissible velocity of a continuity inclusion can be approximately realised as the metric derivative of a solution of the dynamics, and vice versa. Building on these results, we are able to establish necessary and sufficient geometric conditions for the viability and invariance of stationary and time-dependent constraint sets which involve a suitable notion of contingent cones in Wasserstein spaces, and presented in ascending order of generality. We then close the article by exhibiting two prototypical examples of constraints sets appearing in applications for which one can compute relevant subfamilies of contingent directions.
\end{abstract}

{\footnotesize
\textbf{Keywords:} Continuity Inclusions, Optimal Transport, Viability, Invariance, Wasserstein Geometry,

\hspace{1.75cm} Dynamics  with Time Discontinuities.

\vspace{0.25cm}

\textbf{MSC2020 Subject Classification:} 28B20, 34G25, 46N20, 49Q22
}

\tableofcontents


\section{Introduction}
\setcounter{equation}{0} \renewcommand{\theequation}{\thesection.\arabic{equation}}

Recent times have witnessed a surge of interest for the mathematical analysis of large-scale approximations of particle systems. During the past two decades, a series of seminal works concerned with the mean-field approximation of cooperative dynamics \cite{Carrillo2009,Carrillo2010,HaLiu}, the theory of mean-field games \cite{Cardaliaguet2010,Huang2006,Lasry2007} and the sparse control of multiagent systems \cite{Caponigro2013,Caponigro2015,FornasierPR2014} have given rise to several research currents focusing on dynamical and variational problems whose aim is to describe the macroscopic behaviour of many-body systems. Amongst these latter, \textit{mean-field control} is a research branch whose main object lies in designing scale-free and efficient control signals for large microscopic systems, by finely understanding the interplay that exists between discrete models and their continuous approximations. From a technical standpoint, these inquiries often boil down to studying variational problems in the space of probability measures, and are commonly approached using optimal transport techniques and Wasserstein geometry, in the spirit of the reference treatises \cite{AGS,OTAM,villani1}. Without aiming at full exhaustivity, we point the reader to the manuscripts \cite{LipReg,Cavagnari2022,Fornasier2019,Fornasier2014} for various existence and qualitative regularity results on deterministic mean-field optimal control problems, as well as to the following broad series of works dealing with optimality conditions, either in the form of a Pontryagin maximum principle \cite{Averboukh2023,Bongini2017,PMPWassConst,PMPNeurODEs,SetValuedPMP,PMPWass,Pogodaev2016,Pogodaev2020} or Hamilton-Jacobi-Bellman equations \cite{Badreddine2022,Cavagnari2018,Cavagnari2020,Jimenez2020}. We also mention the references \cite{AlbiPZ2014,SparseJQMF,Carrillo2022,ControlKCS} that propose astute control strategies to stir collective systems towards specific asymptotic patterns, and finally \cite{AlbiCFK2017,PMPNeurODEs,Burger2021,Burger2020,Pogodaev2023} for various numerical methods in the context of mean-field optimal control. 

\bigskip

Motivated by this blooming interest for variational problems in measure spaces, several research groups have been investigating relevant generalisations of the core concepts of \textit{set-valued analysis} to the setting of mean-field control \cite{Badreddine2022,ContInc,SetValuedPMP,SemiSensitivity,ContIncPp,CavagnariMP2018,CavagnariMQ2021,Jimenez2020}, a lively trend that reached more recently other closely related topics, such as mean-field games \cite{Arjmand2021,Cannarsa2021} and the study of sufficient conditions for the well-posedness of dynamics in measure spaces \cite{CavagnariSS2022,Karimghasemi2021}. It is de facto widely accepted that the language of correspondences, differential inclusions and generalised gradients provides, in many cases, a synthetic and powerful framework in which most problems stemming from the calculus of variations, games and control theory can be encompassed, as supported by the series of reference monographs \cite{Aubin1984,Aubin1990,Clarke,Vinter} tracing back to the nineteen eighties. For these reasons, the authors of the present manuscript introduced in \cite{ContInc,ContIncPp} a notion of differential inclusion in Wasserstein spaces that is tailored to the study of mean-field control problems. Therein, given a set-valued mapping $V : [0,T] \times \Pcal_p(\R^d) \tto C^0(\R^d,\R^d)$, we defined \textit{continuity inclusions} as the set-valued dynamical systems
\begin{equation*}
\partial_t \mu(t) \in -\Div_x \Big( V(t,\mu(t)) \mu(t) \Big),
\end{equation*}
whose solutions are understood as the collection of all curves $\mu(\cdot) \in \AC([0,T],\Pcal_p(\R^d))$ for which there exists a \textit{measurable selection} $t \in [0,T] \mapsto v(t) \in V(t,\mu(t))$ such that the \textit{continuity equation}
\begin{equation*}
\partial_t \mu(t) + \Div_x(v(t) \mu(t)) = 0 
\end{equation*}
holds in the sense of distributions. While other notions of solutions for differential inclusions in measures spaces had already been proposed in some preexisting works, see e.g. \cite{Averboukh2018,CavagnariMP2018,Jimenez2020}, the approach we just described seemed more natural as well as necessary for several reasons. On the one hand, it remained coherent with the classical theory of set-valued dynamics, see e.g. \cite{Aubin1984} and \cite[Chapter 11]{Aubin1990}, as well as with the geometric interpretation of Wasserstein spaces developed in \cite{AGS,Otto2001} wherein the space $(\Pcal_p(\R^d),W_p)$ is seen as a fiber bundle over which continuity equations essentially play the same role as ODEs for differential manifolds. On the other hand, it complied with one of the most important and desired features of differential inclusions, already formulated in the pioneering article \cite{Filippov1962}, which stipulates that solutions of control systems should be in one-to-one correspondence with those of their set-valued counterparts. In \cite{ContInc,ContIncPp}, based on this definition of differential inclusions, we proved analogues in the setting of Wasserstein spaces of the Filippov estimates and Peano existence theorem, as well as a relaxation principle and a compactness criterion for the underlying solution sets. These fundamental results are known to be extremely useful to investigate the fine properties of optimal control problems, both in the deterministic \cite{Frankowska2018,Vinter} and stochastic \cite{PMPSto} settings, while enjoying natural generalisations to study e.g. evolution equations in Banach spaces \cite{Frankowska1990,FrankowskaMM2018} or mutational dynamics in metric spaces \cite{Badreddine2022Bis,Frankowska2022}. We also point to our recent works \cite{SetValuedPMP,SemiSensitivity}, where we successfully applied such set-theoretic techniques to derive optimality conditions for optimal control problems in Wasserstein spaces, and to study certain qualitative properties of their solutions. 

While permitting to handle a variety of relevant dynamical models in measure spaces, the aforedescribed framework does not cover systems with \textit{state-constraints}, which appear nonetheless in a wide variety of applications ranging from game theory \cite{Cardaliaguet2016} to pedestrian dynamics \cite{Cristiani2017} and traffic flows \cite{Carlier2008}, and more recently to the dynamical formulations of deep neural networks \cite{E2017}. Historically, the problem of ensuring that a differential inclusion admits trajectories that remain within a given set starting from any initial condition was coined \textit{viability}, whereas that of ensuring that all such trajectories be viable was usually called \textit{invariance}. The first results in this direction were established in \cite{Nagumo1942} for differential equations, while their natural counterparts for differential inclusions with stationary right-hand sides later followed in \cite{Bebernes1970} and \cite{Haddad1981}. To this day, the farthest-reaching viability theorems for classical differential inclusions can be found in \cite{Frankowska1996,Frankowska1995} -- which inspired several of our contributions --, wherein the viability of general time-dependent constraint tubes is proven both in the Carathéodory and Cauchy-Lipschitz frameworks. Besides modelling incentives, viability and invariance results can be used to study the existence and uniqueness of viscosity solutions for several nonlinear partial differential equations, see e.g. \cite{Buckdahn2000,Frankowska1995}, as well as to investigate sufficient stability conditions for differential inclusions \cite{AubinSurvey1990}. These viewpoints -- like many others stemming from set-valued analysis -- present the advantage of being readily transposable beyond the setting of finite-dimensional vector spaces, as illustrated e.g. by their recent applications to problems in metric spaces \cite{Averboukh2018,Badreddine2022,Badreddine2022Bis,CavagnariMQ2021}. 

\bigskip

In this article, we study the viability and invariance properties of constraints sets under the action of continuity inclusions in the Wasserstein spaces $(\Pcal_p(\R^d),W_p)$ with $p \in (1,+\infty)$. Given a constraint tube $\Qpazo : [0,T] \tto \Pcal_p(\R^d)$ with proper images (see Definition \ref{def:Proper} below), we provide necessary and sufficient conditions ensuring that either all or some of the solutions of the Cauchy problem 
\begin{equation}
\label{eq:IntroCauchy}
\left\{
\begin{aligned}
& \partial_t \mu(t) \in - \Div_x \Big( V(t,\mu(t)) \mu(t) \Big), \\
& \mu(\tau) = \mu_{\tau},  
\end{aligned}
\right.
\end{equation}
are such that $\mu(t) \in \Qpazo(t)$ for all times $t \in [\tau,T]$, wherein $\tau \in [0,T]$ and $\mu_{\tau} \in \Qpazo(\tau)$ are both arbitrary. These results, which are discussed in Section \ref{section:CauchyLip}, rely on a careful analysis of the infinitesimal behaviour of the reachable sets of \eqref{eq:IntroCauchy}. This latter is the object of Section \ref{section:Quali}, and can be heuristically summarised as follows. In Theorem \ref{thm:Liminf}, we show that for each element $v_{\tau} \in V(\tau,\mu_{\tau})$ taken at some adequate pair $(\tau,\mu_{\tau}) \in [0,T] \times \Pcal_p(\R^d)$ and every $\epsilon > 0$, there exists a solution $\mu_{\epsilon}(\cdot)$ of \eqref{eq:IntroCauchy} which satisfies
\begin{equation*}
W_p \Big( \mu_{\epsilon}(\tau+h^{\epsilon}) , (\Id + h^{\epsilon} v_{\tau})_{\sharp} \mu_{\tau} \Big) \leq \epsilon h^{\epsilon}
\end{equation*}
whenever $h^{\epsilon}>0$ is sufficiently small. In other words, each admissible velocity can be approximately realised -- on a subset of times of full $\Lcal^1$-measure and up to an arbitrarily small error -- as the \textit{right metric derivative} of an admissible curve. In Theorem \ref{thm:Limsup}, we complete this result by showing that for $\Lcal^1$-almost every $\tau \in [0,T]$, each solution $\mu_{\epsilon}(\cdot)$ of \eqref{eq:IntroCauchy} and any sequence $h_i \to 0$, there exists $v_{\tau}^{\epsilon} \in V(\tau,\mu_{\tau})$ such that
\begin{equation*}
W_p \Big( \mu_{\epsilon}(\tau+h_{i_k}^{\epsilon}) , (\Id + h_{i_k}^{\epsilon} v_{\tau}^{\epsilon})_{\sharp} \mu_{\tau} \Big) \leq \epsilon |h_{i_k}^{\epsilon}| , 
\end{equation*}
along a subsequence $h_{i_k}^{\epsilon} \to 0$. Stated otherwise, one can always find an admissible velocity which approximately represents -- again on a subset of full $\Lcal^1$-measure, up to an arbitrary small error, and a subsequence -- the \textit{metric derivative} of a solution of the Cauchy problem. It is worth noting that while these results are proven under Cauchy-Lipschitz assumptions in the present paper, we do believe that they remain valid under weaker Carathéodory-Peano hypotheses similar to those of \cite[Section 3]{ContIncPp}. The main issue in proving so would be to replace the estimates on flow maps by their counterparts for measures concentrated on characteristic curves, whose existence is ensured by the famed superposition principle of Ambrosio, see e.g. \cite[Theorem 3.4]{AmbrosioC2014}. 

Let us now discuss more in depth the main contributions of this article, which are the necessary and sufficient conditions ensuring that a constraints tube $\Qpazo : [0,T] \tto \Pcal_p(\R^d)$ is either viable or invariant under the dynamics of $V : [0,T] \times \Pcal_p(\R^d) \tto C^0(\R^d,\R^d)$ for $p \in (1,+\infty)$. For the sake of clarity, we start our investigation thereof by considering stationary constraints sets. In this context, the relevant geometric objects allowing to study viability and invariance properties are the \textit{metric contingent cones} to $\Qpazo$, defined by
\begin{equation*}
T_{\Qpazo}(\nu) := \bigg\{ \xi \in \Lpazo^p(\R^d,\R^d,\nu) ~\, \textnormal{s.t.}~  \liminf_{h \to 0^+} \tfrac{1}{h} \dist_{\Pcal_p(\R^d)} \Big( (\Id + h \xi)_{\sharp} \nu \, ; \Qpazo \Big) = 0 \bigg\}
\end{equation*}
for each $\nu \in \Qpazo$, where $\Lpazo^p(\R^d,\R^d;\nu)$ stands for the seminormed space of Borel measurable and $p$-integrable maps against $\nu$, and whose expression is akin to that recently introduced in \cite{Badreddine2022}. Building on this notion, we show in Theorem \ref{thm:ViabStationary} that if the geometric compatibility condition
\begin{equation}
\label{eq:IntroViab1}
V(t,\nu) \cap \co T_{\Qpazo}(\nu) \neq \emptyset
\end{equation}
is satisfied for $\Lcal^1$-almost every $t \in [0,T]$ and each $\nu \in \Qpazo$, where ``$\co$'' denotes the \textit{closed convex hull} of a set, then $\Qpazo$ is viable. That is, there exists then for every $(\tau,\mu_{\tau}) \in [0,T] \times \Qpazo$ a solution $\mu(\cdot)$ of the Cauchy problem \eqref{eq:IntroCauchy} satisfying $\mu(t) \in \Qpazo$ for all times $t \in [\tau,T]$. Reciprocally, we prove in Theorem \ref{thm:ViabStationaryNesc} that if $\Qpazo$ is viable, then it necessarily holds that 
\begin{equation}
\label{eq:IntroViab1Bis}
V(t,\nu) \cap T_{\Qpazo}(\nu) \neq \emptyset 
\end{equation}
for $\Lcal^1$-almost every $t \in [0,T]$ and each $\nu \in \Qpazo$. We point the interested reader to our recent work \cite{ViabCDC}, where we leveraged a weaker version of these viability theorems to prove the existence of exponentially stable solutions to a class of continuity inclusions via the second method of Lyapunov, see also the second example in Section \ref{section:Cone} below. In Theorem \ref{thm:Invariance}, we subsequently show that the stronger condition 
\begin{equation}
\label{eq:IntroInv1}
V(t,\nu) \subset \co T_{\Qpazo}(\nu),
\end{equation}
which is assumed to hold for $\Lcal^1$-almost every $t \in [0,T]$ and each $\nu \in \Qpazo$, is equivalent to the invariance of $\Qpazo$ under the dynamics of \eqref{eq:IntroCauchy}, namely to the fact that $\mu(t) \in \Qpazo$ for all times $t \in [\tau,T]$ and every admissible curve $\mu(\cdot)$ starting from $\mu_{\tau} \in \Qpazo$ at time $\tau \in [0,T]$.  

We then turn our attention to the more involved scenario in which the constraints are allowed to be time-dependent. In Theorem \ref{thm:NecessaryConds}, we start by showing that regardless of its regularity, if the tube $\Qpazo : [0,T] \tto \Pcal_p(\R^d)$ is viable for \eqref{eq:IntroCauchy}, namely if for any $\tau \in [0,T]$ and each $\mu_{\tau} \in \Qpazo(\tau)$, there exists a solution of \eqref{eq:IntroCauchy} satisfying $\mu(t) \in \Qpazo(t)$ for all time $t \in [\tau,T]$, then necessarily   
\begin{equation}
\label{eq:IntroViab2}
\big( \{ 1 \} \times V(t,\nu) \big) \cap T_{\Graph(\Qpazo)}(t,\nu) \neq \emptyset 
\end{equation}
for $\Lcal^1$-almost every $t \in [0,T]$ and all $\nu \in \Qpazo(t)$, where
\begin{equation*}
\Graph(\Qpazo) := \Big\{ (t,\nu) \in [0,T] \times \Pcal_p(\R^d) ~\,\textnormal{s.t.}~ \nu \in \Qpazo(t) \Big\}
\end{equation*}
denotes the \textit{graph} of $\Qpazo : [0,T] \tto \Pcal_p(\R^d)$ . In addition, it follows in this context that $\Qpazo : [0,T] \tto \Pcal_p(\R^d)$ is actually \textit{left absolutely continuous} (see Definition \ref{def:AC} below), as it inherits some of the regularity properties of the reachable sets of \eqref{eq:IntroCauchy}. Similarly, if one posits that $\Qpazo : [0,T] \tto \Pcal_p(\R^d)$ is invariant for \eqref{eq:IntroCauchy}, that is if $\mu(t) \in \Qpazo(t)$ for all times $t \in [\tau,T]$ along every solution of \eqref{eq:IntroCauchy} starting from some $\mu_{\tau} \in \Qpazo(\tau)$ at time $\tau \in [0,T]$, then it must hold that
\begin{equation}
\label{eq:IntroInv2}
\big( \{1\} \times V(t,\nu) \big) \subset T_{\Graph(\Qpazo)}(t,\nu) 
\end{equation}
for $\Lcal^1$-almost every $t \in [0,T]$ and all $\nu \in \Qpazo(t)$. 

Unlike the aforedescribed necessary implications, sufficient viability and invariance conditions for time-dependent constraints call for separate analyses depending on the regularity of the tube $\Qpazo : [0,T] \tto \Pcal_p(\R^d)$. When the latter is \textit{absolutely continuous} (see Definition \ref{def:AC} below), we prove in Theorem \ref{thm:ViabTubeAC} that $\Qpazo : [0,T] \tto \Pcal_p(\R^d)$ is viable for \eqref{eq:IntroCauchy} whenever the geometric condition 
\begin{equation}
\label{eq:IntroViab2Bis}
\big( \{ 1 \} \times V(t,\nu) \big) \cap \co T_{\Graph(\Qpazo)}(t,\nu) \neq \emptyset 
\end{equation}
holds for $\Lcal^1$-almost every $t \in [0,T]$ and all $\nu \in \Qpazo(t)$. We stress that in this context, the convexification of the contingent directions requires much more care than in the stationary case, as one must combine tangent velocities corresponding to possibly different time instants. Similarly to what precedes, we then show in Theorem \ref{thm:InvarianceTube} below that $\Qpazo : [0,T] \tto \Pcal_p(\R^d)$ is invariant for \eqref{eq:IntroCauchy} provided that 
\begin{equation}
\label{eq:IntroInv2Bis}
V(t,\nu) \subset \co T_{\Graph(\Qpazo)}(t,\nu)
\end{equation}
for $\Lcal^1$-almost every $t \in [0,T]$ and all $\nu \in \Qpazo(t)$. Finally in Theorem \ref{thm:ViabTubeACL}, we address sufficient viability and invariance conditions when the tube $\Qpazo : [0,T] \tto \Pcal_p(\R^d)$ is merely left absolutely continuous. This regularity framework -- which, as already illustrated above, appears naturally when studying continuity inclusions with state-constraints -- is very similar to the one we previously discussed, with the added subtlety that in this context one cannot convexify the contingent directions anymore. Thence, the sufficient condition which ensures that $\Qpazo : [0,T] \tto \Pcal_p(\R^d)$ is viable for \eqref{eq:IntroCauchy} becomes 
\begin{equation}
\label{eq:IntroViab3}
\big( \{1\} \times V(t,\nu) \big) \cap T_{\Graph(\Qpazo)}(t,\nu) \neq \emptyset
\end{equation}
for $\Lcal^1$-almost every $t \in [0,T]$ and all $\nu \in \Qpazo(t)$. Similarly to the stationary and absolutely continuous cases, one can show that $\Qpazo : [0,T] \tto \Pcal_p(\R^d)$ is invariant for \eqref{eq:IntroCauchy} if the geometric condition 
\begin{equation}
\label{eq:IntroInv3}
\big( \{1\} \times V(t,\nu) \big) \subset T_{\Graph(\Qpazo)}(t,\nu) 
\end{equation}
holds for $\Lcal^1$-almost every $t \in [0,T]$ and each $\nu \in \Qpazo(t)$. At this stage, it is worth noting that while conditions \eqref{eq:IntroViab1} to \eqref{eq:IntroViab3} involve the whole contingent cone to the constraints, it is sufficient in practice to test their validity only on some nice subsets of tangent directions which are easy to compute. This fact is expounded in Section \ref{section:Cone}, in which we exhibit collections of tangents for two simple, yet frequently encountered families of constraints sets, defined respectively in terms of support inclusions or as lifted epigraphs. In the first situation, we display certain adjacent directions which are amenable to computations, while in the second one, we are able to fully characterise a relevant subset of the contingent cone. 

In terms of bibliographical positioning, this work can be seen as a far-reaching extension of \cite{Badreddine2022} by the second author, in which viability and invariance properties are established as a means to prove the well-posedness of general Hamilton-Jacobi-Bellman equations for optimal control problems in Wasserstein spaces, under more restrictive regularity assumptions on the data. We also point to the independent works \cite{Averboukh2018,Averboukh2021,CavagnariMQ2021} which focus on the study of viability properties for another class of set-valued dynamics in Wasserstein spaces, which we already alluded earlier. The main differences between the latter notion and the one considered in the present article and the related works of both authors is thoroughly discussed in \cite{ContInc}. 

\bigskip

The paper is organised as follows. In Section \ref{section:Preliminaries}, we start by recalling a list of preliminary material pertaining to optimal transport, set-valued analysis, and continuity inclusions. Subsequently, we discuss in Section \ref{section:Quali} the infinitesimal behaviour of reachable sets to continuity inclusions, which constitute novel contributions to the theory on which the main results of Section \ref{section:CauchyLip} crucially rely. This latter section is then split into two parts in which we discuss viability and invariance results for continuity inclusions, starting with the case of stationary constraints in Section \ref{subsection:ViabCauchy}. We then approach time-dependent constraints in Section \ref{subsection:ViabCauchyAbs}, wherein the necessary and sufficient conditions for viability and invariance are exposed separately for absolutely continuous and left absolutely continuous constraints. Finally, we show a couple of relevant examples of constraints sets and compute some of their tangent directions in Section \ref{section:Cone}, and close the paper by an appendix that contains the proofs of two intermediate technical results.  


\section{Preliminaries}
\label{section:Preliminaries}
\setcounter{equation}{0} \renewcommand{\theequation}{\thesection.\arabic{equation}}

In this section, we fix the notations that will be used throughout the manuscript, and list a series of prerequisites of optimal transport theory, set-valued analysis and Wasserstein geometry.  


\subsection{Measure theory and optimal transport}

In this first preliminary section, we recollect common notions of measure theory and optimal transport, for which we refer the reader to the monographs \cite{AmbrosioFuscoPallara,EvansGariepy} and \cite{AGS,OTAM,villani1} respectively. 


\paragraph*{Function spaces and measure theory.}

Given two complete separable metric spaces $(X,\dsf_X(\cdot,\cdot))$ and $(Y,\dsf_Y(\cdot,\cdot))$, we denote by $C^0(X,Y)$ the space of continuous functions from $X$ into $Y$, and likewise by $C^0_b(X,Y)$ the subspace of continuous and bounded functions. In this context, we will use the notation $\Lip(\varphi \, ; \Omega) \in \R_+ \cup \{+\infty\}$ for the Lipschitz constant of a map $\varphi : X \to Y$ over some subset $\Omega \subset X$. In the particular case where $(X,\dsf_X(\cdot,\cdot)) = (I,|\cdot|)$ for a closed interval $I \subset \R$, we shall write $\AC(I,Y)$ for the collection of absolutely continuous curves valued in $Y$. We shall also denote by $C^{\infty}_c(\R^d,\R^m)$ the space of infinitely differentiable functions with compact support from $\R^d$ into $\R^m$ for some $d,m \geq 1$.

In what follows, we let $\Pcal(\R^d)$ be the space of Borel probability measures defined over $(\R^d,|\cdot|)$. Recalling that the latter is a subset of the topological dual of $C^0_b(\R^d,\R)$, it can be endowed with the usual weak-$^*$ or \textit{narrow topology}, which is the coarsest topology such that
\begin{equation*}
\mu \in \Pcal(\R^d) \mapsto \INTDom{\varphi(x)}{\R^d}{\mu(x)} \in \R
\end{equation*}
defines a continuous mapping for every $\varphi \in C^0_b(\R^d,\R)$. In this context, given $\mu \in \Pcal(\R^d)$ and some $p \in [1,+\infty)$, the notation $(\Lpazo^p(\R^d,\R^d;\mu),\NormLp{\cdot}{p}{\R^d,\R^d\,;\mu})$ will refer to the seminormed vector space of Borel maps from $\R^d$ into itself which are $p$-integrable with respect to $\mu$ (see e.g. \cite[Chapter 2.4]{Bogachev}). We will also denote by $\Lcal^1$ the standard $1$-dimensional Lebesgue measure, and given a closed interval $I \subset \R$ along with a separable Banach space $(X,\Norm{\cdot}_X)$, we let $(L^1(I,X),\NormL{\cdot}{1}{I})$ stand for the Banach space of (all equivalence classes of) maps which are $\Lcal^1$-measurable and integrable in the sense of Bochner, see for instance \cite[Chapter II]{DiestelUhl}.

In the following definition, given some interval $I \subset \R$, we recall the classical notions of \textit{one-sided density points} of an $\Lcal^1$-measurable subset $\Acal \subset I$, along with that of \textit{one-sided Lebesgue points} for a Lebesgue integrable map defined over $I$. It is a well-known result in measure theory that these sets both have full $\Lcal^1$-measure in $\Acal$ and $I$, respectively. 

\begin{Def}[Lebesgue and density points]
\label{def:LebesgueDensity}
Given a Lebesgue measurable set $\Acal \subset I$, its \textnormal{one-sided density points} are defined as the elements $\tau \in \Acal$ satisfying 
\begin{equation*}
\frac{\Lcal^1([\tau,\tau+h] \cap \Acal)}{|h|} ~\underset{h \to 0}{\longrightarrow}~ 1. 
\end{equation*}
Similarly, given a map $f \in L^1(I,X)$, we denote by $\Tcal_f \subset I$ the subset of its \textnormal{one-sided Lebesgue points}, which are the elements $\tau \in I$ at which 
\begin{equation*}
\frac{1}{h} \INTSeg{\Norm{f(t) - f(\tau)}_X}{t}{\tau}{\tau+h} ~\underset{h \to 0}{\longrightarrow}~ 0. 
\end{equation*}
\end{Def}

For any real number $p \in [1,+\infty)$, we denote by $\Pcal_p(\R^d)$ the subset of Borel probability measures whose \textit{moment of order $p$}, defined by 
\begin{equation*}
\Mpazo_p(\mu) := \bigg( \INTDom{|x|^p}{\R^d}{\mu(x)} \bigg)^{1/p},
\end{equation*}
is finite. In what follows, we write $f_{\sharp} \mu \in \Pcal(\R^d)$ for the \textit{image measure} of an element $\mu \in \Pcal(\R^d)$ through a Borel map $f : \R^d \to \R^d$, which is uniquely characterised by the identity
\begin{equation*}
\INTDom{\varphi(x)}{\R^d}{(f_{\sharp} \mu)(x)} = \INTDom{\varphi \circ f(x)}{\R^d}{\mu(x)}
\end{equation*}
satisfied for every bounded Borel mapping $\varphi : \R^d \to \R^d$, wherein ``$\circ$'' stands for the composition operation between functions. The set of \textit{transport plans} between two measures $\mu,\nu \in \Pcal(\R^d)$ is then defined by
\begin{equation*}
\Gamma(\mu,\nu) := \Big\{ \gamma \in \Pcal_p(\R^{2d}) ~\, \textnormal{s.t.}~  \pi^1_{\sharp} \gamma = \mu ~\text{and}~ \pi^2_{\sharp} \gamma = \nu \Big\},
\end{equation*}
where $\pi^1,\pi^2 : \R^d \times \R^d \to \R^d$ represent the projections onto the first and second factors respectively. In this context, the \textit{Wasserstein distance of order $p$} between $\mu$ and $\nu$ is the quantity given by
\begin{equation*}
W_p(\mu,\nu) := \min_{\gamma \in \Gamma(\mu,\nu)} \bigg( \INTDom{|x-y|^p}{\R^{2d}}{\gamma(x,y)} \bigg)^{1/p},
\end{equation*}
and we henceforth denote by $\Gamma_o(\mu,\nu)$ the set of $p$-optimal transport plans for which the minimum is reached. Owing to the definition of this distance in the form of an infimum, it can be checked straightforwardly following e.g. \cite[Chapter 7]{AGS} that
\begin{equation}
\label{eq:WassEst}
W_p(\zeta_{\sharp}\mu,\xi_{\sharp} \mu) \leq~ \NormLp{\xi - \zeta}{p}{\R^d,\R^d ; \; \mu}
\end{equation}
for any $\mu \in \Pcal_p(\R^d)$ and each pair of elements $\zeta,\xi \in \Lpazo^p(\R^d,\R^d;\mu)$. In the following proposition, we recall elementary facts regarding the topology of Wasserstein spaces, see e.g. \cite[Proposition 7.1.5]{AGS}.

\begin{prop}[Topological properties of Wasserstein spaces]
\label{prop:Wass}
The metric spaces $(\Pcal_p(\R^d),W_p)$ are complete and separable, and their topology is stronger than the narrow topology. Moreover, a set $\Kpazo \subset \Pcal_p(\R^d)$ is \textnormal{relatively compact} for the $W_p$-metric if and only if 
\begin{equation*}
\sup_{\mu \in \Kpazo} \INTDom{|x|^p}{\{x \in \R^d ~\textnormal{s.t.}\, |x| \geq k \}}{\mu(x)} ~\underset{k \to +\infty}{\longrightarrow}~ 0,
\end{equation*}
that is, if and only if it is $p$-uniformly integrable. 
\end{prop}


\paragraph*{Wasserstein geometry.}

In addition to their convenient topological properties, the Wasserstein spaces can be endowed with a geometric structure that greatly resembles that of a Riemannian manifold when $p=2$. For a general $p \in [1,+\infty)$, it is discussed in depth throughout \cite[Chapter 8]{AGS} that $(\Pcal_p(\R^d),W_p)$ can be seen as a bundle, whose fibers are the closed cones 
\begin{equation*}
\Tan_{\mu} \Pcal_p(\R^d) := \overline{\Big\{ j_q(\nabla \varphi) ~\, \textnormal{s.t.}~ \varphi \in C^{\infty}_c(\R^d,\R) \Big\}}^{\Lpazo^p(\R^d,\R^d ;\,\mu)}
\end{equation*}
defined at each $\mu \in \Pcal_p(\R^d)$. Therein, $q \in (1,+\infty]$ stands for the conjugate exponent of $p \in [1,+\infty)$, and $j_q : \Lpazo^q(\R^d,\R^d;\mu) \to \Lpazo^p(\R^d,\R^d;\mu)$ refers to the so-called \textit{duality map}, which is given by 
\begin{equation}
\label{eq:DualityDef}
j_q(\xi) := \left\{
\begin{aligned}
& 0 ~~ & \text{if $\xi =0$},  \\
& |\xi|^{q-2} \, \xi ~~ & \text{otherwise}
\end{aligned}
\right.
\end{equation}
for each $\xi \in \Lpazo^q(\R^d,\R^d;\mu)$. In the following proposition, we provide an adhoc version of the joint directional superdifferentiability inequalities satisfied by the $p$-Wasserstein distance whenever $p > 1$. We point the reader to \cite[Theorem 10.2.2]{AGS} for the general case.

\begin{prop}[Joint directional superdifferentiability of the Wasserstein distance]
\label{prop:SuperdiffWass}
Fix an element $p \in (1,+\infty)$ and $\mu,\nu \in \Pcal_p(\R^d)$. Then for each $\xi \in \Lpazo^p(\R^d,\R^d;\mu)$ and $\zeta \in \Lpazo^p(\R^d,\R^d;\nu)$, it holds that
\begin{equation*}
\tfrac{1}{p} W_p^p \Big( (\Id + h\zeta)_{\sharp} \mu, (\Id + h\xi)_{\sharp}\nu \Big) - \tfrac{1}{p} W_p^p(\mu,\nu) \leq h \INTDom{\big\langle \zeta(x) - \xi(y), j_p(x-y) \big\rangle}{\R^{2d}}{\gamma(x,y)} + r_p(h,\zeta,\xi)
\end{equation*}
for any $h \in \R$ and every $\gamma \in \Gamma_o(\mu,\nu)$, where the remainder term $r_p(h,\zeta,\xi)$ is given explicitly by
\begin{equation}
\label{eq:Remainder1}
\begin{aligned}
r_p(h,\zeta,\xi) & := (p-1) \bigg( W_p(\mu,\nu) + |h| \Big( \hspace{-0.1cm} \NormLp{\zeta}{p}{\R^d,\R^d; \,\mu} + \NormLp{\xi}{p}{\R^d,\R^d; \,\nu} \hspace{-0.1cm} \Big) \bigg)^{p-2} \\
& \hspace{5.15cm} \times \Big( \hspace{-0.1cm} \NormLp{\zeta}{p}{\R^d,\R^d; \, \mu}^2 + \NormLp{\xi}{p}{\R^d,\R^d; \,\nu}^2 \hspace{-0.1cm} \Big) |h|^2
\end{aligned}
\end{equation}
when $p \in [2,+\infty)$, and by
\begin{equation}
\label{eq:Remainder2}
r_p(h,\zeta,\xi) := \tfrac{2}{p-1} \Big( \hspace{-0.1cm} \NormLp{\zeta}{p}{\R^d,\R^d;\, \mu}^p + \NormLp{\xi}{p}{\R^d,\R^d; \, \nu}^p \hspace{-0.1cm} \Big) |h|^p
\end{equation}
when $p \in (1,2]$. In particular, there exists a constant $C_p > 0$ which only depends on the magnitudes of $p,\Mpazo_p(\mu),\Mpazo_p(\nu),\NormLp{\zeta}{p}{\R^d,\R^d \, ; \mu}$ and $\NormLp{\xi}{p}{\R^d,\R^d \, ; \nu}$, such that 
\begin{equation*}
r_p(h,\zeta,\xi) \leq C_p |h|^{\min\{p,2\}}
\end{equation*}
whenever $h \in (0,1]$. Moreover, note that $r_p(h,\zeta,\xi) = 0$ if $\zeta = \xi = 0$. 
\end{prop}

\begin{proof}
For the sake of readability, the proof of this result is deferred to Appendix \ref{section:AppendixSuperdiff}. 
\end{proof}


\subsection{Set-valued analysis and topological properties of $C^0(\R^d,\R^d)$}

In this second preliminary section, we recollect pivotal concepts of set-valued analysis and discuss some of the topological features of the space $C^0(\R^d,\R^d)$ that will prove useful in the sequel. We point the reader to the reference treatises \cite{Aubin1984,Aubin1990} for the former topics, and to \cite{Kelley1975,Rudin1987} for the latter.


\paragraph*{Elementary notations.}

Given a metric space $(X,\dsf_X(\cdot,\cdot))$, we will denote the closed ball of radius $R>0$ centered at $x \in X$ by $\B_X(x,R) := \big\{ x' \in X ~\, \textnormal{s.t.}~ \dsf_X(x,x') \leq R \big\}$, and write
\begin{equation*}
\dist_X(\Qpazo \, ; \Qpazo') := \inf_{(x,x') \in \Qpazo \times \Qpazo'} \dsf_X(x,x')
\end{equation*}
for the usual distance between two closed sets $\Qpazo,\Qpazo' \subset X$. Throughout the article, we will also work with the \textit{Hausdorff metric}, which is defined by 
\begin{equation}
\label{eq:Hausdorff}
\dsf_{\Hpazo}(\Kpazo,\Kpazo') := \inf \Big\{ \epsilon > 0 ~\, \textnormal{s.t.}~ \Kpazo \subset \B_X(\Kpazo',\epsilon) ~\textnormal{and}~ \Kpazo' \subset \B_X(\Kpazo,\epsilon) \Big\}
\end{equation}
for each pair of compact sets $\Kpazo,\Kpazo' \subset X$, where we used the condensed notation
\begin{equation*}
\B_X \big( \Kpazo,\epsilon \big) := \bigcup_{x \in \Kpazo} \B_X(x,\epsilon). 
\end{equation*}
In our subsequent developments, we will denote by $\textnormal{int}(\Qpazo)$ and $\partial \Qpazo := \Qpazo \setminus \textnormal{int}(\Qpazo)$ the interior and topological boundary of a set $\Qpazo$. In the particular case in which $(X,\dsf_X(\cdot,\cdot))$ possesses a linear structure -- e.g. when it is a Banach or a Fr\'echet space, see for instance \cite{Horvath2012} --, we define the \textit{closed convex hull} of any subset $\Bpazo \subset X$ as 
\begin{equation*}
\co(\Bpazo) = \overline{\textnormal{co}(\Bpazo)}^X := \overline{\bigcup_{N \geq 1} \bigg\{ \mathsmaller{\sum}\limits_{j=1}^N \alpha_j b_j ~\, ~\textnormal{s.t.}~~ b_j \in \Bpazo, ~ \alpha_j \geq 0 ~~ \text{for $j \in \{1,\dots,N\}$} ~~ \text{and} ~~ \mathsmaller{\sum}\limits_{j=1}^N \alpha_j =1  \bigg\}}^X,
\end{equation*}
wherein ``$\overline{\bullet}^X$'' stands for the sequential closure with respect to $\dsf_X(\cdot,\cdot)$. We finally recall the definition of the so-called \textit{proper subsets} of a metric space. 

\begin{Def}[Proper subsets of a metric space]
\label{def:Proper}
A closed set $\Qpazo \subset X$ is \textnormal{proper} provided that $\Qpazo \cap \B_X(x,R) $ is compact for each $x \in X$ and every $R>0$. 
\end{Def}

Note that in the previous definition, one could alternatively require that $\Qpazo \cap \B_X(x,R)$ be compact for each $x \in \Qpazo$ instead of $x \in X$.


\paragraph*{Set-valued analysis.}

We recall that a \textit{set-valued map} -- or \textit{correspondence} -- between two metric spaces $(X,\dsf_X(\cdot,\cdot))$ and $(Y,\dsf_Y(\cdot,\cdot))$ is an application $\Fpazo : X \tto Y$ whose images are subsets of $Y$, namely $\Fpazo(x) \subset Y$ for all $x \in X$. In this context, the \textit{graph} of $\Fpazo$ is the subset of $X \times Y$ defined by
\begin{equation*}
\Graph(\Fpazo) := \Big\{ (x,y) \in X \times Y ~\, \textnormal{s.t.}~ y \in \Fpazo(x) \Big\}.
\end{equation*}
In the coming definitions, we recall the main regularity notions for set-valued mappings with values in metric spaces, starting with those of \textit{continuity} and \textit{Lipschitz regularity}.

\begin{Def}[Continuity of set-valued maps]
\label{def:Continuity}
A correspondence $\Fpazo : X \tto Y$ is said to be \textnormal{continuous} at $x \in X$ if both the following conditions hold. 
\begin{enumerate}
\item[$(i)$] $\Fpazo$ is \textnormal{lower-semicontinuous} at $x$, i.e. for any $y \in \Fpazo(x)$ and $\epsilon > 0$, there exists $\delta > 0$ such that 
\begin{equation*}
\Fpazo(x') \cap \B_Y(y,\epsilon) \neq \emptyset
\end{equation*}
for each $x' \in \B_X(x,\delta)$.
\item[$(ii)$] $\Fpazo$ is \textnormal{upper-semicontinuous} at $x$, i.e. for any $\epsilon > 0$, there exists $\delta >0$ such that 
\begin{equation*}
\Fpazo(x') \subset \B_Y \big( \Fpazo(x),\epsilon \big)
\end{equation*}
for each $x' \in \B_X(x,\delta)$.
\end{enumerate} 
\end{Def}

\begin{Def}[Lipschitz continuity of set-valued maps]
A correspondence $\Fpazo : X \tto Y$ is said to be \textnormal{Lipschitz continuous} with constant $L >0$ provided that 
\begin{equation*}
\Fpazo(x') \subset \B_Y \Big( \Fpazo(x) , L \, \dsf_X(x,x') \Big)
\end{equation*}
for all $x,x' \in X$. When $\Fpazo : X \tto Y$ has compact images, this is equivalent to requiring that
\begin{equation*}
\dsf_{\Hpazo}(\Fpazo(x),\Fpazo(x')) \leq L \, \dsf_X(x,x')
\end{equation*}
for all $x,x' \in X$. 
\end{Def}

As for functions defined over the real line taking values in a metric space, it is possible to formulate several relevant notions of \textit{absolute continuity} for set-valued maps. Some of them involve the quantity
\begin{equation*}
\Delta_{x,R}(\Qpazo \, ; \Qpazo') := \inf \Big\{ \epsilon > 0 ~\, \textnormal{s.t.}~ \Qpazo \cap \B_X(x,R) \subset \B_X(\Qpazo',\epsilon) \Big\} \in \R_+ \cup \{+\infty\}, 
\end{equation*}
defined for each $x \in X$, $R > 0$ and every pair of nonempty closed sets $\Qpazo,\Qpazo' \subset X$, which can be seen as a sort of asymmetric and localised version of the Hausdorff distance recalled in \eqref{eq:Hausdorff}. 

\begin{Def}[Notions of absolute continuity for set-valued mappings]
\label{def:AC}
We say that a correspondence $\Fpazo : I \tto X$ with nonempty closed images is \textnormal{absolutely continuous} if for every $x \in X$ and each $R >0$, there exists a map $m_{x,R}(\cdot) \in L^1(I,\R_+)$ such that
\begin{equation*}
\max \Big\{ \Delta_{x,R} \big( \Fpazo(\tau) \, ; \Fpazo(t) \big) \, , \, \Delta_{x,R} \big( \Fpazo(t) \, ; \Fpazo(\tau) \big) \Big\} \leq \INTSeg{m_{x,R}(s)}{s}{\tau}{t}
\end{equation*}
for all times $\tau,t \in I$ satisfying $\tau \leq t$. Analogously, we say that $\Fpazo : I \tto X$ is \textnormal{left absolutely continuous} if only the one-sided inequality
\begin{equation*}
\Delta_{x,R} \big( \Fpazo(\tau) \, ; \Fpazo(t) \big) \leq \INTSeg{m_{x,R}(s)}{s}{\tau}{t}
\end{equation*} 
holds. In the case where $\Fpazo : I \tto X$ has compact images, we say that it is \textnormal{absolutely continuous in the Hausdorff metric} if there exists a map $m_{\Fpazo}(\cdot) \in L^1(I,\R_+)$ such that
\begin{equation*}
\dsf_{\Hpazo}(\Fpazo(\tau),\Fpazo(t)) \leq \INTSeg{m_{\Fpazo}(s)}{s}{\tau}{t}
\end{equation*}
for all times $\tau,t \in I$ such that $\tau \leq t$. 
\end{Def}

These notions of absolute continuity allow us to establish the following regularity statements on the distance between set-valued maps, which will prove crucial in some of our subsequent developments.

\begin{prop}[Regularity of the distance between set-valued maps]
\label{prop:AC}
Let $\Kpazo : I \tto X$ be a set-valued map with nonempty compact images that is absolutely continuous in the Hausdorff metric and $\Qpazo : I \tto X$ be an absolutely continuous set-valued map with nonempty closed images. Then, the map
\begin{equation*}
g : t \in I \mapsto \dist_X(\Kpazo(t) \, ; \Qpazo(t)) \in \R_+
\end{equation*}
is absolutely continuous. In the case where $\Qpazo : [0,T] \tto X$ is only left absolutely continuous, then the set-valued mapping
\begin{equation*}
\Ecal : t \in I \tto \Big\{ \alpha \in \R_+ ~\, \textnormal{s.t.}~ \alpha = g(t) + r ~ \text{for some $r \geq 0$} \Big\}
\end{equation*}
is left absolutely continuous as well.
\end{prop}

\begin{proof}
Being somewhat long and technical, the proof of this statement is deferred to Appendix \ref{section:AppendixAC}. 
\end{proof}

Analogously to the classical notions of regularity exposed hereinabove, it is possible to generalise the concept of \textit{measurability} to set-valued mappings, as highlighted by the following definition.

\begin{Def}[Measurability of set-valued maps]
A set-valued map $\Fpazo : I \tto X$ is said to be \textnormal{$\Lcal^1$-measurable} if the preimages 
\begin{equation*}
\Fpazo^{-1}(\Opazo) := \Big\{ t \in I ~\, \textnormal{s.t.}~ \Fpazo(t) \cap \Opazo \neq \emptyset \Big\}
\end{equation*}
are $\Lcal^1$-measurable for each open set $\Opazo \subset X$. Moreover, we say that an $\Lcal^1$-measurable map $f : I \to X$ is a \textnormal{measurable selection} of $\Fpazo : I \tto X$ if $f(t) \in \Fpazo(t)$ for $\Lcal^1$-almost every $t \in I$. 
\end{Def}

In the following theorem, we recall an instrumental result of set-valued analysis excerpted from \cite[Theorem 8.1.3]{Aubin1990}, which asserts that measurable correspondences with nonempty closed images always admit measurable selections. 

\begin{thm}[Existence of measurable selections]
\label{thm:MeasurableSel}
Suppose that $(X,\dsf_X(\cdot,\cdot))$ is a complete separable metric space. Then every  $\Lcal^1$-measurable set-valued map $\Fpazo : I \tto X$ with nonempty closed images admits a measurable selection.
\end{thm}

We end this primer in set-valued analysis by recollecting a fine adaptation of the Scorza-Dragoni theorem for set-valued mappings between metric spaces, for which we refer to \cite[Theorem 1]{Bonnano1989}. 

\begin{thm}[Scorza-Dragoni property for set-valued mappings]
\label{thm:ScorzaMulti}
Suppose that $(X,\dsf_X(\cdot,\cdot))$ and $(Y,\dsf_Y(\cdot,\cdot))$ are complete separable metric spaces, and let $\Fpazo : I \times X \tto Y$ be a set-valued map with nonempty closed images, such that $t \in I \tto \Fpazo(t,x)$ is $\Lcal^1$-measurable for all $x \in X$ and $x \tto \Fpazo(t,x)$ is continuous for $\Lcal^1$-almost every $t \in I$. 

Then for every $\epsilon > 0$, there exists a compact set $\Acal_{\epsilon} \subset I$ such that $\Lcal^1(I \setminus \Acal_{\epsilon}) < \epsilon$, and for which the following holds. 
\begin{enumerate}
\item[$(i)$] The restricted set-valued mapping $\Fpazo : \Acal_{\epsilon} \times X \tto Y$ is lower-semicontinuous. 
\item[$(ii)$] The graph of the restricted set-valued mapping
\begin{equation*}
\Graph(\Fpazo)_{|\Acal_{\epsilon} \times X \times Y} := \Big\{ (t,x,y) \in \Acal_{\epsilon} \times X \times Y ~\, \textnormal{s.t.}~ y \in \Fpazo(t,x) \Big\}
\end{equation*}
is closed in $\Acal_{\epsilon} \times X \times Y$. 
\end{enumerate}
\end{thm}


\paragraph*{Topological structures over the space of continuous functions.}

Throughout the coming paragraphs, we recall some useful topological properties of the space $C^0(\R^d,\R^d)$. In what follows, we let
\begin{equation*}
\dsf_{\sup}(v,w) := \sup_{x \in \R^d} |v(x) - w(x)| \in \R_+ \cup \{+\infty\}
\end{equation*}
be the supremum extended-distance between a pair of elements $v,w \in C^0(\R^d,\R^d)$. While this latter is useful to control the global discrepancy between two continuous functions -- which may be equal to $+\infty$ --, the topology that it induces is not separable and thus ill-adapted to the application of measurable selection theorems. For this reason, we will systematically endow the space $C^0(\R^d,\R^d)$ with the topology of \textit{local uniform convergence}, whose definition is recalled hereinbelow. 

\begin{Def}[Topology of local uniform convergence]
\label{def:dcc}
A sequence of maps $(v_n) \subset C^0(\R^d,\R^d)$ converges \textnormal{locally uniformly} -- or \textnormal{uniformly on compact sets} -- to some $v \in C^0(\R^d,\R^d)$ provided that
\begin{equation*}
\NormC{v - v_n}{0}{K,\R^d} ~\underset{n \to +\infty}{\longrightarrow}~ 0
\end{equation*}
for each compact set $K \subset \R^d$. This notion of convergence endows $C^0(\R^d,\R^d)$ with the structure of a separable Fr\'echet space, whose topology is induced by the translation invariant metric
\begin{equation}
\label{eq:dcc}
\dsf_{cc}(v,w) := \sum_{k=1}^{+\infty} 2^{-k} \min \Big\{ 1 \, , \, \NormC{v-w}{0}{B(0,k),\R^d} \Big\}
\end{equation}
that is well-defined and finite for any $v,w \in C^0(\R^d,\R^d)$.
\end{Def}

Amongst its interesting properties, the topology of local uniform convergence enjoys a very explicit and amenable characterisation of compactness, for which we refer e.g. to \cite[Chapter 7, Theorem 18]{Kelley1975}. 

\begin{thm}[Ascoli-Arzel\`a compactness criterion]
\label{thm:Ascoli}
A closed set $V \subset C^0(\R^d,\R^d)$ is compact for the topology induced by $\dsf_{cc}(\cdot,\cdot)$ \textnormal{if and only if} its elements are locally uniformly equicontinuous and if, for every $x \in \R^d$, there exists a compact set $K_x \subset \R^d$ such that $v(x) \in K_x$ for each $v \in V$.
\end{thm}

We recall below a fact whose proof can be found in \cite{Papageorgiou1986}, which establishes a one-to-one correspondence between $\Lcal^1$-measurable maps $t \in I \mapsto v(t) \in C^0(\R^d,\R^d)$ and Carathéodory vector fields. We recall that a map $v : I \times \R^d \to \R^d$ is Carathéodory if $t \in I \mapsto v(t,x)$ is $\Lcal^1$-measurable for all $x \in \R^d$ and $x\in \R^d \mapsto v(t,x)$ is continuous for $\Lcal^1$-almost every $t \in [0,T]$. 

\begin{lem}[Measurable selections in $C^0(\R^d,\R^d)$ and Carathéodory vector fields]
\label{lem:Carathéodory}
A vector field $(t,x) \in I \times \R^d \mapsto v(t,x) \in \R^d$ is Carathéodory \textnormal{if and only if} its functional lift $t \in I \mapsto v(t) \in C^0(\R^d,\R^d)$ is $\Lcal^1$-measurable with respect to the topology induced by $\dsf_{cc}(\cdot,\cdot)$. 
\end{lem}

Lastly, we prove a technical result which states that for sequences of sublinear continuous functions, the convergence with respect to $\dsf_{cc}(\cdot,\cdot)$ implies the convergence in $\Lpazo^p(\R^d,\R^d;\mu)$ for every $p \in [1,+\infty)$ and each $\mu \in \Pcal_p(\R^d)$. 

\begin{lem}[Link between local uniform and Lebesgue convergences]
\label{lem:ConvdccWass}
Let $(v_n) \subset C^0(\R^d,\R^d)$ be a sequence of maps such that for some $m>0$, there holds
\begin{equation*}
|v_n(x)| \leq m(1+|x|)
\end{equation*}
for all $x \in \R^d$ and each $n \geq 1$. Moreover, suppose that $\dsf_{cc}(v_n,v) \to 0$ as $n \to +\infty$ for some $v \in C^0(\R^d,\R^d)$. Then, for each $p \in [1,+\infty)$ and $\mu \in \Pcal_p(\R^d)$, it holds that $\{v_n\}_{n=1}^{+\infty} \subset \Lpazo^p(\R^d,\R^d;\mu)$ and also
\begin{equation*}
\NormLp{v - v_n}{p}{\R^d,\R^d ; \, \mu} ~\underset{n \to +\infty}{\longrightarrow}~ 0. 
\end{equation*}
\end{lem}

\begin{proof}
The fact that $v_n \in \Lpazo^p(\R^d,\R^d;\mu)$ for each $n \geq 1$ simply follows from the observation that
\begin{equation*}
\NormLp{v_n}{p}{\R^d,\R^d ; \; \mu} \leq m \bigg( \INTDom{(1+|x|)^p}{\R^d}{\mu(x)} \bigg)^{1/p} < +\infty,
\end{equation*}
and likewise $v \in \Lpazo^p(\R^d,\R^d;\mu)$. Fix now an arbitrary $\epsilon > 0$ and consider some $R_{\epsilon} > 0$ for which
\begin{equation*}
\bigg( \INTDom{(1+|x|)^p}{\{ x \in \R^d ~\textnormal{s.t.} \, |x| \geq R_{\epsilon} \}}{\mu(x)} \bigg)^{1/p} \leq \frac{\epsilon}{4(1+m)}. 
\end{equation*}
Then, choose an integer $N_{\epsilon} \geq 1$ such that 
\begin{equation*}
\NormC{v - v_n}{0}{B(0,R_{\epsilon}),\R^d} \leq \frac{\epsilon}{2}
\end{equation*}
for each $n \geq N_{\epsilon}$, which is always possible by the definition \eqref{eq:dcc} of $\dsf_{cc}(\cdot,\cdot)$. Whence, it follows that
\begin{equation*}
\begin{aligned}
\NormLp{v - v_n}{p}{\R^d,\R^d ; \; \mu} & \leq \NormC{v - v_n}{0}{B(0,R_{\epsilon}),\R^d} + 2m \bigg( \INTDom{(1+|x|)^p}{\{ x \in \R^d ~\textnormal{s.t.} \, |x| \geq R_{\epsilon} \}}{\mu(x)} \bigg)^{1/p} \leq \epsilon
\end{aligned}
\end{equation*}
for each $n \geq N_{\epsilon}$, which yields the desired convergence result. 
\end{proof}


\subsection{Continuity equations and inclusions in Wasserstein spaces}

In this last preliminary section, we expose well-posedness results and estimates for solutions of continuity equations and inclusions in Wasserstein spaces. These latter are mostly borrowed from our previous works \cite{ContInc,ContIncPp}, but we also point the reader to \cite[Chapter 8]{AGS} and \cite{AmbrosioC2014,Pedestrian} for more standard versions thereof.

\paragraph*{Continuity equations in the Carathéodory framework.}

In the ensuing paragraphs, we recall some elementary results pertaining to the qualitative properties of \textit{continuity equations} of the form 
\begin{equation*}
\partial_t \mu(t) + \Div_x(v(t)\mu(t)) = 0,
\end{equation*}
defined over some time interval $[0,T]$ with $T > 0$, whose solutions are understood in the sense of distributions, namely, satisfying
\begin{equation*}
\INTSeg{\INTDom{\Big( \partial_t \varphi(t,x) + \big\langle \nabla_x \varphi(t,x) , v(t,x) \big\rangle \Big)}{\R^d}{\mu(t)(x)}}{t}{0}{T} = 0
\end{equation*}
against smooth test functions $\varphi \in C^{\infty}_c((0,T) \times \R^d,\R)$. In this context, given a real number $p \in [1,+\infty)$ and a pair of elements $(\tau,\mu_{\tau}) \in [0,T] \times \Pcal_p(\R^d)$, we will study the well-posedness of the Cauchy problem
\begin{equation}
\label{eq:ContEqCauchy}
\left\{
\begin{aligned}
& \partial_t \mu(t) + \Div_x(v(t)\mu(t)) = 0, \\
& \mu(\tau) = \mu_{\tau},
\end{aligned}
\right.
\end{equation}
in the case where the velocity field $v : [0,T] \times \R^d \to \R^d$ satisfies either the following standard Carathéodory assumptions, or some of their variants.

\begin{taggedhyp}{\textbn{(CE)}} \hfill
\label{hyp:CE}
\begin{enumerate}
\item[$(i)$] The velocity field $v : [0,T] \times \R^d \to \R^d$ is Carathéodory, i.e. $t \in [0,T] \mapsto v(t,x)$ is $\Lcal^1$-measurable for all $x \in \R^d$ while $x \in \R^d \mapsto v(t,x)$ is continuous for $\Lcal^1$-almost every $t \in [0,T]$. Moreover, there exists a map $m(\cdot) \in L^1([0,T],\R_+)$ such that 
\begin{equation*}
|v(t,x)| \leq m(t)(1+|x|)
\end{equation*}
for $\Lcal^1$-almost every $t \in [0,T]$ and all $x \in \R^d$.
\item[$(ii)$] There exists a map $l(\cdot) \in L^1([0,T],\R_+)$ such that 
\begin{equation*}
\Lip(v(t) \, ; \R^d) \leq l(t)
\end{equation*}
for $\Lcal^1$-almost every $t \in [0,T]$. 
\end{enumerate}
\end{taggedhyp}

In their strongest form, the well-posedness results stated in Theorem \ref{thm:CE} below for continuity equations involve the notion of \textit{characteristic flows} generated by a velocity field.

\begin{Def}[Characteristic flow]
\label{def:Flows}
Given a velocity field $v : [0,T] \times \R^d \to \R^d$ satisfying Hypotheses \ref{hyp:CE}, we define the \textnormal{characteristic flows} $(\Phi_{(\tau,t)}^v)_{\tau,t \in [0,T]} \subset C^0(\R^d,\R^d)$ as the unique collection of maps satisfying
\begin{equation}
\label{eq:FlowDef}
\Phi_{(\tau,t)}^v(x) = x + \INTSeg{v \big( s , \Phi_{(\tau,s)}^v(x) \big)}{s}{\tau}{t}
\end{equation}
for all times $\tau,t \in [0,T]$ and any $x \in \R^d$.
\end{Def}

\begin{thm}[Well-posedness in the Carathéodory framework]
\label{thm:CE}
Let $v : [0,T] \times \R^d \to \R^d$ be a velocity field satisfying Hypothesis \ref{hyp:CE}-$(i)$, and fix some $(\tau,\mu_{\tau}) \in [0,T] \times \R^d$.

Then, the Cauchy problem \eqref{eq:ContEqCauchy} admits solutions $\mu(\cdot) \in \AC([\tau,T],\Pcal_p(\R^d))$. In the case where Hypothesis \ref{hyp:CE}-$(ii)$ holds as well, the latter is then uniquely defined on the whole interval $[0,T]$, and represented explicitly by the formula
\begin{equation*}
\mu(t) = \Phi^{v}_{(\tau,t) \, \sharp} \mu_{\tau}
\end{equation*}
for all times $t \in [0,T]$. 
\end{thm}


\paragraph*{Set-valued dynamics in Wasserstein spaces.} In the next few  paragraphs, we recollect for the sake of completeness the definition of \textit{continuity inclusions} introduced in our earlier works \cite{ContInc,ContIncPp}, along with several estimates on which our main contributions will strongly rely. In what follows, given some $p \in [1,+\infty)$, we focus on the set-valued Cauchy problems of the form
\begin{equation}
\label{eq:ContIncCauchy}
\left\{
\begin{aligned}
& \partial_t \mu(t) \in - \Div_x \Big( V(t,\mu(t)) \mu(t) \Big), \\
& \mu(\tau) = \mu_{\tau},  
\end{aligned}
\right.
\end{equation}
wherein $(\tau,\mu_{\tau}) \in [0,T] \times \Pcal_p(\R^d)$ and $V : [0,T] \times \Pcal_p(\R^d) \tto C^0(\R^d,\R^d)$ are given, and whose solutions are understood in the following sense.

\begin{Def}[Solutions to continuity inclusions]
\label{def:CI}
A curve of measures $\mu(\cdot) \in \AC([0,T],\Pcal_p(\R^d))$ is said to be a solution of the Cauchy problem \eqref{eq:ContIncCauchy} if there exists an $\Lcal^1$-measurable selection $t \in [0,T] \mapsto v(t) \in V(t,\mu(t))$ such that the \textnormal{trajectory-selection pair} $(\mu(\cdot),v(\cdot))$ satisfies
\begin{equation*}
\left\{
\begin{aligned}
& \partial_t \mu(t) + \Div_x(v(t)\mu(t)) = 0, \\
& \mu(\tau) = \mu_{\tau},
\end{aligned}
\right.
\end{equation*}
in the sense of distributions. 
\end{Def}

Based on our earlier contributions, we will assume throughout this article that the dynamics satisfies the following assumptions. Therein and in what follows, $C^0(\R^d,\R^d)$ is systematically endowed with the separable Fr\'echet structure induced by $\dsf_{cc}(\cdot,\cdot)$, exposed in Definition \ref{def:dcc}. 

\begin{taggedhyp}{\textbn{(CI)}} 
\label{hyp:CI} \hfill
\begin{enumerate}
\item[$(i)$] For any $\mu \in \Pcal_p(\R^d)$, the set-valued map $t \in [0,T] \tto V(t,\mu) \subset C^0(\R^d,\R^d)$ is $\Lcal^1$-measurable with closed nonempty images.
\item[$(ii)$] There exists a map $m(\cdot) \in L^1([0,T],\R_+)$ such that for $\Lcal^1$-almost every $t \in [0,T]$, any $\mu \in \Pcal_p(\R^d)$, every $v \in V(t,\mu)$ and all $x \in \R^d$, it holds 
\begin{equation*}
|v(x)| \leq m(t) \Big(1 + |x| + \Mpazo_p(\mu) \Big).
\end{equation*}
\item[$(iii)$] There exists a map $l(\cdot) \in L^1([0,T],\R_+)$ such that for $\Lcal^1$-almost every $t \in [0,T]$, any $\mu \in \Pcal_p(\R^d)$ and every $v \in V(t,\mu)$, it holds 
\begin{equation*}
\Lip(v \, ; \R^d) \leq l(t). 
\end{equation*} 
\item[$(iv)$] There exists a map $L(\cdot) \in L^1([0,T],\R_+)$ such that for $\Lcal^1$-almost every $t \in [0,T]$, any $\mu,\nu \in \Pcal_p(\R^d)$ and each $v \in V(t,\mu)$, there exists an element $w \in V(t,\nu)$ for which
\begin{equation*}
\dsf_{\sup}(v,w) \leq L(t) W_p(\mu,\nu).
\end{equation*}
\end{enumerate}
\end{taggedhyp}

Examples of classical set-valued mappings defined in terms of control systems satisfying localised variants of \ref{hyp:CI} are provided in \cite[Section 4]{ContInc}. In our subsequent developments, we will frequently refer to solutions of \eqref{eq:ContIncCauchy} by using the terminology of \textit{reachable} and \textit{solutions sets}, defined as follows.

\begin{Def}[Reachable and solution sets of continuity inclusions]
Given a pair of elements $(\tau,\mu_{\tau}) \in [0,T] \times \Pcal_p(\R^d)$, we define the (forward) \textnormal{solution set} of the Cauchy problem \eqref{eq:ContIncCauchy} as 
\begin{equation*}
\Spazo_{[\tau,T]}(\tau,\mu_{\tau}) := \bigg\{ \mu(\cdot) \in \AC([\tau,T],\Pcal_p(\R^d)) ~\, \textnormal{s.t.}~ \mu(\cdot) ~\text{is a solution of \eqref{eq:ContIncCauchy}} \bigg\},
\end{equation*}
and denote by
\begin{equation*}
\Rpazo_{(\tau,t)}(\mu_{\tau}) := \Big\{ \mu(t) ~\,\textnormal{s.t.}~ \mu(\cdot) \in \Spazo_{[\tau,T]}(\tau,\mu_{\tau}) \Big\}
\end{equation*}
the corresponding \textnormal{reachable sets} at time $t \in [\tau,T]$.
\end{Def}

By combining classical concatenation results for solutions of continuity equations (see e.g. \cite[Lemma 4.4]{Dolbeault2009}) and Definition \ref{def:CI}, it can be shown that the reachable sets satisfy the semigroup property
\begin{equation}
\label{eq:SemigroupReach}
\Rpazo_{(\tau,t)}(\mu_{\tau}) = \Rpazo_{(s,t)} \circ \Rpazo_{(\tau,s)}(\mu_{\tau}) 
\end{equation}
for all times $\tau \leq s \leq t \leq T$. Besides, it follows from Hypotheses \ref{hyp:CI} and Theorem \ref{thm:CE} that solution curves are also well-defined and unique backward in time. Hence, each element of $\Spazo_{[\tau,T]}(\tau,\mu_{\tau})$ can be seen as a restriction to $[\tau,T]$ of some curve in $\Spazo_{[0,T]}(\tau,\mu_{\tau})$  

In the next propositions, we recall several a priori estimates for solutions of \eqref{eq:ContIncCauchy}, along with some useful topological properties for the reachable and solution sets. Therein and in what follows, given a map $m(\cdot) \in L^1([0,T],\R_+)$, we will frequently use the shorthand notation $\Norm{m(\cdot)}_1 := \NormL{m(\cdot)}{1}{[0,T],\R}$.

\begin{prop}[Moment, equi-integrability and absolute continuity estimates] 
\label{prop:MomentumCI}
Let $V : [0,T] \times \Pcal_p(\R^d) \tto C^0(\R^d,\R^d)$ be a set-valued map satisfying Hypotheses \ref{hyp:CI} and $(\tau,\mu_{\tau}) \in [0,T] \times \Pcal_p(\R^d)$. 

Then, there exists a constant $\Cpazo_T > 0$ which only depends on the magnitudes of $p,\Mpazo_p(\mu_{\tau})$ and $\Norm{m(\cdot)}_1$ such that every curve $\mu(\cdot) \in \Spazo_{[0,T]}(\tau,\mu_{\tau})$ complies with the a priori moment bound
\begin{equation}
\label{eq:MomentumCI}
\Mpazo_p(\mu(t)) \leq \Cpazo_T, 
\end{equation}
as well as the uniform equi-integrability estimate
\begin{equation}
\label{eq:UnifIntegCI}
\INTDom{|x|^p}{\{x \in \R^d ~\textnormal{s.t.} \; |x| \geq R \}}{\mu(t)(x)} \leq \Cpazo_T^p \INTDom{(1+|x|)^p}{\{x \in \R^d ~\textnormal{s.t.} \; |x| \geq R/\Cpazo_T -1 \}}{\mu_{\tau}(x)}
\end{equation}
for all times $t \in [0,T]$ and each $R>0$. Moreover, the following uniform absolute continuity inequality 
\begin{equation}
\label{eq:ACCI}
W_p(\mu(t_1),\mu(t_2)) \leq (1+\Cpazo_T) \INTSeg{m(s)}{s}{t_1}{t_2}
\end{equation}
holds for all times $0 \leq t_1 \leq t_2 \leq T$ and every $\mu(\cdot) \in \Spazo_{[0,T]}(\tau,\mu_{\tau})$. 
\end{prop}

\begin{prop}[Topological properties of the reachable and solution sets]
\label{prop:Topological}
Assume that the hypotheses of Proposition \ref{prop:MomentumCI} hold and that $V : [0,T] \times \Pcal_p(\R^d) \tto C^0(\R^d,\R^d)$ has convex images. Then the reachable sets $\Rpazo_{(\tau,t)}(\mu_{\tau}) \subset \Pcal_p(\R^d)$ are compact for all times $t \in [0,T]$, and the solution set $\Spazo_{[0,T]}(\tau,\mu_{\tau}) \subset C^0([0,T],\Pcal_p(\R^d))$ is compact for the topology of uniform convergence.
\end{prop}

\begin{proof}
The fact that the solution set $\Spazo_{[0,T]}(\tau,\mu_{\tau}) \subset C^0([\tau,T],\Pcal_p(\R^d))$ is compact when $V : [0,T] \times \Pcal_p(\R^d) \tto C^0(\R^d,\R^d)$ has convex images was proven in \cite[Theorem 3.5]{ContIncPp}. It is then straightforward to show that the underlying reachable sets are compact for all times $t \in [0,T]$.
\end{proof}

We end this preliminary section by recalling a simplified and condensed version of one of the main results of \cite{ContIncPp}, which combines an existence result for \eqref{eq:ContIncCauchy} together with a powerful estimate ``à la Gr\"onwall'' involving the distance to an a priori given curve of measures. 

\begin{thm}[Local Filippov estimates for continuity inclusions]
\label{thm:LocalFilippov}
Let $V : [0,T] \times\Pcal_p(\R^d) \tto C^0(\R^d,\R^d)$ be a set-valued map satisfying Hypotheses \ref{hyp:CI} and $\nu(\cdot) \in \AC([0,T],\Pcal_p(\R^d))$ be a solution of the continuity equation 
\begin{equation*}
\partial_t \nu(t) + \Div_y (w(t)\nu(t)) = 0
\end{equation*}
driven by a Carathéodory vector field $w : [0,T] \times\R^d \to \R^d$ satisfying the sublinearity estimate
\begin{equation*}
|w(t,y)| \leq m(t)(1+|y|)
\end{equation*}
for $\Lcal^1$-almost every $t \in [0,T]$ and all $y \in \R^d$. Given $R > 0$, denote by $\eta_R(\cdot) \in L^1([0,T],\R_+)$ the \textnormal{local mismatch function}, defined by 
\begin{equation*}
\eta_R(t) := \dist_{C^0(B(0,R),\R^d )} \Big( w(t) \, ; V(t,\nu(t)) \Big)
\end{equation*}
for $\Lcal^1$-almost every $t \in [0,T]$.

Then for every $(\tau,\mu_{\tau}) \in [0,T] \times \Pcal_p(\R^d)$ and each $R>0$, there exists a curve of measures $\mu(\cdot) \in \Spazo_{[0,T]}(\tau,\mu_{\tau})$ which satisfies the a priori estimate
\begin{equation}
\label{eq:LocalFilippov}
W_p(\mu(t),\nu(t)) \leq \Cpazo_T' \bigg( W_p(\mu_{\tau},\nu(\tau)) + \INTSeg{\eta_R(s)}{s}{\tau}{t} + \Epazo_{\nu}(\tau,t,R) \bigg)
\end{equation}
for all times $t \in [\tau,T]$. Therein, the constant $\Cpazo_T' > 0$ only depends on the magnitudes of the data $p,\Mpazo_p(\mu_{\tau}),\Norm{m(\cdot)}_1,\Norm{l(\cdot)}_1$ and $\Norm{L(\cdot)}_1$, while the error term $\Epazo_{\nu}(\tau,t,R)$ is given explicitly by
\begin{equation*}
\Epazo_{\nu}(\tau,t,R) := 2 \NormL{m(\cdot)}{1}{[\tau,t]} (1+\Cpazo_T) \bigg( \INTDom{(1+|y|)^p}{\{y \;\textnormal{s.t.}\; |y| \geq R/\Cpazo_T-1  \}}{\nu(\tau)(y)} \bigg)^{1/p}
\end{equation*}
for all times $t \in [\tau,T]$, where $\Cpazo_T > 0$ is the constant appearing in Proposition \ref{prop:MomentumCI}. 
\end{thm}

\begin{rmk}[Link between continuity equations and inclusions]
\label{rmk:LinkIncEq}
In the particular case in which $V : [0,T] \times \Pcal_p(\R^d) \tto C^0(\R^d,\R^d)$ happens to be single valued and independent of $\mu$ -- that is if $V(t,\mu) = \{ v(t) \}$ for $\Lcal^1$-almost every $t \in [0,T]$ and each $\mu \in \Pcal_p(\R^d)$ --, then the corresponding velocity field satisfies Hypotheses \ref{hyp:CE}. In addition, the solution of \eqref{eq:ContIncCauchy} is then unique, coincides with that of \eqref{eq:ContEqCauchy}, and complies with the a priori estimates of Proposition \ref{prop:MomentumCI} and Theorem \ref{thm:LocalFilippov}.
\end{rmk}


\section{Infinitesimal behaviour of the reachable sets}
\label{section:Quali}

\setcounter{equation}{0} \renewcommand{\theequation}{\thesection.\arabic{equation}}

In this section, we prove two fundamental results concerning the metric differentiability properties of solutions of the Cauchy problem
\begin{equation}
\label{eq:ContIncCauchyBis}
\left\{
\begin{aligned}
& \partial_t \mu(t) \in - \Div_x \Big( V(t,\mu(t)) \mu(t) \Big), \\
& \mu(\tau) = \mu_{\tau},  
\end{aligned}
\right.
\end{equation}
which are largely inspired by the analysis carried out in \cite[Section 2]{Frankowska1995}. The first one, discussed in the ensuing theorem, deals with the existence of curves with (approximately) prescribed initial velocities.

\begin{thm}[Existence of admissible curves with approximate initial velocities]
\label{thm:Liminf}
Let $V : [0,T] \times \Pcal_p(\R^d) \tto C^0(\R^d,\R^d)$ be a set-valued map satisfying Hypotheses \ref{hyp:CI}. 

Then, there exists a subset $\Tcal \subset (0,T)$ of full $\Lcal^1$-measure such that for every $\tau \in \Tcal$, all $\mu_{\tau} \in \Pcal_p(\R^d)$, each $v_{\tau} \in V(\tau,\mu_{\tau})$ and any $\epsilon > 0$, there exist some $h_{\epsilon} > 0$ along with a curve $\mu_{\epsilon}(\cdot) \in \Spazo_{[0,T]}(\tau,\mu_{\tau})$ such that
\begin{equation}
\label{eq:LiminfIneq}
W_p \Big( \mu_{\epsilon}(\tau+h) , (\Id + h v_{\tau})_{\sharp} \mu_{\tau} \Big) \leq \epsilon h,
\end{equation}
for all $h \in [0,h_{\epsilon}]$. 
\end{thm}

\begin{proof}
To begin with, denote by $\Tcal_m,\Tcal_l,\Tcal_L \subset (0,T)$ the sets of one-sided Lebesgue points of the maps $m(\cdot),l(\cdot)$, and $L(\cdot)$, respectively, and by $\Tcal_{\textnormal{H}} \subset (0,T)$ the subset of full $\Lcal^1$-measure over which Hypotheses \ref{hyp:CI}-$(ii)$, $(iii)$ and $(iv)$ hold. By Theorem \ref{thm:ScorzaMulti}, there exists for every $k \geq 1$ a compact set $\Apazo^k \subset [0,T]$ satisfying $\Lcal^1([0,T] \setminus \Apazo^k) < \tfrac{1}{2^k}$, and such that $V : \Apazo^k \times \Pcal_p(\R^d) \tto C^0(\R^d,\R^d)$ is lower-semicontinuous in the sense of Definition \ref{def:Continuity}-$(i)$. For each $n \geq 1$, define then $\Acal_n \subset [0,T]$ as
\begin{equation}
\label{eq:AcalDef}
\Acal_n := \bigcap_{k > n} \Apazo^k, 
\end{equation}
and denote by $\tilde{\Acal}_n \subset \Acal_n$ the subset of its one-sided density points understood in the sense of Definition \ref{def:LebesgueDensity}, which can be characterised as the subset of full $\Lcal^1$-measure in $\Acal_n$ such that 
\begin{equation}
\label{eq:DensityPoint}
\lim_{h \to 0^+} \frac{\Lcal^1([\tau,\tau+h] \setminus \Acal_n)}{h} = 0
\end{equation}
for each $\tau \in \tilde{\Acal}_n$. Upon noting that
 for each $n \geq 1$, one has 
\begin{equation*}
\begin{aligned}
\Lcal^1([0,T] \setminus \tilde{\Acal}_n) & = \Lcal^1 \bigg( \bigcup\nolimits_{k > n} \Big( [0,T] \setminus \Apazo^k \Big) \bigg) \\
& \leq \sum_{k=n+1}^{+\infty} \Lcal^1([0,T] \setminus \Apazo^k) \leq \frac{1}{2^n}
\end{aligned}
\end{equation*}
while observing that the sequence of measurable sets $(\tilde{\Acal}_n)$ is increasing by construction, it holds that
\begin{equation*}
\begin{aligned}
\Lcal^1 \bigg( [0,T] \setminus \Big( \bigcup\nolimits_{n \geq 1} \tilde{\Acal}_n \Big) \bigg) =  \lim_{n \to +\infty} \Lcal^1([0,T] \setminus \tilde{\Acal}_n) = 0. 
\end{aligned}
\end{equation*}
Therefore, the set $\Tcal \subset (0,T)$ defined by 
\begin{equation}
\label{eq:TcalDef}
\Tcal := \bigg( \bigcup\nolimits_{n \geq 1} \tilde{\Acal}_n \bigg) \cap \Tcal_m \cap \Tcal_l \cap \Tcal_L \cap \Tcal_{\textnormal{H}},
\end{equation}
has full $\Lcal^1$-measure in $[0,T]$, and in the sequel we fix an element $\tau \in \Tcal$. We also pick a radius $R_{\epsilon} > 0$ satisfying
\begin{equation}
\label{eq:RepsDef}
2 m(\tau)(1+\Cpazo_T) \bigg( \INTDom{\big( 1+|x|^p \big)}{\{ x \in \R^d ~ \textnormal{s.t.} \; |x| \geq R_{\epsilon}/\Cpazo_T -1\}}{\mu_{\tau}(x)} \bigg)^{1/p} \leq \epsilon,
\end{equation}
where $\Cpazo_T > 0$ is given as in Proposition \ref{prop:MomentumCI} and Theorem \ref{thm:LocalFilippov}. 


\paragraph*{Step 1 -- Construction of an admissible curve.}

Observe that under Hypotheses \ref{hyp:CI}, each element $v_{\tau} \in V(\tau,\mu_{\tau})$ is a time-independent vector field satisfying Hypotheses \ref{hyp:CE} with constants $m(\tau),l(\tau) \geq 0$. Thus, by Theorem \ref{thm:CE}, there exists a unique solution $\nu(\cdot) \in \AC([0,T],\Pcal_p(\R^d))$ of 
\begin{equation*}
\left\{
\begin{aligned}
& \partial_t \nu(t) + \Div_x (v_{\tau} \nu(t)) =0, \\
& \nu(\tau) = \mu_{\tau},
\end{aligned}
\right.
\end{equation*}
and the latter can be represented explicitly as 
\begin{equation*}
\nu(t) = \Phi^{v_{\tau}}_{(\tau,t) \,\sharp} \mu_{\tau}
\end{equation*}
for all times $t \in [0,T]$. Therein, the maps $(\Phi_{(\tau,t)}^{v_{\tau}})_{t \in [0,T]} \subset C^0(\R^d,\R^d)$ are the characteristic flows generated by $v_{\tau} \in V(\tau,\mu_{\tau})$ in the sense of Definition \ref{def:Flows}. By standard linearisation techniques (see e.g. \cite[Appendix A]{SemiSensitivity}), it can further be shown that 
\begin{equation*}
\Phi^{v_{\tau}}_{(\tau,\tau+h)}(x) = x + hv_{\tau}(x) + o_{\tau,x}(h)
\end{equation*}
for all $x \in \R^d$ and any sufficiently small $h>0$, where $\INTDom{|o_{\tau,x}(h)|^p}{\R^d}{\mu_{\tau}(x)} = o_{\tau}(|h|^p)$. Thence, upon remarking that 
\begin{equation*}
\Big( \Id + h v_{\tau} , \Phi_{(\tau,\tau+h)}^{v_{\tau}} \Big)_{\raisebox{4pt}{$\scriptstyle{\sharp}$}} \mu_{\tau} \in \Gamma \Big( (\Id + h v_{\tau})_{\sharp}\mu_{\tau} , \nu(\tau+h) \Big),
\end{equation*}
we straightforwardly deduce from \eqref{eq:WassEst} the distance estimate
\begin{equation}
\label{eq:ScorzaEst1}
W_p \Big( \nu(\tau+h),(\Id + h v_{\tau})_{\sharp} \mu_{\tau} \Big) \leq o_{\tau}(h),
\end{equation}
which holds for every sufficiently small $h> 0$. Moreover, since $\nu(\tau) = \mu_{\tau}$ by construction, there exists by the Filippov estimates of Theorem \ref{thm:LocalFilippov} a curve $\mu_{\epsilon}(\cdot) \in \Spazo_{[0,T]}(\tau,\mu_{\tau})$ which satisfies
\begin{equation}
\label{eq:ScorzaEst2}
W_p(\mu_{\epsilon}(t),\nu(t)) \leq \Cpazo_T' \bigg( \INTSeg{\dist_{C^0(B(0,R_{\epsilon}),\R^d)} \Big( v_{\tau} \, ; V(s,\nu(s)) \Big)}{s}{\tau}{t} + \Epazo_{\nu}(\tau,t,R_{\epsilon}) \bigg)
\end{equation}
for all times $t \in [\tau,T]$, where $\Cpazo_T' >0$ only depends on the magnitudes of $p,\Mpazo_p(\mu_{\tau}),\Norm{m(\cdot)}_1,\Norm{l(\cdot)}_1$, and $\Norm{L(\cdot)}_1$.


\paragraph*{Step 2 -- Distance estimate in the vicinity of $\tau \in \Tcal$.}

In order to conclude, we need to show that the right-hand side of \eqref{eq:ScorzaEst2} is bounded from above by  $\epsilon h + o_{\tau,\epsilon}(h)$ when $t = \tau+h$ with $h >0$ sufficiently small. Observe first that by our choice of $R_{\epsilon} > 0$ via \eqref{eq:RepsDef}, one has that
\begin{equation}
\label{eq:ScorzaEst3}
\begin{aligned}
\Epazo_{\nu}(\tau,\tau+h,R_{\epsilon}) & = 2 (1+\Cpazo_T) \NormL{m(\cdot)}{1}{[\tau,\tau+h]} \bigg( \INTDom{\big( 1+|x|^p \big)}{\{x \in \R^d ~ \textnormal{s.t.} \; |x| \geq R_{\epsilon}/\Cpazo_T -1\}}{\mu_{\tau}(x)} \bigg)^{1/p} \\
& \leq \epsilon h + o_{\tau,\epsilon}(h)
\end{aligned}
\end{equation}
whenever $h >0$ is small enough, since $\tau \in \Tcal$ is a one-sided Lebesgue point of $m(\cdot) \in L^1([0,T],\R_+)$.

In order to derive an upper-bound on the integral of the mismatch function, recall that by the definition \eqref{eq:TcalDef} of $\Tcal \subset (0,T)$, there exists an integer $n \geq 1$ such that $\tau \in \tilde{\Acal}_n$ and the set-valued map $t \in \Acal_n \tto V(t,\mu_{\tau})$ is lower-semicontinuous. Hence for each $\epsilon' > 0$ there exists some $\delta_n >0$ for which
\begin{equation*}
\dist_{C^0(\R^d,\R^d)} \Big( v_{\tau} \, ; V(t,\mu_{\tau}) \Big) \leq \epsilon'
\end{equation*}
for every $t \in [\tau,\tau+\delta_n] \cap \Acal_n$, where we recall that $(C^0(\R^d,\R^d),\dsf_{cc}(\cdot,\cdot))$ is equipped with the Fr\'echet structure described in Definition \ref{def:dcc}. In particular, by choosing $\epsilon' > 0$ to be sufficiently small, it follows from the definition \eqref{eq:dcc} of the metric $\dsf_{cc}(\cdot,\cdot)$ that 
\begin{equation}
\label{eq:ScorzaEst4}
\dist_{C^0(B(0,R_{\epsilon}),\R^d)} \Big( v_{\tau} \, ; V(t,\mu_\tau) \Big) \leq \epsilon, 
\end{equation}
for every $t \in [\tau,\tau+\delta_n] \cap \Acal_n$, where $(C^0(B(0,R_{\epsilon}),\R^d),\NormC{\cdot}{0}{B(0,R_{\epsilon}),\R^d})$ is endowed with its usual Banach space structure. Besides, noting that the sets $V(t,\mu_{\tau}) \subset C^0(\R^d,\R^d)$ are compact under Hypotheses \ref{hyp:CI}-$(ii)$ and $(iii)$ as a consequence of Theorem \ref{thm:Ascoli}, while observing that the map 
\begin{equation*}
w \in C^0(\R^d,\R^d) \mapsto \NormC{v_{\tau}-w}{0}{B(0,R_{\epsilon},\R^d} \in \R_+
\end{equation*}
is continuous, it follows from Theorem \ref{thm:MeasurableSel} applied with Hypothesis \ref{hyp:CI}-$(i)$ that there exists a measurable selection $t \in [0,T] \mapsto w_{\tau}(t) \in V(t,\mu_{\tau})$ such that 
\begin{equation}
\label{eq:ScorzaEst4bis}
\NormC{v_{\tau} - w_{\tau}(t)}{0}{B(0,R_{\epsilon}),\R^d} = \dist_{C^0(B(0,R_{\epsilon}),\R^d)} \Big( v_{\tau} \, ; V(t,\mu_\tau) \Big) 
\end{equation}
for every $t \in [\tau,\tau+\delta_n] \cap \Acal_n$. Furthermore, by Hypotheses \ref{hyp:CI}-$(i)$ and $(iv)$ combined again with Theorem \ref{thm:MeasurableSel}, there exists a measurable selection $t \in [0,T] \mapsto w(t) \in V(t,\nu(t))$ satisfying
\begin{equation*}
\NormC{w(t) - w_{\tau}(t)}{0}{B(0,R_{\epsilon}),\R^d} \leq \, L(t) W_p(\mu_{\tau},\nu(t))
\end{equation*}
for $\Lcal^1$-almost every $t \in [0,T]$, which together with \eqref{eq:ScorzaEst4} and \eqref{eq:ScorzaEst4bis} further yields 
\begin{equation}
\label{eq:ScorzaEst5}
\dist_{C^0(B(0,R_{\epsilon}),\R^d)} \Big( v_{\tau} \, ; V(t,\nu(t)) \Big) \leq \epsilon + L(t)W_p(\mu_{\tau},\nu(t))
\end{equation}
for $\Lcal^1$-almost every $t \in [\tau,\tau+\delta_n] \cap \Acal_n$. There now remains to estimate the integral over $[\tau,\tau+h]$ of the local mismatch function. The latter can be decomposed into the sum of two terms as 
\begin{equation}
\begin{aligned}
\label{eq:ScorzaEst61}
\INTSeg{\dist_{C^0(B(0,R_{\epsilon}),\R^d)} \Big( v_{\tau} \, ; V(t,\nu(t)) \Big)}{t}{\tau}{\tau + h} & = \INTDom{\dist_{C^0(B(0,R_{\epsilon}),\R^d)} \Big( v_{\tau} \, ; V(t,\nu(t)) \Big)}{[\tau,\tau + h] \setminus \Acal_n}{t} \\
& \hspace{0.45cm} + \INTDom{\dist_{C^0(B(0,R_{\epsilon}),\R^d)} \Big( v_{\tau} \, ; V(t,\nu(t)) \Big)}{[\tau,\tau+h] \cap \Acal_n}{t}. 
\end{aligned}
\end{equation}
As a consequence of Hypothesis \ref{hyp:CI}-$(ii)$, the first of these two integrals can be estimated as 
\begin{equation}
\label{eq:ScorzaEst62}
\begin{aligned}
& \INTDom{\dist_{C^0(B(0,R_{\epsilon}),\R^d)} \Big( v_{\tau} \, ; V(t,\nu(t)) \Big)}{[\tau,\tau+h] \setminus \Acal_n}{t} \\
& \hspace{2cm} \leq \INTDom{\bigg( \NormC{v_{\tau}}{0}{B(0,R_{\epsilon}),\R^d} +  \sup_{w \in V(t,\nu(t))} \NormC{w}{0}{B(0,R_{\epsilon}),\R^d} \bigg)}{[\tau,\tau+h] \setminus \Acal_n}{t} \\
& \hspace{2cm} \leq (1+R_{\epsilon} + \Cpazo_T) \INTDom{\big( m(\tau) + m(t) \big)}{[\tau,\tau+h] \setminus \Acal_n}{t} \\
& \hspace{2cm} \leq (1+R_{\epsilon} + \Cpazo_T) \bigg(  2 m(\tau) \Lcal^1([\tau,\tau+h] \setminus \Acal_n) + \INTDom{\big( m(t) - m(\tau) \big)}{[\tau,\tau+h] \setminus \Acal_n}{t} \bigg)  
\\
& \hspace{2cm} = o_{\tau,\epsilon}(h)
\end{aligned}
\end{equation}
where we used the characterisation \eqref{eq:DensityPoint} of the one-sided density points of $\Acal_n$, along with the fact that $\tau \in \Tcal$ is a one-sided Lebesgue point of $m(\cdot) \in L^1([0,T],\R_+)$, as well as the moment bound of Proposition \ref{prop:MomentumCI}. By \eqref{eq:ScorzaEst5}, the second term in \eqref{eq:ScorzaEst61} can be bounded from above as
\begin{equation}
\label{eq:ScorzaEst63}
\begin{aligned}
& \INTDom{\dist_{C^0(B(0,R_{\epsilon}),\R^d)} \Big( v_{\tau} \, ; V(t,\nu(t)) \Big)}{[\tau,\tau+h] \cap \Acal_n}{t} \\
& \hspace{2cm} \leq \INTDom{\Big( \epsilon + L(t) W_p(\mu_{\tau},\nu(t)) \Big)}{[\tau,\tau+h] \cap \Acal_n}{t} \\
& \hspace{2cm} \leq \epsilon h + \Cpazo_T' \INTSeg{L(t) \INTSeg{\NormC{v_{\tau}}{0}{B(0,R_{\epsilon}),\R^d}}{s}{\tau}{t} \,}{t}{\tau}{\tau+h} \\
& \hspace{2.9cm} + \Cpazo_T' \INTSeg{ 2 \NormL{m(\cdot)}{1}{[\tau,\tau+h]}(1 +\Cpazo_T) \bigg( \INTDom{\big( 1+|x| \big)^p}{\{x \; \textnormal{s.t.} \; |x| \geq R_{\epsilon}/\Cpazo_T -1\}}{\mu_{\tau}(x)}\bigg)^{1/p}}{t}{\tau}{\tau+h} \\
& \hspace{2cm} \leq \epsilon h + o_{\tau,\epsilon}(h),
\end{aligned} 
\end{equation}
where we used the single-valued version of the distance estimate of Theorem \ref{thm:LocalFilippov} -- see Remark \ref{rmk:LinkIncEq} -- along with the fact that $\tau \in \Tcal$ is a one-sided Lebesgue point of $m(\cdot),L(\cdot) \in L^1([0,T],\R_+)$. By plugging \eqref{eq:ScorzaEst62} and \eqref{eq:ScorzaEst63} into \eqref{eq:ScorzaEst61} and combining the resulting estimate with \eqref{eq:ScorzaEst2}, one then obtains
\begin{equation}
\label{eq:ScorzaEst6}
W_p \Big( \mu_{\epsilon}(\tau+h) , \nu(\tau+h) \Big) \leq \epsilon h + o_{\tau,\epsilon}(h)
\end{equation}
for every small $h \in (0,\delta_n]$. Upon merging \eqref{eq:ScorzaEst6} with \eqref{eq:ScorzaEst1} and taking $h_{\epsilon} > 0$ so that $o_{\tau,\epsilon}(h) \leq \epsilon h$ for all $h \in [0,h_{\epsilon}]$, one finally has up to rescaling $\epsilon > 0$ by a constant that
\begin{equation*}
W_p \Big( \mu_{\epsilon}(\tau+h) , (\Id + h v_{\tau})_{\sharp} \mu_{\tau} \Big) \leq \epsilon h, 
\end{equation*}
for all $h \in [0,h_{\epsilon}]$, which concludes the proof.
\end{proof}

In the following theorem, we establish a property that is complementary to the one we previously discussed, which ensures the existence of an admissible velocity representing (approximately) the local behaviour of any given solution of \eqref{eq:ContIncCauchy} when the admissible velocities are convex.

\begin{thm}[Infinitesimal behaviour of reachable sets]
\label{thm:Limsup}
Let $V : [0,T] \times \Pcal_p(\R^d) \tto C^0(\R^d,\R^d)$ be a set-valued map with convex images satisfying Hypotheses \ref{hyp:CI}. 

Then, there exists a subset $\Tcal \subset (0,T)$ of full $\Lcal^1$-measure such that for every $\tau \in \Tcal$, all $\mu_{\tau} \in \Pcal_p(\R^d)$, each solution $\mu(\cdot) \in \AC([0,T],\Pcal_p(\R^d))$ of \eqref{eq:ContIncCauchyBis}, any $\epsilon > 0$ and every sequence $h_i \to 0$, there exists an element $v_{\tau}^{\epsilon} \in V(\tau,\mu_{\tau})$ such that 
\begin{equation}
\label{eq:InfinitesimalReach}
W_p \Big( \mu(\tau+h_{i_k}^{\epsilon}) , (\Id + h_{i_k}^{\epsilon} v_{\tau}^{\epsilon})_{\sharp} \mu_{\tau} \Big) \leq \epsilon |h_{i_k}^{\epsilon}|
\end{equation}
along a subsequence $h_{i_k}^{\epsilon} \to 0$ which depends both on $v_{\tau}^{\epsilon} \in V(\tau,\mu_{\tau})$ and $\epsilon >0$. 
\end{thm}

\begin{proof}
First, let $\Tcal_m,\Tcal_l, \Tcal_L$ and $\Tcal_{\textnormal{H}}$ be the subsets of full $\Lcal^1$-measure in $(0,T)$ defined as in the proof of Theorem \ref{thm:Liminf}. By Theorem \ref{thm:ScorzaMulti}, there exists for each $k \geq 1$ a compact set $\Apazo^k \subset [0,T]$ satisfying $\Lcal^1([0,T] \setminus \Apazo^k) < \tfrac{1}{2^k}$, and such that the graph  
\begin{equation*}
\Graph (V)_{|\Apazo^k \times \Pcal_p(\R^d) \times C^0(\R^d,\R^d)} := \Big\{ (t,\mu,v) \in \Apazo^k \times \Pcal_p(\R^d) \times C^0(\R^d,\R^d) ~\, \textnormal{s.t.}~ v \in V(t,\mu) \Big\}
\end{equation*}
of the restriction of $V : [0,T] \times \Pcal_p(\R^d) \tto C^0(\R^d,\R^d)$ is closed in $\Apazo^k \times \Pcal_p(\R^d) \times C^0(\R^d,\R^d)$. For each $n \geq 1$, consider the increasing sequence of $\Lcal^1$-measurable sets $(\hat{\Bpazo}_n)$ defined by 
\begin{equation*}
\hat{\Bpazo}_n := \Big\{ t \in [0,T] ~\, \textnormal{s.t.}~ m(t) + l(t) \leq n \Big\}. 
\end{equation*}
Then, there exists an increasing subsequence $(\Bpazo_n)$ of Lebesgue measurable sets that we do not relabel satisfying $\Bpazo_n \subset \hat{\Bpazo}_n$ as well as $\Lcal^1(\Bpazo_n) >0$ for each $n \geq 1$, and such that 
\begin{equation*}
\Lcal^1([0,T] \setminus \Bpazo_n) ~\underset{n \to +\infty}{\longrightarrow}~ 0.
\end{equation*}
By the inner regularity of the Lebesgue measure, one can find an increasing sequence of closed sets $(\Cpazo_n) \subset [0,T]$ satisfying $\Cpazo_n \subset \Bpazo_n$ for each $n \geq 1$, as well as 
\begin{equation*}
\Lcal^1(\Cpazo_n) \geq \Big( 1 - \tfrac{1}{T}\Lcal^1([0,T] \setminus \Bpazo_n)  \Big) \Lcal^1(\Bpazo_n).
\end{equation*}
Then, for each $n \geq1$, define the closed set $\Acal_n \subset [0,T]$ by 
\begin{equation*}
\Acal_n := \bigg( \bigcap\nolimits_{k > n} \Apazo^k  \bigg) \cap \Cpazo_n, 
\end{equation*}
and as in the proof of Theorem \ref{thm:Liminf}, denote by $\tilde{\Acal}_n$ the subset of all one-sided density points of $\Acal_n$. Noting that the sequence of measurable sets $(\tilde{\Acal}_n)$ is increasing by construction, and that it satisfies
\begin{equation*}
\begin{aligned}
\Lcal^1([0,T] \setminus \tilde{\Acal}_n) & = \Lcal^1 \bigg( \bigcup\nolimits_{k>n} \Big( [0,T] \setminus \Apazo^k \Big) \cup \Big( [0,T] \setminus \Cpazo_n \Big) \bigg) \\
& \leq \sum_{k = n+1}^{+ \infty} \Lcal^1([0,T] \setminus \Apazo^k) + T - \Lcal^1(\Bpazo_n) + \Lcal^1([0,T] \setminus \Bpazo_n) \\
& \leq \frac{1}{2^n} + 2 \, \Lcal^1([0,T] \setminus \Bpazo_n) ~\underset{n \to +\infty}{\longrightarrow}~0, 
\end{aligned}
\end{equation*}
it follows that the set $\Tcal \subset (0,T)$ defined by
\begin{equation*}
\Tcal := \bigg( \bigcup\nolimits_{n \geq 1} \tilde{\Acal}_n \bigg) \cap \Tcal_m \cap \Tcal_l \cap \Tcal_L \cap \Tcal_{\textnormal{H}}
\end{equation*}
has full $\Lcal^1$-measure in $[0,T]$. Fix now some $\tau \in \Tcal$, a measure $\mu_{\tau} \in \Pcal_p(\R^d)$ as well as a solution $\mu(\cdot) \in \Spazo_{[0,T]}(\tau,\mu_{\tau})$ of \eqref{eq:ContIncCauchyBis} and a sequence $h_i \to 0^+$. Choose also $R_{\epsilon} > 0$ in such a way that 
\begin{equation}
\label{eq:RepsDefBis}
2m(\tau)(1+\Cpazo_T) \bigg( \INTDom{\big( 1+|x|^p \big)}{\{ x \in \R^d ~ \textnormal{s.t.} \; |x| \geq R_{\epsilon}\}}{\mu_{\tau}(x)} \bigg)^{1/p} \leq \epsilon,
\end{equation}
where $\Cpazo_T > 0$ is the constant appearing in Proposition \ref{prop:MomentumCI}, which we recall only depends on the magnitudes of $p,\Mpazo_p(\mu_{\tau})$ and $\Norm{m(\cdot)}_1$. 

By the definition of solutions to \eqref{eq:ContIncCauchyBis}, there exists an $\Lcal^1$-measurable selection $t \in [0,T] \mapsto v(t) \in V(t,\mu(t)) \subset C^0(\R^d,\R^d)$ such that the curve $\mu(\cdot)$ solves the Cauchy problem
\begin{equation*}
\left\{
\begin{aligned}
& \partial_t \mu(t) + \Div(v(t)\mu(t)) = 0, \\
& \mu(\tau) = \mu_{\tau}.
\end{aligned}
\right.
\end{equation*}
Besides, it can be deduced from Lemma \ref{lem:Carathéodory} that $(t,x) \in [0,T] \times \R^d \mapsto v(t,x) \in \R^d$ is a Carathéodory vector field which satisfies Hypotheses \ref{hyp:CE}, see for instance \cite{ContIncPp}. Hence by Theorem \ref{thm:CE}, the curve $\mu(\cdot) \in \AC([0,T],\Pcal_p(\R^d))$ is given explicitly by 
\begin{equation*}
\mu(t) = \Phi^v_{(\tau,t) \, \sharp} \mu_{\tau}
\end{equation*}
for all times $t \in[0,T]$. Moreover, there exists by construction an integer $n \geq 1$ such that $\tau \in \tilde{\Acal}_n$ and the restricted set-valued map $t \in \Acal_n \tto V(t,\mu(t))$ has closed graph. Since $m(\cdot)$ and $l(\cdot)$ are both bounded from above over $\Acal_n$, it stems from Hypotheses \ref{hyp:CI}-$(i)$, $(ii)$ and $(iii)$ together with Theorem \ref{thm:Ascoli} that there exists a compact set $\Kpazo_n \subset C^0(\R^d,\R^d)$ such that 
\begin{equation*}
V(t,\mu(t)) \subset \Kpazo_n
\end{equation*}
for all times $t \in \Acal_n$. Thence, it follows e.g. from \cite[Proposition 1.4.8]{Aubin1990} that $t \in \Acal_n \tto V(t,\mu(t))$ is upper-semicontinuous in the sense of Definition \ref{def:Continuity}-$(ii)$. In particular, for each $\epsilon > 0$, there exists a measurable selection $t \in [0,T] \mapsto \tilde{v}_{\tau}^{\epsilon}(t) \in V(\tau,\mu_{\tau})$ along with some $\delta_n > 0$ such that
\begin{equation}
\label{eq:LimsupSelection}
\NormC{v(t) - \tilde{v}_{\tau}^{\epsilon}(t)}{0}{B(0,R_{\epsilon}),\R^d} \leq \epsilon
\end{equation}
for every $t \in \Acal_n \cap [\tau,\tau+\delta_n]$, where we leveraged the expression \eqref{eq:dcc} of the metric $\dsf_{cc}(\cdot,\cdot)$.

In what follows, we use the selection $t \in [0,T] \mapsto \tilde{v}_{\tau}^{\epsilon}(t) \in V(\tau,\mu_{\tau})$ to build an element $v_{\tau}^{\epsilon} \in V(\tau,\mu_{\tau})$ satisfying \eqref{eq:InfinitesimalReach} along some subsequence $h_{i_k}^{\epsilon} \to 0$. We assume without loss of generality that $h_i \to 0^+$, the general case being similar. By the definition \eqref{eq:FlowDef} of characteristic flows, one has that
\begin{equation}
\label{eq:LimsupFlowExp}
\Phi^v_{(\tau,\tau+h_i)}(x) = x + \INTSeg{v(t,x)}{t}{\tau}{\tau+h_i} +  \INTSeg{\bigg( v \Big(t, \Phi^v_{(\tau,t)}(x) \Big) - v(t,x) \bigg)}{t}{\tau}{\tau+h_i}
\end{equation}
for all $x \in \R^d$. As a consequence of Hypothesis \ref{hyp:CE}-$(ii)$, the second integral term in the right-hand side of \eqref{eq:LimsupFlowExp} can be estimated from above for all $x \in B(0,R_{\epsilon})$ as
\begin{equation}
\label{eq:LimsupEst1}
\begin{aligned}
\bigg| \INTSeg{\bigg( v \Big(t, \Phi^v_{(\tau,t)}(x) \Big) - v(t,x) \bigg)}{t}{\tau}{\tau+h_i} \hspace{0.05cm} \bigg| & \leq \INTSeg{l(t) \big| \Phi_{(\tau,t)}^v(x) - x \big|}{t}{\tau}{\tau+h_i} \\
& 
\leq (1+R_{\epsilon} + \Cpazo_T) \INTSeg{l(t) \bigg( \INTSeg{m(s)}{s}{\tau}{t} \bigg)}{t}{\tau}{\tau+h_i} \\
& = o_{\tau,\epsilon}(h_i)
\end{aligned}
\end{equation}
for $h_i > 0$ sufficiently small, since $\tau \in \Tcal$ is a one-sided Lebesgue point of $m(\cdot),l(\cdot) \in L^1([0,T],\R_+)$. At this stage, one may further decompose the first integral in the right-hand side of \eqref{eq:LimsupFlowExp} into
\begin{equation}
\label{eq:LimsupEst20}
\begin{aligned}
\INTSeg{v(t,x)}{t}{\tau}{\tau+h_i} = \INTSeg{\tilde{v}_{\tau}^{\epsilon}(t,x)}{t}{\tau}{\tau+h_i} & + \INTDom{\Big( v(t,x) - \tilde{v}_{\tau}^{\epsilon}(t,x) \Big)}{[\tau,\tau+h_i] \setminus \Acal_n}{t} \\
& + \INTDom{\Big( v(t,x) - \tilde{v}_{\tau}^{\epsilon}(t,x) \Big)}{[\tau,\tau+h_i] \cap \Acal_n}{t}
\end{aligned}
\end{equation}
and use Hypothesis \ref{hyp:CI}-$(ii)$ to estimate the second expression in \eqref{eq:LimsupEst20} as
\begin{equation}
\label{eq:LimsupEst21}
\bigg| \INTDom{\Big( v(t,x) - \tilde{v}_{\tau}^{\epsilon}(t,x) \Big)}{[\tau,\tau+h_i] \setminus \Acal_n}{t} \, \bigg| \leq 2(1+R_{\epsilon} + \Cpazo_T) \INTDom{m(t)}{[\tau,\tau+h_i] \setminus \Acal_n}{t} = o_{\tau,\epsilon}(h_i)
\end{equation}
for all $x \in B(0,R_{\epsilon})$, since $\tau \in \Tcal$ is a one-sided Lebesgue point of $m(\cdot)$ as well as a one-sided density point of $\Acal_n$. Regarding the third term in the right-hand side of \eqref{eq:LimsupEst20}, it follows from \eqref{eq:LimsupSelection} that 
\begin{equation}
\label{eq:LimsupEst22}
\bigg\| \INTDom{\Big( v(t) - \tilde{v}_{\tau}^{\epsilon}(t) \Big)}{[\tau,\tau+h_i] \cap \Acal_n}{t} \, \bigg\|_{C^0(B(0,R_{\epsilon}),\R^d)} \leq \epsilon h_i, 
\end{equation}
and by merging the estimates of \eqref{eq:LimsupEst21}-\eqref{eq:LimsupEst22} with \eqref{eq:LimsupEst20} while inserting the resulting expression together with \eqref{eq:LimsupEst1} inside \eqref{eq:LimsupFlowExp}, one further obtains that
\begin{equation}
\label{eq:LimsupFlowExp2}
\bigg| \, \Phi_{(\tau,\tau+h_i)}^v(x) - x - \INTSeg{\tilde{v}_{\tau}^{\epsilon}(t,x)}{t}{\tau}{\tau+h_i} \, \bigg| \leq \epsilon h_i + o_{\tau,\epsilon}(h_i)
\end{equation}
for all $x \in B(0,R_{\epsilon})$. Observe now that since $V(\tau,\mu_{\tau})_{|B(0,R_{\epsilon})} \subset C^0(B(0,R_{\epsilon}),\R^d)$ is convex and closed, there exists by the separation theorem a family of elements $(v_{\tau}^{\epsilon}(h_i)) \subset V(\tau,\mu_{\tau})$ such that  
\begin{equation*}
\frac{1}{h_i} \INTSeg{\tilde{v}_{\tau}^{\epsilon}(t)_{|B(0,R_{\epsilon})}}{t}{\tau}{\tau+h_i} = v_{\tau}^{\epsilon}(h_i)_{|B(0,R_{\epsilon})}. 
\end{equation*}
Besides, noting that $V(\tau,\mu_{\tau}) \subset C^0(\R^d,\R^d)$ is a compact set as a consequence of Hypotheses \ref{hyp:CI}-$(ii)$ and $(iii)$ and Theorem \ref{thm:Ascoli}, there exists some $v_{\tau}^{\epsilon} \in V(\tau,\mu_{\tau})$ for which 
\begin{equation*}
\dsf_{cc} \big( v_{\tau}^{\epsilon}(h_{i_k}^{\epsilon}) , v_{\tau}^{\epsilon} \big) ~\underset{h_{i_k}^{\epsilon} \to \, 0^+}{\longrightarrow}~ 0
\end{equation*}
along a subsequence $h_{i_k}^{\epsilon} \to 0^+$. In particular for $h_{i_k}^{\epsilon} >0$ sufficiently small, one has that
\begin{equation*}
\bigg\| \INTSeg{\tilde{v}_{\tau}^{\epsilon}(t)}{t}{\tau}{\tau+h_{i_k}^{\epsilon}} - h_{i_k}^{\epsilon} v_{\tau}^{\epsilon} \, \bigg\|_{C^0(B(0,R_{\epsilon}),\R^d)} = o_{\tau,\epsilon}(h_{i_k}^{\epsilon}), 
\end{equation*}
which together with \eqref{eq:LimsupFlowExp2} finally yields that 
\begin{equation}
\label{eq:LimsupFinal1}
\bigg| \, \Phi_{(\tau,\tau+h_{i_k}^{\epsilon})}^v(x) - x - h_{i_k}^{\epsilon} v_{\tau}^{\epsilon}(x) \, \bigg| \leq \epsilon h_{i_k}^{\epsilon} + o_{\tau,\epsilon}(h_{i_k}^{\epsilon}), 
\end{equation}
for all $x \in B(0,R_{\epsilon})$. To conclude the proof, there remains to observe that 
\begin{equation}
\label{eq:LimsupFinal2}
\Big( \Phi_{(\tau,\tau+h_{i_k}^{\epsilon})}^{v} , \Id + h_{i_k}^{\epsilon} v_{\tau}^{\epsilon} \Big)_{\raisebox{4pt}{$\scriptstyle{\sharp}$}} \mu_{\tau} \in \Gamma \Big( \mu(\tau+h_{i_k}^{\epsilon}) , (\Id + h_{i_k}^{\epsilon} v_{\tau}^{\epsilon})_{\sharp}\mu_{\tau} \Big)
\end{equation}
and to use the standard Wasserstein inequality \eqref{eq:WassEst} along with the estimate
\begin{equation}
\label{eq:LimsupFinal3}
\begin{aligned}
& \bigg( \INTDom{\big| \Phi_{(\tau,\tau+h_{i_k}^{\epsilon})}^{v}(x) - x - h_{i_k}^{\epsilon} v_{\tau}^{\epsilon}(x) \big|^p}{\big\{x \in \R^d \; \textnormal{s.t.}\; |x| \geq R_{\epsilon}\big\}}{\mu_{\tau}(x)} \bigg)^{1/p} \\
& \hspace{2.8cm} \leq (1+\Cpazo_T) \Big( \NormL{m(\cdot)}{1}{[\tau,\tau+h_{i_k}^{\epsilon}]} + h_{i_k}^{\epsilon} m(\tau) \Big) \bigg( \INTDom{(1+|x|)^p}{\big\{x \in \R^d \; \textnormal{s.t.}\; |x| \geq R_{\epsilon} \big\}}{\mu_{\tau}(x)} \bigg)^{1/p} \\
& \hspace{2cm} \leq \epsilon h_{i_k}^{\epsilon} + o_{\tau,\epsilon}(h_{i_k}^{\epsilon})
\end{aligned}
\end{equation}
which holds as a consequence of Hypothesis \ref{hyp:CI}-$(ii)$ combined with the equi-integrability bound of Proposition \ref{prop:MomentumCI}, the definition  \eqref{eq:RepsDefBis} of $R_{\epsilon} > 0$, and the fact that $\tau \in \Tcal$ is a one-sided Lebesgue point of $m(\cdot) \in L^1([0,T],\R_+)$. Thence, by merging \eqref{eq:LimsupFinal1}, \eqref{eq:LimsupFinal2} and \eqref{eq:LimsupFinal3}, one finally recovers that
\begin{equation*}
W_p \Big( \mu(\tau+h_{i_k}^{\epsilon}) , (\Id + h_{i_k}^{\epsilon} v_{\tau}^{\epsilon})_{\sharp} \mu_{\tau} \Big) \leq \epsilon h_{i_k}^{\epsilon} 
\end{equation*}
whenever $h_{i_k}^{\epsilon}>0$ is taken sufficiently small, in particular to ensure that $o_{\tau,\epsilon}(h_{i_k}^{\epsilon}) \leq \epsilon h^{\epsilon}_{i_k}$, and up to rescaling the free parameter $\epsilon > 0$ by a fixed constant. The case of an unsigned sequence $h_i \to 0$ being completely similar, this concludes the proof. 
\end{proof}

\begin{rmk}[More general assumptions for Theorem \ref{thm:Liminf} and Theorem \ref{thm:Limsup}]
While the two theorems displayed in this section have been proven under the Cauchy-Lipschtz regularity assumptions \ref{hyp:CI}, we expect them to hold under less stringent requirements, where one only asks for the local uniform continuity of the dynamics with respect to the space and measure variables in the spirit of \cite[Section 3]{ContIncPp}. For the sake of conciseness, we postpone these refinements to an ulterior work.
\end{rmk}

\begin{rmk}[Exact differential formulations for compactly supported measures]
When the initial data $\mu_{\tau} \in \Pcal_p(\R^d)$ are compactly supported, the quantitative estimates on the difference quotients derived in Theorem \ref{thm:Liminf} and Theorem \ref{thm:Limsup} can be expressed more simply in terms of metric derivatives. In this case, following e.g. \cite[Proposition 3]{ContInc}, there exists for each $r>0$ a radius $R_r > 0$ such that each $\mu(\cdot) \in \Spazo_{[0,T]}(\tau,\mu_{\tau})$ satisfies $\mu(t) \in \Pcal(B(0,R_r))$ for all times $t \in [0,T]$, whenever $\mu_{\tau} \in \Pcal(B(0,r))$. Owing to this uniform support bound, one can check that the error terms $o_{\tau,\epsilon}(h)$ in the proofs of both Theorem \ref{thm:Liminf} and Theorem \ref{thm:Limsup} become independent of $\epsilon > 0$, which consequently means that the statements of these results hold respectively for every small $h >0$ for the former, and along a subsequence $h_{i_k} \to 0$ for the latter, which are both independent of $\epsilon >0$.

In Theorem \ref{thm:Liminf}, upon making the additional assumption that the velocity sets $V(\tau,\mu_{\tau})$ are convex, it follows from Proposition \ref{prop:Topological} that the solution set $\Spazo_{[\tau,T]}(\mu_{\tau})$ is compact. Then, one can first then let $\epsilon \to 0^+$ up to considering a subsequence, and then take the limit as $h \to 0^+$ to  obtain that for every $v_{\tau} \in V(\tau,\mu_{\tau})$, there exists a solution $\mu(\cdot) \in \AC([0,T],\Pcal_p(\R^d))$ of \eqref{eq:ContIncCauchyBis} such that 
\begin{equation*}
\lim_{h \to 0^+} \tfrac{1}{h} W_p \Big( \mu(\tau+h) , (\Id + h v_{\tau})_{\sharp} \mu_{\tau} \Big) = 0. 
\end{equation*}
Similarly for Theorem \ref{thm:Limsup}, upon observing that under Hypotheses \ref{hyp:CI} the set $V(\tau,\mu_{\tau})$ is compact for the topology of local uniform convergence by the Ascoli-Arzel\`a theorem, one can let $\epsilon \to 0^+$ in conjunction with Lemma \ref{lem:ConvdccWass} to obtain that, for each solution $\mu(\cdot) \in \AC([0,T],\Pcal_p(\R^d))$ of \eqref{eq:ContIncCauchyBis} and every sequence $h_i \to 0$, there exists an element $v_{\tau} \in V(\tau,\mu_{\tau})$ such that
\begin{equation*}
\lim_{h_{i_k} \to 0} \tfrac{1}{|h_{i_k}|} W_p \Big( \mu(\tau+h_{i_k}) , (\Id + h_{i_k} v_{\tau})_{\sharp} \mu_{\tau} \Big) = 0
\end{equation*}
along a subsequence $h_{i_k} \to 0$. 
\end{rmk}


\section{Viability and invariance theorems for proper constraints sets}
\label{section:CauchyLip}
\setcounter{equation}{0} \renewcommand{\theequation}{\thesection.\arabic{equation}}

In this section, we discuss several necessary and sufficient conditions for the viability and invariance of general constraint tubes $\Qpazo : [0,T] \tto \Pcal_p(\R^d)$ under the action of the dynamics 
\begin{equation}
\label{eq:CauchyViab}
\left\{
\begin{aligned}
& \partial_t \mu(t) \in - \Div_x \Big( V(t,\mu(t)) \mu(t) \Big), \\
& \mu(\tau) = \mu_{\tau}.
\end{aligned}
\right.
\end{equation}
In this context, we will always assume that $V : [0,T] \times \Pcal_p(\R^d) \tto C^0(\R^d,\R^d)$ is a set-valued map satisfying Hypotheses \ref{hyp:CI} for some $p \in (1,+\infty)$, and that $(\tau,\mu_{\tau}) \in [0,T] \times \Pcal_p(\R^d)$ represents a some arbitrary initial condition complying with the admissibility constraint $\mu_{\tau} \in \Qpazo(\tau)$.

\begin{Def}[Viability and invariance]
\label{def:Viab}
We say that the tube $\Qpazo : [0,T] \tto \Pcal_p(\R^d)$ is \textnormal{viable} for \eqref{eq:CauchyViab} if for every $\tau \in [0,T]$ and any $\mu_{\tau} \in \Qpazo(\tau)$, there exists a solution $\mu(\cdot) \in \Spazo_{[\tau,T]}(\tau,\mu_{\tau})$ of the Cauchy problem such that 
\begin{equation*}
\mu(t) \in \Qpazo(t)
\end{equation*}
for all times $t \in [\tau,T]$. Similarly, we say $\Qpazo : [0,T] \tto \Pcal_p(\R^d)$ is \textnormal{invariant} for \eqref{eq:CauchyViab} if all the solution curves $\mu(\cdot) \in \Spazo_{[\tau,T]}(\tau,\mu_{\tau})$ satisfy $\mu(t) \in \Qpazo(t)$ for all times $t \in [\tau,T]$. 
\end{Def}

In what follows, we split the exposition of our main results into two separate parts, starting in Section \ref{subsection:ViabCauchy} with the simpler case in which the constraints sets are stationary, and then treating the more involved situation in which they are time-dependent in Section \ref{subsection:ViabCauchyAbs}.


\subsection{The case of stationary constraints}
\label{subsection:ViabCauchy}

We start our investigation of the viability and invariance under the action of continuity inclusions by considering the case in which the constraints are represented by a fixed proper subset $\Qpazo \subset \Pcal_p(\R^d)$. In this context, the crucial geometric object which allows to characterise the viability or the invariance of the latter is its \textit{contingent cone}, whose definition is inspired by that of \cite{Badreddine2022} and presented below. 

\begin{Def}[Contingent cone in Wasserstein spaces]
\label{def:Cone}
Given a closed set $\Qpazo \subset \Pcal_p(\R^d)$, we define its \textnormal{contingent cone} at some $\mu \in \Qpazo$ by 
\begin{equation*}
T_{\Qpazo}(\mu) := ~ \bigg\{ \xi \in \Lpazo^p(\R^d,\R^d;\mu) ~\, \textnormal{s.t.}~  \liminf_{h \to 0^+} \, \tfrac{1}{h} \dist_{\Pcal_p(\R^d)} \Big( (\Id + h \xi)_{\sharp} \mu \, ; \Qpazo \Big) = 0 \bigg\}.
\end{equation*}
The latter can also be characterised in terms of sequences as 
\begin{equation*}
\begin{aligned}
T_{\Qpazo}(\mu) & = \bigg\{ \xi \in \Lpazo^p(\R^d,\R^d;\mu) ~\, \textnormal{s.t.}~ \text{there exists a sequence $h_i \to 0^+$} \\
& \hspace{6.5cm} \text{for which}~ \dist_{\Pcal_p(\R^d)} \Big( (\Id + h_i \xi)_{\sharp} \mu \, ; \Qpazo \Big) = o(h_i) \bigg\}.
\end{aligned}
\end{equation*}
\end{Def}

\begin{rmk}[On the choice of defining contingent cones using Borel maps]
It is worth noting that eventhough $\Lpazo^p(\R^d,\R^d;\mu)$ is merely a seminormed space for any $\mu \in \Pcal_p(\R^d)$, the contingent cones introduced in Definition \ref{def:Cone} are closed for the convergence induced by the pseudometric $\NormLp{\cdot}{p}{\R^d,\R^d;\,\mu}$. Indeed, while the limit of a sequence $(\xi_n) \subset T_{\Qpazo}(\mu)$ should be a $\mu$-measurable map by construction (see e.g. \cite[Definition 1.12]{AmbrosioFuscoPallara}), the latter always coincides with a Borel function outside of a Borel set with zero $\mu$-measure by \cite[Proposition 2.1.11]{Bogachev}. In what ensues, the closures of the convex hulls of contingent cones will therefore always be understood in with respect to the pseudometric $\NormLp{\cdot}{p}{\R^d,\R^d;\,\mu}$.
\end{rmk}

By leveraging this notion, we are able to prove the following sufficient viability conditions for proper time-independent constraints sets, which is one of our main contributions. We chose to start by presenting this latter separately, as it thoroughly illustrates the main ideas supporting the more general results of Section \ref{subsection:ViabCauchyAbs} a simpler setting.

\begin{thm}[Sufficient viability conditions for stationary constraints]
\label{thm:ViabStationary}
Suppose that $p \in (1,+\infty)$, let $V : [0,T] \times \Pcal_p(\R^d) \tto C^0(\R^d,\R^d)$ be a set-valued map with convex images satisfying Hypotheses \ref{hyp:CI} and $\Qpazo \subset \Pcal_p(\R^d)$ be a proper set such that 
\begin{equation}
\label{eq:ViabCompactHyp}
V(t,\nu) \cap \co T_{\Qpazo}(\nu) \neq \emptyset 
\end{equation}
for $\Lcal^1$-almost every $t \in [0,T]$ and each $\nu \in \Qpazo$. Then $\Qpazo$ is viable for \eqref{eq:CauchyViab}. 
\end{thm}

In the proof of this theorem and several others in the manuscript, we will extensively use the following regularity property of the reachable sets. 

\begin{lem}[Regularity in time of the reachable sets]
\label{lem:ACReach}
If $V : [0,T] \times \Pcal_p(\R^d) \tto C^0(\R^d,\R^d)$ is a set-valued map with convex images satisfying Hypotheses \ref{hyp:CI}, then the reachable maps $t \in [\tau,T] \tto \Rpazo_{(\tau,t)}(\mu_{\tau}) \subset \Pcal_p(\R^d)$ are absolutely continuous in the Hausdorff metric for all $(\tau,\mu_{\tau}) \in [0,T] \times \Pcal_p(\R^d)$.  
\end{lem}

\begin{proof}
Let $t_1,t_2 \in [\tau,T]$ be such that $\tau \leq t_1 < t_2 \leq T$, and  observe that for each $\mu_{t_2} \in \Rpazo_{(\tau,t_2)}(\mu_{\tau})$, there exists some $\mu_{t_1} \in \Rpazo_{(\tau,t_1)}(\mu_{\tau})$ such that $\mu_{t_2} \in \Rpazo_{(t_1,t_2)}(\mu_{t_1})$. Thence, it follows from the regularity estimate \eqref{eq:ACCI} of Proposition \ref{prop:MomentumCI} that 
\begin{equation*}
\dist_{\Pcal_p(\R^d)} \Big( \mu_{t_2} \, ; \Rpazo_{(\tau,t_1)}(\mu_{\tau}) \Big) \leq W_p(\mu_{t_2},\mu_{t_1}) \leq (1+\Cpazo_T) \INTSeg{m(s)}{s}{t_1}{t_2} 
\end{equation*}
for some constant $\Cpazo_T >0$ that only depends on the magnitudes of $p,\Mpazo(\mu_{\tau})$ and $\Norm{m(\cdot)}_1$. Analogously, it can be shown that 
\begin{equation*}
\dist_{\Pcal_p(\R^d)} \Big( \mu_{t_1} \, ; \Rpazo_{(\tau,t_2)}(\mu_{\tau}) \Big) \leq (1+\Cpazo_T) \INTSeg{m(s)}{s}{t_1}{t_2},
\end{equation*}
for every $\mu_{t_1} \in \Rpazo_{(\tau,t_1)}(\mu_{\tau})$. By combining both estimates while remarking that 
\begin{equation*}
\begin{aligned}
\dsf_{\Hpazo} \Big( \Rpazo_{(\tau,t_1)}(\mu_{\tau}) , \Rpazo_{(\tau,t_2)}(\mu_{\tau}) \Big) = \max \bigg\{ & \sup \Big\{ \dist_{\Pcal_p(\R^d)} \Big( \mu_{t_1} \, ; \Rpazo_{(\tau,t_2)}(\mu_{\tau}) \Big) ~\,\textnormal{s.t.}~ \mu_{t_1} \in \Rpazo_{(\tau,t_1)}(\mu_{\tau})  \Big\} , \\
& \sup \Big\{ \dist_{\Pcal_p(\R^d)} \Big( \mu_{t_2} \, ; \Rpazo_{(\tau,t_1)}(\mu_{\tau}) \Big) ~\,\textnormal{s.t.}~ \mu_{t_2} \in \Rpazo_{(\tau,t_2)}(\mu_{\tau})  \Big\} \bigg\},
\end{aligned}
\end{equation*}
it then holds that 
\begin{equation*}
\dsf_{\Hpazo} \Big( \Rpazo_{(\tau,t_1)}(\mu_{\tau}) , \Rpazo_{(\tau,t_2)}(\mu_{\tau}) \Big) \leq (1+\Cpazo_T) \INTSeg{m(s)}{s}{t_1}{t_2}, 
\end{equation*}
which concludes the proof.
\end{proof}

This technical result being established, we can move on to the proof of Theorem \ref{thm:ViabStationary}.

\begin{proof}[Proof of Theorem \ref{thm:ViabStationary}]
In what follows, we assume without loss of generality that $(\tau,\mu_{\tau}) = (0,\mu^0)$ in \eqref{eq:CauchyViab} for some $\mu^0 \in \Qpazo$. Our goal will be to show that the function measuring the distance between $\Qpazo$ and the reachable sets, which is defined for all times $t \in [0,T]$ by 
\begin{equation*}
g(t) := \dist_{\Pcal_p(\R^d)} \Big( \Rpazo_{(0,t)}(\mu^0) \, ; \Qpazo \Big),
\end{equation*}
is identically equal to zero. By Proposition \ref{prop:AC} and Lemma \ref{lem:ACReach}, it holds that $g(\cdot) \in \AC([0,T],\R_+)$, and we denote by $\Dcal \subset (0,T)$ the set of full $\Lcal^1$-measure over which it is differentiable. Note also that the set $\cup_{t \in [0,T]} \Rpazo_{(0,t)}(\mu^0) \subset \Pcal_p(\R^d)$ is compact by Proposition \ref{prop:Topological} and Lemma \ref{lem:ACReach}. Since $\mu^0 \in \Qpazo$, it is possible to choose a radius $R > 0$ such that 
\begin{equation}
\label{eq:RViabDef}
\dist_{\Pcal_p(\R^d)} \Big( \Rpazo_{(0,t)}(\mu^0) \, ; \partial \B_{\Pcal_p(\R^d)}(\mu^0,R)\Big) \geq \dist_{\Pcal_p(\R^d)} \Big( \Rpazo_{(0,t)}(\mu^0) \, ; \Qpazo \Big) + 1
\end{equation}
for all times $t \in [0,T]$, by following e.g. the arguments detailed in Appendix \ref{section:AppendixAC}. Therefore, the set $\Qpazo_R := \Qpazo \cap \B_{\Pcal_p(\R^d)}(\mu^0,R)$, which is nonempty as well as compact by construction, is such that 
\begin{equation*}
g(t) = \dist_{\Pcal_p(\R^d)} \Big( \Rpazo_{(0,t)}(\mu^0) \, ; \Qpazo_R \Big)
\end{equation*}
for all times $t \in [0,T]$. 


\paragraph*{Step 1 -- A Gr\"onwall estimate on the distance function.}

In this first step, we show by contradiction that the map $g : [0,T] \to \R_+$ is identically equal to zero. Otherwise, since $g(0) = 0$, there should exist a time $t \in [0,T)$ and some $\delta > 0$ such that $g(t) = 0$ and $g(\tau) > 0$ for $\tau \in (t,t+\delta)$. 

Let $\Tcal \subset (0,T)$ be the subset of full $\Lcal^1$-measure over which the statement of Theorem \ref{thm:Liminf} and Hypotheses \ref{hyp:CI}-$(ii)$, $(iii)$ and $(iv)$ hold, and fix an element $\tau \in (t,t+\delta) \cap \Tcal \cap \Dcal$. Since $\Rpazo_{(0,\tau)}(\mu^0)$ and $\Qpazo_R$ are both compact, one has that
\begin{equation}
\label{eq:ViabCompact0}
g(\tau) = W_p(\mu_{\tau},\nu_{\tau}) 
\end{equation}
for some $\mu_{\tau} \in \Rpazo_{(0,\tau)}(\mu^0)$ and $\nu_{\tau} \in \Qpazo_R$. Moreover, recalling that $R>0$ is defined in such a way that \eqref{eq:RViabDef} is satisfied, it necessarily holds 
\begin{equation*}
\nu_{\tau} \in \Qpazo \cap \textnormal{int} \Big( \B_{\Pcal_p(\R^d)}(\mu^0,R) \Big),
\end{equation*}
which implies in particular that $T_{\Qpazo_R}(\nu_{\tau}) = T_{\Qpazo}(\nu)$. Thence, by using the definition of contingent cones provided in Definition \ref{def:Cone}, there exists for each $\xi_{\tau} \in T_{\Qpazo}(\nu_{\tau})$ a sequence $h_i \to 0^+$ such that 
\begin{equation}
\label{eq:ViabCompact1}
\begin{aligned}
W_p \Big(\mu_{\tau}, (\Id + h_i \xi_{\tau})_{\sharp} \nu_{\tau} \Big) & \geq \dist_{\Pcal_p(\R^d)} \big( \mu_{\tau} \, ; \Qpazo_R \big) + o_{\tau}(h_i) \\ 
& = W_p(\mu_{\tau},\nu_{\tau}) + o_{\tau}(h_i).
\end{aligned}
\end{equation}
Besides, by the directional superdifferentiability property of Proposition \ref{prop:SuperdiffWass} above, one further has
\begin{equation}
\label{eq:ViabCompact2}
\tfrac{1}{p} W_p^p \Big(\mu_{\tau}, (\Id + h_i \xi_{\tau})_{\sharp} \nu_{\tau} \Big) - \tfrac{1}{p} W_p^p(\mu_{\tau},\nu_{\tau}) \leq h_i \INTDom{\big\langle \xi_{\tau}(y) , j_p(y-x) \big\rangle}{\R^{2d}}{\gamma_{\tau}(x,y)} + o_{\tau}(h_i)
\end{equation}
for any small $h_i > 0$ and every $\gamma_{\tau} \in \Gamma_o(\mu_{\tau},\nu_{\tau})$. Hence, by combining the estimates of \eqref{eq:ViabCompact1} and \eqref{eq:ViabCompact2}, dividing the resulting expression by $h_i > 0$ and letting $h_i \to 0^+$, one obtains the following inequality
\begin{equation}
\label{eq:ViabCompact3}
\INTDom{\big\langle \xi_{\tau}(y) , j_p(x-y) \big\rangle}{\R^{2d}}{\gamma_{\tau}(x,y)} \leq 0, 
\end{equation}
which holds for every $\xi_{\tau} \in T_{\Qpazo}(\nu_{\tau})$ and each $\gamma_{\tau} \in \Gamma_o(\mu_{\tau},\nu_{\tau})$. 

Observe now that by Theorem \ref{thm:Liminf}, there exists for any $\epsilon > 0$ and every $v_{\tau} \in V(\tau,\mu_{\tau})$ a curve of measures $\mu_{\epsilon}(\cdot) \in \Spazo_{[0,T]}(\tau,\mu_{\tau})$ such that 
\begin{equation}
\label{eq:ViabCompact31}
W_p \Big( \mu_{\epsilon}(\tau+h) , (\Id + h v_{\tau})_{\sharp} \mu_{\tau} \Big) \leq  \epsilon h, 
\end{equation}
whenever $h > 0$ is sufficiently small. Furthermore, since $\mu_{\epsilon}(\tau+h) \in \Rpazo_{(\tau,\tau+h)}(\mu_{\tau}) \subset \Rpazo_{(0,\tau+h)}(\mu^0)$, one can estimate from above the forward difference quotient of $\tfrac{1}{p} g^p(\cdot)$ at $\tau \in (t,t+\delta) \cap \Tcal \cap \Dcal$ as
\begin{equation}
\label{eq:ViabCompact4}
\tfrac{1}{p} g^p(\tau + h) - \tfrac{1}{p} g^p(\tau) \leq \tfrac{1}{p} W_p^p \big( \mu_{\epsilon}(\tau+h) , \nu_{\tau} \big) - \tfrac{1}{p}  W_p^p( \mu_{\tau} , \nu_{\tau}).
\end{equation}
Besides, assuming without loss of generality that $\epsilon,h \in (0,1]$ and noting that $v_{\tau} \in \Lpazo^p(\R^d,\R^d;\mu_{\tau})$ by Hypothesis \ref{hyp:CI}-$(ii)$, it can be deduced from \eqref{eq:ViabCompact31} along with the estimates of Lemma \ref{lem:PnormEst} below that 
\begin{equation}
\label{eq:ViabCompact4Bis}
\begin{aligned}
\tfrac{1}{p} W_p^p \big( \mu_{\epsilon}(\tau+h) , \nu_{\tau} \big) & - \tfrac{1}{p} W_p^p \Big( (\Id + h v_{\tau})_{\sharp} \mu_{\tau} , \nu_{\tau} \Big) \\
& \leq W_p^{p-1} \Big( (\Id + h v_{\tau})_{\sharp} \mu_{\tau} , \nu_{\tau} \Big) \bigg( W_p \big( \mu_{\epsilon}(\tau+h) , \nu_{\tau} \big) -W_p  \Big( (\Id + h v_{\tau})_{\sharp} \mu_{\tau} , \nu_{\tau} \Big) \bigg) \\
& \hspace{0.45cm} + C_p \Big| W_p \big( \mu_{\epsilon}(\tau+h) , \nu_{\tau} \big) -W_p  \Big( (\Id + h v_{\tau})_{\sharp} \mu_{\tau} , \nu_{\tau} \Big)\Big|^{\min\{p,2\}} \\
& \leq C_p' \bigg( W_p \Big( \mu_{\epsilon}(\tau+h) , (\Id + h v_{\tau})_{\sharp} \mu_{\tau} \Big) + W_p \Big( \mu_{\epsilon}(\tau+h) , (\Id + h v_{\tau})_{\sharp} \mu_{\tau} \Big)^{\min\{p,2\}} \bigg) \\
& \leq C_p' \epsilon h + o_{\tau,\epsilon}(h)
\end{aligned}
\end{equation}
for $h > 0$ sufficiently small, and where the constants $C_p,C_p' > 0$ only depend on the magnitudes of $p,\Mpazo(\mu_{\tau}),\Mpazo(\nu_{\tau})$, $\Norm{m(\cdot)}_1$ and $\NormLp{v_{\tau}}{p}{\R^d,\R^d ; \, \mu_{\tau}}$. Thus, by merging the estimate of \eqref{eq:ViabCompact4Bis} with \eqref{eq:ViabCompact4}, one further obtains up to rescaling $\epsilon >0$ by a positive constant that
\begin{equation}
\label{eq:ViabCompact41}
\tfrac{1}{p} g^p(\tau + h) - \tfrac{1}{p} g^p(\tau) \leq \tfrac{1}{p} W_p^p \Big( (\Id + h v_{\tau})_{\sharp} \mu_{\tau} , \nu_{\tau} \Big) - \tfrac{1}{p} W_p^p(\mu_{\tau},\nu_{\tau}) + \epsilon h + o_{\tau,\epsilon}(h)
\end{equation}
Besides, it follows again from the directional superdifferentiability property of Proposition \ref{prop:SuperdiffWass} that 
\begin{equation}
\label{eq:ViabCompact42}
\tfrac{1}{p} W_p^p \Big( (\Id + h v_{\tau})_{\sharp} \mu_{\tau} , \nu_{\tau} \Big) - \tfrac{1}{p}  W_p^p( \mu_{\tau} , \nu_{\tau}) \leq h \INTDom{\big\langle v_{\tau}(x) , j_p(x-y) \big\rangle}{\R^{2d}}{\gamma_{\tau}(x,y)} + o_{\tau,\epsilon}(h) 
\end{equation}
for each optimal transport plan $\gamma_{\tau} \in \Gamma_o(\mu_{\tau},\nu_{\tau})$. Therefore, by combining the estimates of \eqref{eq:ViabCompact41} and \eqref{eq:ViabCompact42}, dividing the resulting expression by $h>0$ and then letting $h \to 0^+$ while recalling that $\tfrac{1}{p} g^p(\cdot)$ is differentiable at $ \tau \in \Dcal$ and that $\epsilon > 0$ is arbitrary, we obtain the differential inequality
\begin{equation}
\label{eq:ViabCompact5}
g^{p-1}(\tau) \dot g(\tau) \leq \INTDom{\big\langle v_{\tau}(x) , j_p(x-y) \big\rangle}{\R^{2d}}{\gamma_{\tau}(x,y)}
\end{equation}
which holds for every time $\tau \in (t,t+\delta) \cap \Tcal \cap \Dcal$.

Our goal in what follows is to use differential estimate derived in \eqref{eq:ViabCompact5} to show that $g(\cdot)$ vanishes identically on $(t,t+\delta)$. By inserting crossed terms in the latter expression, one can easily check that
\begin{equation}
\label{eq:ViabCompact61}
\begin{aligned}
g^{p-1}(\tau) \dot g(\tau) & \leq \INTDom{\big\langle v_{\tau}(x) - v_{\tau}(y) , j_p(x-y) \big\rangle}{\R^{2d}}{\gamma_{\tau}(x,y)} \\
& \hspace{0.5cm} + \INTDom{\big\langle v_{\tau}(y) - \xi_{\tau}(y) , j_p(x-y) \big\rangle}{\R^{2d}}{\gamma_{\tau}(x,y)} \\
& \hspace{0.5cm} + \INTDom{\big\langle \xi_{\tau}(y) , j_p(x-y) \big\rangle}{\R^{2d}}{\gamma_{\tau}(x,y)} \\
& \leq l(\tau) g^p(\tau) +  \INTDom{\big\langle v_{\tau}(y) - \xi_{\tau}(y) , j_p(x-y) \big\rangle}{\R^{2d}}{\gamma_{\tau}(x,y)}
\end{aligned}
\end{equation}
where we resorted to \eqref{eq:ViabCompact3} and Hypothesis \ref{hyp:CI}-$(iii)$, as well as to the elementary observation that
\begin{equation*}
\INTDom{|x-y| |j_p(x-y)|}{\R^{2d}}{\gamma_{\tau}(x,y)} = \INTDom{|x-y|^p}{\R^{2d}}{\gamma_{\tau}(x,y)} = g^p(\tau)
\end{equation*}
which follows from \eqref{eq:ViabCompact0} together with the fact that $\gamma_{\tau} \in \Gamma_o(\mu_{\tau},\nu_{\tau})$. Recall now that as a consequence of Hypothesis \ref{hyp:CI}-$(iv)$, there exists for each $w_{\tau} \in V(\tau,\nu_{\tau})$ an element $v_{\tau} \in V(\tau,\mu_{\tau})$ such that 
\begin{equation*}
\dsf_{\sup}(v_{\tau},w_{\tau}) \leq L(\tau) W_p(\mu_{\tau},\nu_{\tau}).
\end{equation*}
This, together with the fact that the estimates in \eqref{eq:ViabCompact61} hold for every $v_{\tau} \in V(\tau,\mu_{\tau})$, further yields
\begin{equation}
\label{eq:ViabCompact62}
g^{p-1}(\tau) \dot g(\tau) \leq \Big( l(\tau) + L(\tau) \Big) g^p(\tau) + \INTDom{\big\langle w_{\tau}(y) - \xi_{\tau}(y) , j_p(x-y) \big\rangle}{\R^{2d}}{\gamma_{\tau}(x,y)}
\end{equation}
for all times $\tau \in (t,t+\delta) \cap \Tcal \cap \Dcal$, each $\xi_{\tau} \in T_{\Qpazo}(\nu_{\tau})$ and every $w_{\tau} \in V(\tau,\nu_{\tau})$. Noting that the right-hand side of the previous expression is linear and strongly continuous with respect to $\xi_{\tau} \in \Lpazo^p(\R^d,\R^d;\nu_{\tau})$, one gets that \eqref{eq:ViabCompact62} also holds for all $\xi_{\tau} \in \co T_{\Qpazo}(\nu_{\tau})$. Thus, by choosing 
\begin{equation*}
\xi_{\tau} = w_{\tau} \in V(\tau,\nu_{\tau}) \cap \co T_{\Qpazo}(\nu_{\tau}),
\end{equation*}
where the intersection is nonempty as a consequence of our standing assumption \eqref{eq:ViabCompactHyp}, one finally has
\begin{equation*}
\dot g(\tau) \leq \Big( l(\tau) + L(\tau) \Big) g(\tau)
\end{equation*}
for $\Lcal^1$-almost every $\tau \in (t,t+\delta)$. As we assumed that $g(t) = 0$, a direct application of Gr\"onwall's lemma yields that $g(\tau) = 0$ for all times $\tau \in (t,t+\delta)$, thus leading to a contradiction.


\paragraph*{Step 2 -- Existence of a viable curve.}

In the first step of the proof, we have shown without loss of generality that 
\begin{equation}
\label{eq:DistZero}
\dist_{\Pcal_p(\R^d)} \Big( \Rpazo_{(\tau,t)}(\mu_{\tau}) \, ; \Qpazo_R \Big) = 0,
\end{equation}
which equivalently means that $\Rpazo_{(\tau,t)}(\mu_{\tau}) \cap \Qpazo \neq \emptyset$ for all times  $0 \leq \tau \leq t \leq T-\tau$ and each $\mu_{\tau} \in \Rpazo_{(0,\tau)}(\mu^0) \cap \Qpazo_R$. Given an integer $n \geq 1$, consider the following dyadic subdivision $[0,T] := \cup_{k=0}^{2^n-1}[ t_k,t_{k+1}]$ of the time interval, wherein $t_k := T k/2^n$ for $k \in \{0,\dots,2^n-1\}$. By inductively leveraging \eqref{eq:DistZero} along with the semigroup property \eqref{eq:SemigroupReach} of the reachable sets, we can build for each $n \geq 1$ a curve $\mu_n(\cdot) \in \Spazo_{[0,T]}(0,\mu^0)$ that is such that 
\begin{equation}
\label{eq:InclusionCompact}
\mu_n(t_k) \in \Qpazo
\end{equation}
for each $k \in \{ 0,\dots,2^n-1 \}$. At this stage recall that, as a consequence of Proposition \ref{prop:Topological}, the solution set $\Spazo_{[0,T]}(0,\mu^0)$ is compact for the topology of uniform convergence, so that 
\begin{equation*}
\sup_{t \in [0,T]} W_p(\mu_{n_j}(t),\mu(t)) ~\underset{n_j \to +\infty}{\longrightarrow}~ 0
\end{equation*}
for some $\mu(\cdot) \in \Spazo_{[0,T]}(\mu^0)$ and along a subsequence $(\mu_{n_j}(\cdot)) \subset \AC([0,T],\Pcal_p(\R^d))$. In particular, it then follows from \eqref{eq:InclusionCompact} that 
\begin{equation*}
\mu \big( \tfrac{k T}{2^m} \big) \in \Qpazo, 
\end{equation*}
for every integer $m \geq 1$ and each $k \in \{0,\dots,2^m-1\}$. From there, we conclude by a classical density argument that $\mu(t) \in \Qpazo$ for all times $t \in [0,T]$.
\end{proof} 

To complement the sufficient viability conditions stated in Theorem \ref{thm:ViabStationary}, we derive below necessary viability conditions which also involve the contingent cone to the constraints. 

\begin{thm}[Necessary viability conditions for stationary constraints]
\label{thm:ViabStationaryNesc}
Suppose that $p \in (1,+\infty)$, let $V : [0,T] \times \Pcal_p(\R^d) \tto C^0(\R^d,\R^d)$ be a set-valued map with convex images satisfying Hypotheses \ref{hyp:CI} and $\Qpazo \subset \Pcal_p(\R^d)$ be a proper set. Then if $\Qpazo$ is viable for \eqref{eq:CauchyViab}, it necessarily holds that
\begin{equation*}
V(t,\nu) \cap T_{\Qpazo}(\nu) \neq \emptyset
\end{equation*}
for $\Lcal^1$-almost every $t \in [0,T]$ and each $\nu \in \Qpazo$. 
\end{thm}

\begin{proof}
This result is a particular case of Theorem \ref{thm:NecessaryConds} whose proof is detailed in Section \ref{subsection:ViabCauchyAbs} below. 
\end{proof}

Lastly, we end this section by providing necessary and sufficient conditions for the invariance of a stationary constraints set, based on a geometric condition that is stronger than that of Theorem \ref{thm:ViabStationary}.

\begin{thm}[Invariance conditions for stationary constraint sets]
\label{thm:Invariance}
Under the assumptions of Theorem \ref{thm:ViabStationary}, the set $\Qpazo \subset \Pcal_p(\R^d)$ is invariant for \eqref{eq:CauchyViab} \textnormal{if and only if}
\begin{equation*}
V(t,\nu) \subset \co T_{\Qpazo}(\nu) 
\end{equation*}
for $\Lcal^1$-almost every $t \in [0,T]$ and all $\nu \in \Qpazo$. \end{thm}

\begin{proof}
This result is a particular case of Theorems \ref{thm:NecessaryConds} and Theorem \ref{thm:ViabTubeACL}, whose proofs are discussed in depth in Section \ref{subsection:ViabCauchyAbs} below. 
\end{proof} 


\subsection{The case of time-dependent constraints}
\label{subsection:ViabCauchyAbs}

As mentioned hereinabove, the viability and invariance results exposed in Theorem \ref{thm:ViabStationary}, Theorem \ref{thm:ViabStationaryNesc} and Theorem \ref{thm:Invariance} can be generalised to time-dependent constraint tubes $\Qpazo : [0,T] \tto \Pcal_p(\R^d)$. In this setting, the relevant geometric objects on which these statements are based are the contingent cones to the graph of the constraints, which are defined for each $(\tau,\mu) \in \Graph(\Qpazo)$ as
\begin{equation*}
\begin{aligned}
T_{\Graph(\Qpazo)}(\tau,\mu) := \bigg\{ (\zeta,\xi) \in & \; \R \times \Lpazo^p(\R^d,\R^d;\mu) ~\, \textnormal{s.t.}~ \\
& \liminf_{h \to 0^+} \tfrac{1}{h} \dist_{[0,T] \times \Pcal_p(\R^d)} \Big( \big(\tau + h \zeta , (\Id + h \xi)_{\sharp} \mu \big) \, ; \Graph(\Qpazo) \Big) = 0 \bigg\}.
\end{aligned}
\end{equation*}
By a simple adaptation of \cite[Proposition 5.1.4]{Aubin1990} following \cite[Sections 2.3 and 2.4]{Badreddine2022}, this set can be equivalently characterised as 
\begin{equation}
\label{eq:ContingentGraphCharac}
\begin{aligned}
T_{\Graph(\Qpazo)}(\tau,\mu) = \bigg\{ (\zeta,\xi) & \in \R \times \Lpazo^p(\R^d,\R^d;\mu) ~\, \textnormal{s.t.}~ \text{there exist sequences $h_i  \to 0^+$ and $\zeta_i  \to \zeta$ for} \\
& \hspace{3.6cm} \text{which} ~ \dist_{\Pcal_p(\R^d)} \Big( ( \Id + h_i \xi)_{\sharp} \mu \; ; \Qpazo(\tau + h_i \zeta_i) \Big) = o_{\tau}(h_i) \bigg\}.
\end{aligned}
\end{equation}
In the following theorem, we begin our investigation by a discussion centered around necessary viability and invariance conditions, as these latter do not depend on the regularity of the constraints tubes.

\begin{thm}[Necessary viability, invariance and regularity conditions for constraints tubes]
\label{thm:NecessaryConds}
Suppose that $p \in (1,+\infty)$, let $V : [0,T] \times \Pcal_p(\R^d) \tto C^0(\R^d,\R^d)$ be a set-valued map with convex images satisfying Hypotheses \ref{hyp:CI}, and $\Qpazo : [0,T] \tto \Pcal_p(\R^d)$ be a constraints tube with proper images. 

Then if $\Qpazo : [0,T] \tto \Pcal_p(\R^d)$ is viable for \eqref{eq:CauchyViab}, it must be left absolutely continuous and satisfy
\begin{equation}
\label{eq:ViabCompactNesc}
\big( \{1\} \times V(t,\nu) \big) \cap T_{\Graph(\Qpazo)}(t,\nu) \neq \emptyset
\end{equation}
for $\Lcal^1$-almost every $t \in [0,T]$ and each $\nu \in \Qpazo(t)$. Analogously if $\Qpazo : [0,T] \tto \Pcal_p(\R^d)$ is invariant for \eqref{eq:CauchyViab}, it necessarily holds that 
\begin{equation}
\label{eq:InvCompactNesc}
\big( \{1\} \times V(t,\nu) \big) \subset T_{\Graph(\Qpazo)}(t,\nu)
\end{equation}
for $\Lcal^1$-almost every $t \in [0,T]$ and each $\nu \in \Qpazo(t)$. 
\end{thm}

\begin{proof}
Let us start by showing the necessity of \eqref{eq:ViabCompactNesc} when $\Qpazo : [0,T] \tto \Pcal_p(\R^d)$ is viable for \eqref{eq:CauchyViab}. Let $\Tcal \subset (0,T)$ be the set of full $\Lcal^1$-measure over which the statements of Theorem \ref{thm:Liminf} and Theorem \ref{thm:Limsup} as well as Hypotheses \ref{hyp:CI}-$(ii)$, $(iii)$ and $(iv)$ hold. Fix some $\tau \in \Tcal$, an element $\mu_{\tau} \in \Qpazo(\tau)$, a sequence $h_i \to 0^+$ and a viable curve $\mu(\cdot) \in \Spazo_{[\tau,T]}(\tau,\mu_{\tau})$. By Theorem \ref{thm:Limsup}, there exists for each $\epsilon > 0$ a velocity $v_{\tau}^{\epsilon} \in V(\tau,\mu_{\tau})$ such that 
\begin{equation*}
W_p \Big( \mu(\tau + h_{i_k}^{\epsilon}) , (\Id + h_{i_k}^{\epsilon} v_{\tau}^{\epsilon})_{\sharp} \mu_{\tau} \Big) \leq \epsilon h_{i_k}^{\epsilon} 
\end{equation*}
along an adequate subsequence $h_{i_k}^{\epsilon} \to 0^+$. Observe that by Theorem \ref{thm:Ascoli} and our choice of $\tau \in \Tcal$, the set $V(\tau,\mu_{\tau}) \subset C^0(\R^d,\R^d)$ is compact for the topology induced by $\dsf_{cc}(\cdot,\cdot)$. In particular for each sequence $\epsilon_n \to 0^+$, there exists a subsequence that we do not relabel and some $v_{\tau} \in V(\tau,\mu_{\tau})$ for which 
\begin{equation*}
\NormLp{v_{\tau} - v_{\tau}^{\epsilon_n}}{p}{\R^d,\R^d ; \, \mu_{\tau}} ~\underset{\epsilon_n \to 0^+}{\longrightarrow}~ 0,
\end{equation*}
where we used Lemma \ref{lem:ConvdccWass}. Note also that for every $\epsilon_n >0$, one can choose $\delta_n := h_{i_n}^{\epsilon_n}$ in such a way that $o_{\tau,\epsilon_n}(h_{i_n}^{\epsilon_n}) \leq \epsilon_n h_{i_n}^{\epsilon_n}$. Thus, recalling that $\mu(t) \in \Qpazo(t)$ for all times $t \in [\tau,T]$, one further has that
\begin{equation*}
\begin{aligned}
& \dist_{\Pcal_p(\R^d)} \Big( (\Id + \delta_n v_{\tau})_{\sharp} \mu_{\tau} \, ; \Qpazo(\tau + \delta_n) \Big) \\
& \hspace{0.6cm} \leq \dist_{\Pcal_p(\R^d)} \Big( (\Id + \delta_n v_{\tau}^{\epsilon_n})_{\sharp} \mu_{\tau} \, ; \Qpazo(\tau+\delta_n) \Big) + W_p \Big( (\Id + \delta_n v_{\tau}^{\epsilon_n})_{\sharp} \mu_{\tau} , (\Id + \delta_n v_{\tau})_{\sharp} \mu_{\tau} \Big) \\
& \hspace{0.6cm} \leq W_p \Big( \mu(\tau + \delta_n) , (\Id + \delta_n v_{\tau}^{\epsilon_n})_{\sharp} \mu_{\tau} \Big) + \delta_n \NormLp{v_{\tau} - v_{\tau}^{\epsilon_n}}{p}{\R^d,\R^d ; \, \mu_{\tau}} \\
& \hspace{0.6cm} \leq \delta_n \Big( \epsilon_n + \NormLp{v_{\tau} - v_{\tau}^{\epsilon_n}}{p}{\R^d,\R^d ; \, \mu_{\tau}} \hspace{-0.05cm} \Big),
\end{aligned}
\end{equation*}
which in turn implies 
\begin{equation*}
\liminf_{\delta_n \to 0^+}\tfrac{1}{\delta_n} \dist_{\Pcal_p(\R^d)} \Big( (\Id + \delta_n v_{\tau})_{\sharp} \mu_{\tau} \, ; \Qpazo(\tau+\delta_n) \Big) = 0.
\end{equation*}
By \eqref{eq:ContingentGraphCharac}, this is tantamount to the fact that $(1,v_{\tau}) \in T_{\Graph(\Qpazo)}(\tau,\mu_{\tau})$, and thus yields \eqref{eq:ViabCompactNesc}.

Suppose now that $\Qpazo : [0,T] \tto \Pcal_p(\R^d)$ is invariant for \eqref{eq:CauchyViab}, fix an arbitrary $v_{\tau} \in V(\tau,\mu_{\tau})$, and observe that by Theorem \ref{thm:Liminf}, there exists for every $\epsilon > 0$ some $h_{\epsilon} > 0$ and a curve $\mu_{\epsilon}(\cdot) \in \Spazo_{[\tau,T]}(\tau,\mu_{\tau})$ such that 
\begin{equation*}
W_p \Big( \mu_{\epsilon}(\tau + h) , (\Id + h v_{\tau})_{\sharp} \mu_{\tau} \Big) \leq \epsilon h 
\end{equation*}
for all $h \in [0,h_{\epsilon}]$. This estimate, combined with the fact that $\mu_{\epsilon}(t) \in \Qpazo(t)$ for all times $t \in [\tau,T]$ owing to the invariance of the tube, implies that
\begin{equation*}
\begin{aligned}
\dist_{\Pcal_p(\R^d)} \Big( (\Id + h v_{\tau})_{\sharp} \mu_{\tau} \, ; \Qpazo(\tau+h) \Big) & \leq W_p \Big( \mu_{\epsilon}(\tau + h) , (\Id + h v_{\tau})_{\sharp} \mu_{\tau} \Big) \\
& \leq \epsilon h.
\end{aligned}
\end{equation*}
Thus, dividing by $h >0$ and letting $h \to 0^+$ while recalling that $\epsilon > 0$ is arbitrary, we finally obtain
\begin{equation*}
\liminf_{h \to 0^+} \tfrac{1}{h} \dist_{\Pcal_p(\R^d)} \Big( (\Id + h v_{\tau})_{\sharp} \mu_{\tau} \, ; \Qpazo(\tau+h) \Big) =0, 
\end{equation*}
which equivalently means that $(1,v_{\tau}) \in T_{\Graph(\Qpazo)}(\tau,\mu_{\tau})$ for all $v_{\tau} \in V(\tau,\mu_{\tau})$, and thus yields \eqref{eq:InvCompactNesc}. 

Let us finally prove that if $\Qpazo : [0,T] \tto \Pcal_p(\R^d)$ is viable for \eqref{eq:CauchyViab}, then it is left absolutely continuous. To do so, fix $\mu \in \Pcal_p(\R^d)$ and $R > 0$ in such a way that
\begin{equation*}
\Qpazo(\tau) \cap \B_{\Pcal_p(\R^d)}(\mu,R) \neq \emptyset.
\end{equation*}
Then, since $\Qpazo : [0,T] \tto \Pcal_p(\R^d)$ is viable, there exists for all $\tau \in [0,T]$ and each $\mu_{\tau} \in \Qpazo(\tau) \cap \B_{\Pcal_p(\R^d)}(\mu,R)$ a curve $\mu(\cdot) \in \Spazo_{[\tau,T]}(\tau,\mu_{\tau})$ such that $\mu(t) \in \Qpazo(t)$ for all times $t \in [\tau,T]$. Besides by Proposition \ref{prop:MomentumCI}, there exists a constant $\Cpazo_{\mu,R} > 0$ depending only on the magnitudes of $p,\Mpazo_p(\mu),R$ and $\Norm{m(\cdot)}_1$ such that
\begin{equation*}
W_p(\mu_{\tau},\mu(t)) \leq (1+\Cpazo_{\mu,R}) \INTSeg{m(s)}{s}{\tau}{t}
\end{equation*}
for all times $t \in [\tau,T]$. Thus, noting by construction that for any $0 \leq \tau \leq t \leq T$, it holds that
\begin{equation*}
\begin{aligned}
\Delta_{\mu,R} (\Qpazo(\tau),\Qpazo(t)) & \leq \sup \bigg\{ \dist_{\Pcal_p(\R^d)} \big( \mu_{\tau} \, ; \Qpazo(t) \big) ~\textnormal{s.t.}~ \mu_{\tau} \in \Qpazo(\tau) \cap \B_{\Pcal_p(\R^d)}(\mu,R) \bigg\} \\
& \leq \sup \bigg\{ W_p(\mu_{\tau},\mu(t)) ~\, \textnormal{s.t.}~ \mu_{\tau} \in \Qpazo(\tau) \cap \B_{\Pcal_p(\R^d)}(\mu,R) ~\text{and}~ \mu(\cdot) \in \Spazo_{[\tau,T]}(\tau,\mu_{\tau}) \\
& \hspace{9.65cm} \text{satisfies}~ \mu(t) \in \Qpazo(t) \bigg\} \\
& \leq (1+\Cpazo_{\mu,R}) \INTSeg{m(s)}{s}{\tau}{t}, 
\end{aligned}
\end{equation*}
we can conclude that $\Qpazo : [0,T] \tto \Pcal_p(\R^d)$ is left absolutely continuous. 
\end{proof}

\begin{rmk}[On the role of left absolute continuity]
It is worth noting that in the previous theorem, we have shown that being viable for \eqref{eq:CauchyViab} under Hypotheses \ref{hyp:CI} entails the left absolute continuity of the constraint tube. This supports the fact that this regularity framework -- for which we provide sufficient viability conditions in Theorem \ref{thm:ViabTubeACL} -- appears quite naturally when studying Cauchy-Lipschitz continuity inclusions with state-constraints. 
\end{rmk}

In the next theorem, we provide sufficient viability conditions for absolutely continuous constraint tubes, which are the natural generalisation of Theorem \ref{thm:ViabStationary} to the time-dependent setting.  

\begin{thm}[Sufficient viability conditions for absolutely continuous constraints tubes]
\label{thm:ViabTubeAC}
Suppose that $p \in (1,+\infty)$, let $V : [0,T] \times \Pcal_p(\R^d) \tto C^0(\R^d,\R^d)$ be a set-valued map with convex images satisfying Hypotheses \ref{hyp:CI} and $\Qpazo : [0,T] \tto \Pcal_p(\R^d)$ be an absolutely continuous tube with proper images such that 
\begin{equation}
\label{eq:ViabCompactHypBis}
( \{1\} \times V(t,\nu)) \cap \co T_{\Graph(\Qpazo)}(t,\nu) \neq \emptyset
\end{equation}
for $\Lcal^1$-almost every $t \in [0,T]$ and all $\nu \in \Qpazo(t)$. Then $\Qpazo : [0,T] \tto \Pcal_p(\R^d)$ is viable for \eqref{eq:CauchyViab}.
\end{thm}

\begin{proof}
As in the proof of Theorem \ref{thm:ViabStationary}, we assume without loss of generality that $(\tau,\mu_{\tau}) = (0,\mu^0)$ for some $\mu^0 \in \Qpazo(0)$. From there on, the arguments will essentially follow along the same line as those of Theorem \ref{thm:ViabStationary}, in which one aims at showing that the distance function, defined by  
\begin{equation*}
g(t) := \dist_{\Pcal_p(\R^d)} \Big( \Rpazo_{(0,t)}(\mu^0) \, ; \Qpazo(t) \Big)
\end{equation*}
for all times $t \in [0,T]$, is identically equal to $0$.  

In what follows, we let $\Tcal \subset (0,T)$ be the set of full $\Lcal^1$-measure on which the statements of Theorem \ref{thm:Liminf} and Theorem \ref{thm:Limsup} as well as  Hypotheses \ref{hyp:CI}-$(ii)$, $(iii)$ and $(iv)$ hold. Observing that $g(\cdot) \in \AC([0,T],\R_+)$ by Proposition \ref{prop:AC} and Lemma \ref{lem:ACReach} , we also denote by $\Dcal \subset (0,T)$ the subset of full $\Lcal^1$-measure where it is differentiable. Moreover, owing to the absolute continuity of $\Qpazo : [0,T] \tto \Pcal_p(\R^d)$ and to the fact that $\cup_{t \in [0,T]} \Rpazo_{(0,t)}(\mu^0)$ is compact by Proposition \ref{prop:Topological} and Lemma \ref{lem:ACReach}, one can find some radius $R>0$ satisfying
\begin{equation}
\label{eq:RDefTube}
\dist_{\Pcal_p(\R^d)} \Big( \Rpazo_{(0,t)}(\mu^0) \, ; \partial \B_{\Pcal_p(\R^d)}(\mu^0,R) \Big) \geq \dist_{\Pcal_p(\R^d)} \Big( \Rpazo_{(0,t)}(\mu^0) \, ; \Qpazo(t) \Big) + 1
\end{equation}
for all times $t \in [0,T]$, by following e.g. the arguments detailed in Appendix \ref{section:AppendixAC} below. Note that by construction, the sets $\Qpazo_R(t) := \Qpazo(t) \cap \B_{\Pcal_p(\R^d)}(\mu^0,R)$ are nonempty and such that
\begin{equation*}
g(t) = \dist_{\Pcal_p(\R^d)} \Big( \Rpazo_{(0,t)}(\mu^0) \, ; \Qpazo_R(t) \Big)
\end{equation*}
for all times $t \in [0,T]$.


\paragraph*{Step 1 -- Local variations of the distance along contingent directions.}

Suppose by contradiction that there exist some $t \in [0,T)$ along with a small $\delta > 0$ such that $g(t) = 0$ and $g(\tau) > 0$ for all times $\tau \in (t,t+\delta)$, and fix an arbitrary element $\tau \in (t,t+\delta) \cap \Tcal \cap \Dcal$. Since $\Rpazo_{(0,\tau)}(\mu^0)$ and $\Qpazo_R(\tau)$ are both compact because $\Qpazo(\tau)$ is proper, one has that
\begin{equation*}
g(\tau) = W_p(\mu_{\tau},\nu_{\tau})
\end{equation*}
for some $\mu_{\tau} \in \Rpazo_{(0,\tau)}(\mu^0)$ and $\nu_{\tau} \in \Qpazo_R(\tau)$. By \eqref{eq:ContingentGraphCharac} along with the estimate \eqref{eq:RDefTube} imposed on $R>0$, there exists for every $(\zeta_{\tau},\xi_{\tau}) \in T_{\Graph(\Qpazo)}(\tau,\nu_{\tau})$ two sequences $h_i \to 0^+$ and $\zeta_{\tau}^{\,i} \to \zeta_{\tau}$ for which
\begin{equation*}
\dist_{\Pcal_p(\R^d)} \Big( (\Id + h_i \xi_{\tau})_{\sharp} \nu_{\tau} \, ; \Qpazo_R(\tau + h_i \zeta_{\tau}^{\,i}) \Big) = o_{\tau}(h_i).
\end{equation*}
This allows us to estimate from above the variation of $g^p(\cdot)$ around $\tau$ as 
\begin{equation}
\label{eq:ViabTube1}
\begin{aligned}
\tfrac{1}{p} g^p(\tau + h_i \zeta_{\tau}^{\,i}) - \tfrac{1}{p} g^p(\tau) & = \tfrac{1}{p} \dist_{\Pcal_p(\R^d)}^p \Big( \Rpazo_{(0,\tau+h_i \zeta_{\tau}^{\,i})}(\mu^0) \, ; \Qpazo_R(\tau + h_i \zeta_{\tau}^{\,i}) \Big) - \tfrac{1}{p} W_p^p(\mu_{\tau},\nu_{\tau}) \\
& \leq \tfrac{1}{p} \dist_{\Pcal_p(\R^d)}^p \Big( \Rpazo_{(0,\tau+h_i \zeta_{\tau}^{\,i})}(\mu^0) \, ; (\Id + h_i \xi_{\tau})_{\sharp} \nu_{\tau} \Big) - \tfrac{1}{p} W_p^p(\mu_{\tau},\nu_{\tau}) + o_{\tau}(h_i)
\end{aligned}
\end{equation}
for any $h_i > 0$ that is sufficiently small. 

In order to extract further information from \eqref{eq:ViabTube1}, we need to discriminate between two possible scenarios depending on the asymptotic behaviour of the sequence $(\zeta_{\tau}^i) \subset \R$. If there exists a subsequence $i_k \to +\infty$ for which $\zeta_{\tau}^{\, i_k} \geq 0$, it follows from Theorem \ref{thm:Liminf} that for every $\epsilon > 0$ and any $v_{\tau} \in V(\tau,\mu_{\tau})$, there exists a curve $\mu_{\epsilon}(\cdot) \in \Spazo_{[\tau,T]}(\tau,\mu_{\tau})$ such that
\begin{equation}
\label{eq:ViabTube21}
W_p \Big( \mu_{\epsilon}(\tau + h_{i_k} \zeta_{\tau}^{\, i_k}) , (\Id + h_{i_k} \zeta_{\tau}^{\, i_k} v_{\tau})_{\sharp} \mu_{\tau} \Big) \leq \epsilon h_{i_k} \zeta_{\tau}^{\, i_k}.
\end{equation}
On the other hand, if $\zeta_{\tau}^i < 0$ for all large $i \geq 1$, we can apply Theorem \ref{thm:Limsup} to obtain for each curve $\mu(\cdot) \in \Spazo_{[0,T]}(\tau,\mu_{\tau})$ the existence of an element $v_{\tau}^{\epsilon} \in V(\tau,\mu_{\tau})$ and of a subsequence $i_k \to +\infty$, both depending on $\epsilon > 0$, for which
\begin{equation}
\label{eq:ViabTube22}
W_p \Big( \mu(\tau+h_{i_k} \zeta_{\tau}^{\, i_k}) , \big( \Id + h_{i_k} \zeta_{\tau}^{\, i_k} v_{\tau}^{\epsilon} \big)_{\sharp} \mu_{\tau} \Big) \leq \epsilon h_{i_k} |\zeta_{\tau}^{\, i_k}|.
\end{equation}
Thus by combining \eqref{eq:ViabTube21} and \eqref{eq:ViabTube22}, one may assert that there exist curves $\mu_{\epsilon}(\cdot) \in \Spazo_{[0,T]}(\tau,\mu_{\tau})$ along with two subsequences $h_{i_k} \to 0^+$ and $\zeta_{\tau}^{\, i_k} \to \zeta_{\tau}$, all possibly depending on $\epsilon > 0$, such that 
\begin{equation}
\label{eq:ViabTubeVtau}
W_p \Big( \mu_{\epsilon}(\tau+h_{i_k} \zeta_{\tau}^{\, i_k}) , (\Id + h_{i_k} \zeta_{\tau}^{\, i_k} v_{\tau}^{\epsilon})_{\sharp} \mu_{\tau} \Big) \leq \epsilon h_{i_k} |\zeta_{\tau}^{\, i_k}| \quad \text{for} \quad
\left\{
\begin{aligned}
& \text{every $v_{\tau}^{\epsilon} \in V(\tau,\mu_{\tau})$} && \text{if $\zeta_{\tau}^{i_k} \geq 0$,} \\
& \text{some $v_{\tau}^{\epsilon} \in V(\tau,\mu_{\tau})$} && \text{if $\zeta_{\tau}^{i_k} < 0$,}
\end{aligned}
\right.
\end{equation}
when $i_k \geq 1$ is large enough. Furthermore, observing that
\begin{equation*}
\mu_{\epsilon}(\tau+h_{i_k} \zeta_{\tau}^{\, i_k}) \in \Rpazo_{(0,\tau + h_{i_k} \zeta_{\tau}^{\, i_k})}(\mu^0), 
\end{equation*}
one may refine the estimate of \eqref{eq:ViabTube1} using the curves $\mu_{\epsilon}(\cdot)$ as follows
\begin{equation}
\label{eq:ViabTubeVtau1}
\begin{aligned}
\tfrac{1}{p} g^p(\tau + h_{i_k} \zeta_{\tau}^{\, i_k}) - \tfrac{1}{p} g^p(\tau) & \leq \tfrac{1}{p} W_p^p \Big( \mu_{\epsilon}(\tau+h_{i_k} \zeta_{\tau}^{\, i_k}) , (\Id + h_{i_k} \xi_{\tau})_{\sharp} \nu_{\tau} \Big)  - \tfrac{1}{p} W_p^p(\mu_{\tau},\nu_{\tau}) + o_{\tau}(h_{i_k}).
\end{aligned}
\end{equation}
At this stage, upon noting that $v_{\tau} \in \Lpazo^p(\R^d,\R^d;\mu_{\tau})$ by Hypothesis \ref{hyp:CI}-$(ii)$ while $\xi_{\tau} \in \Lpazo^p(\R^d,\R^d;\nu_{\tau})$ by definition, and assuming without loss of generality that $\epsilon \in (0,1]$ and $h_{i_k} \in (0,1]$, one may reproduce the computations of \eqref{eq:ViabCompact4Bis} in the proof of Theorem \ref{thm:ViabStationary} and combine them with \eqref{eq:ViabTubeVtau} to obtain that
\begin{equation}
\label{eq:ViabTubeVtau2}
\begin{aligned}
& \tfrac{1}{p} W_p^p \Big( \mu_{\epsilon}(\tau+h_{i_k} \zeta_{\tau}^{\, i_k}) , (\Id + h_{i_k} \xi_{\tau})_{\sharp} \nu_{\tau} \Big) \\
& \hspace{2.05cm} - \tfrac{1}{p} W_p^p \Big( (\Id + h_{i_k} \zeta_{\tau}^{\, i_k} v_{\tau}^{\epsilon})_{\sharp} \mu_{\tau} , (\Id + h_{i_k} \xi_{\tau})_{\sharp} \nu_{\tau} \Big) \leq \epsilon h_{i_k} + o_{\tau,\epsilon}(h_{i_k})
\end{aligned}
\end{equation}
for $h_{i_k} > 0$ small enough, and up to rescaling $\epsilon > 0$ by a constant since the sequence $(\zeta_{\tau}^{\, i_k}) \subset \R$ is bounded. Whence, by merging \eqref{eq:ViabTubeVtau1} and \eqref{eq:ViabTubeVtau2}, it then holds that
\begin{equation}
\label{eq:ViabTube3}
\begin{aligned}
\tfrac{1}{p} g^p(\tau + h_{i_k} \zeta_{\tau}^{\, i_k}) - \tfrac{1}{p} g^p(\tau) & \leq \tfrac{1}{p} W_p^p \Big( (\Id + h_{i_k} \zeta_{\tau}^{\, i_k} v_{\tau}^{\epsilon} )_{\sharp} \mu_{\tau}, (\Id + h_{i_k} \xi_{\tau})_{\sharp} \nu_{\tau} \Big) \\
& \hspace{2.075cm} - \tfrac{1}{p} W_p^p(\mu_{\tau},\nu_{\tau}) + \epsilon h_{i_k} + o_{\tau,\epsilon}(h_{i_k}), 
\end{aligned}
\end{equation}
for $i_k \geq 1$ sufficiently large, with $v_{\tau}^{\epsilon} \in V(\tau,\mu_{\tau})$ being either fixed or arbitrary depending on the asymptotic behaviour of $(\zeta_{\tau}^{\, i_k}) \subset \R$. One may then apply the joint superdifferentiability inequality of Proposition \ref{prop:SuperdiffWass} to obtain
\begin{equation}
\label{eq:ViabTube41}
\begin{aligned}
\tfrac{1}{p} W_p^p \Big( (\Id + h_{i_k} \zeta_{\tau}^{\, i_k} v_{\tau}^{\epsilon})_{\sharp} \mu_{\tau}, (\Id + & h_{i_k} \xi_{\tau})_{\sharp} \nu_{\tau} \Big) - \tfrac{1}{p} W_p^p(\mu_{\tau}, \nu_{\tau}) \\
& \leq h_{i_k} \INTDom{\hspace{-0.05cm} \big\langle \zeta_{\tau}^{\, i_k} v_{\tau}^{\epsilon}(x) - \xi_{\tau}(y), j_p(x-y) \big\rangle}{\R^{2d}}{\gamma_{\tau}(x,y)} + o_{\tau,\epsilon}(h_{i_k}),
\end{aligned}
\end{equation}
for each $\gamma_{\tau} \in \Gamma_o(\mu_{\tau},\nu_{\tau})$, where we used the analytical expressions \eqref{eq:Remainder1}-\eqref{eq:Remainder2} of the remainder term, along with the fact that $(\zeta_{\tau}^{\, i_k}) \subset \R$ is bounded. In turn, by combining \eqref{eq:ViabTube3} and \eqref{eq:ViabTube41}, letting $i_k \to +\infty$ and recalling that $\tfrac{1}{p} g^p(\cdot)$ is differentiable at $\tau \in \Dcal$, one finally gets
\begin{equation}
\label{eq:ViabTubeDiffEst}
\zeta_{\tau} \hspace{0.05cm} g^{p-1}(\tau) \dot g(\tau) \leq \INTDom{\big\langle \zeta_{\tau} v_{\tau}^{\epsilon}(x) - \xi_{\tau}(y) , j_p(x-y) \big\rangle}{\R^{2d}}{\gamma_{\tau}(x,y)} \, + \, \epsilon \quad \text{for} \quad 
\left\{
\begin{aligned}
& \text{every $v_{\tau}^{\epsilon} \in V(\tau,\mu_{\tau})$} ~~ & \text{if $\zeta_{\tau} \geq 0$}, \\
& \text{some $v_{\tau}^{\epsilon} \in V(\tau,\mu_{\tau})$} ~~ & \text{if $\zeta_{\tau} < 0$,}
\end{aligned}
\right.
\end{equation}
where $(\zeta_{\tau},\xi_{\tau}) \in T_{\Graph(\Qpazo)}(\tau,\nu_{\tau})$ and $\gamma_{\tau} \in \Gamma_o(\mu_{\tau},\nu_{\tau})$ are arbitrary while $v_{\tau}^{\epsilon} \in V(\tau,\mu_{\tau})$ may possibly depend on $\zeta_{\tau} \in \R$ as well as on the free parameter $\epsilon >0$. 


\paragraph*{Step 2 -- Convexification of the contingent directions and viability.}

In this second step, we show how one can convexify the contingent directions in \eqref{eq:ViabTubeDiffEst} and then prove the existence of viable curves. With this goal in mind, we draw inspiration from \cite[Lemma 4.9]{Frankowska1995} and consider arbitrary collections of $N \geq 1$ elements 
\begin{equation*}
(\zeta_{\tau}^{\,j},\xi_{\tau}^{\,j}) \in T_{\Graph(\Qpazo)}(\tau,\nu_{\tau}) \qquad \text{and} \qquad \alpha_j \in [0,1]
\end{equation*}
indexed by $j \in \{1,\dots,N\}$, which are chosen in such a way that 
\begin{equation}
\label{eq:ViabTubeZetaDef}
\sum_{j=1}^N \alpha_j = 1 \qquad \text{and} \qquad \zeta_{\tau} := \sum_{j=1}^N \alpha_j \zeta_{\tau}^{\,j} > 0.
\end{equation}
Up to reordering the labels, we may posit that there exists an $m \in \{1,\dots,N\}$ such that $\zeta_{\tau}^{\,j} \geq 0$ if $j \geq m$ and $\zeta_{\tau}^{\,j} < 0$ otherwise. By applying \eqref{eq:ViabTubeDiffEst} to each $(\zeta_{\tau}^{\,j},\xi_{\tau}^{\,j}) \in T_{\Graph(\Qpazo)}(\tau,\nu_{\tau})$ with $j \in \{1,\dots,N\}$, and then summing the resulting expressions depending on whether $j < m$ or $j \geq m$, one has that 
\begin{equation}
\label{eq:ViabTubeSum1}
\sum_{j < m} \alpha_j \zeta_{\tau}^{\,j} \hspace{0.05cm} g^{p-1}(\tau) \dot g(\tau) \leq \sum_{j < m} \alpha_j \bigg( \INTDom{\big\langle \zeta_{\tau}^{\,j} v_{\tau}^{\epsilon,j}(x) - \xi_{\tau}^{\,j}(y) , j_p(x-y) \big\rangle}{\R^{2d}}{\gamma_{\tau}(x,y)} + \epsilon \bigg),
\end{equation}
for some fixed (and potentially empty) tuple $(v_{\tau}^{\epsilon,j})_{j < m} \in V(\tau,\mu_{\tau})^{m-1}$, as well as 
\begin{equation}
\label{eq:ViabTubeSum2}
\sum_{j \geq m} \alpha_j \zeta_{\tau}^{\,j} \hspace{0.05cm} g^{p-1}(\tau) \dot g(\tau) \leq \sum_{j \geq m} \alpha_j \bigg( \INTDom{\big\langle \zeta_{\tau}^{\,j} v_{\tau}(x) - \xi_{\tau}^{\,j}(y) , j_p(x-y) \big\rangle}{\R^{2d}}{\gamma_{\tau}(x,y)} + \epsilon \bigg)
\end{equation}
for any element $v_{\tau} \in V(\tau,\mu_{\tau})$. Introducing in turn the coefficients 
\begin{equation*}
\beta_j := \frac{\alpha_j |\zeta_{\tau}^{\,j}|}{\sum_{j \geq m} \alpha_j \zeta_{\tau}^{\,j}} \in (0,1) ~~ \text{for each $j < m$} \qquad \text{and} \qquad \beta := 1 - \sum_{j < m} \beta_j, 
\end{equation*}
which are well-defined as a consequence of \eqref{eq:ViabTubeZetaDef}, while recalling that set $V(\tau,\mu_{\tau})$ is convex by assumption, it holds for each $v_{\tau} \in V(\tau,\nu_{\tau})$ that
\begin{equation*}
v_{\tau}' := \beta v_{\tau} + \sum_{j < m} \beta_j v_{\tau}^{\epsilon,j} \in V(\tau,\nu_{\tau}). 
\end{equation*}
Whence, by merging the estimate of \eqref{eq:ViabTubeSum1} and that of \eqref{eq:ViabTubeSum2} evaluated at $v_{\tau}' \in V(\tau,\mu_{\tau})$ defined via the previous expression, one eventually obtains that
\begin{equation}
\label{eq:ViabTubeFinalEst}
\zeta_{\tau} \, g^{p-1}(\tau) \dot g(\tau) \leq \INTDom{\big\langle \zeta_{\tau} v_{\tau}(x) - \xi_{\tau}(y) , j_p(x-y) \big\rangle}{\R^{2d}}{\gamma_{\tau}(x,y)} + \epsilon
\end{equation}
for any given $(\zeta_{\tau},\xi_{\tau}) \in \textnormal{co} T_{\Graph(\Qpazo)}(\tau,\nu_{\tau})$ satisfying $\zeta_{\tau} >0$ and every $v_{\tau} \in V(\tau,\mu_{\tau})$. Remarking that the right-hand side in  \eqref{eq:ViabTubeFinalEst} is linear and continuous with respect $(\zeta_{\tau},\xi_{\tau}) \in \R \times \Lpazo^p(\R^d,\R^d;\nu_{\tau})$, the latter expression remains valid for every $(\zeta_{\tau},\xi_{\tau}) \in \co T_{\Graph(\Qpazo)}(\tau,\nu_{\tau})$ such that $\zeta_{\tau} > 0$.

At this stage, starting from \eqref{eq:ViabTubeFinalEst}, one may repeat the argument discussed at the end of Step 1 in the proof of Theorem \ref{thm:ViabStationary} while using the facts that $\epsilon > 0$ is arbitrary and $\gamma_{\tau} \in \Gamma_o(\mu_{\tau},\nu_{\tau})$ to show that the latter estimate further yields 
\begin{equation*}
\zeta_{\tau} \, g^{p-1}(\tau) \dot g(\tau) \leq \zeta_{\tau} \Big( l(\tau) + L(\tau) \Big) g^p(\tau) + \INTDom{\big\langle \zeta_{\tau} w_{\tau}(y) - \xi_{\tau}(y) , j_p(x-y) \big\rangle}{\R^{2d}}{\gamma_{\tau}(x,y)},
\end{equation*}
for every $(\zeta_{\tau},\xi_{\tau}) \in \co T_{\Graph(\Qpazo)}(\tau,\nu_{\tau})$ such that $\zeta_{\tau} > 0$ and each $w_{\tau} \in V(\tau,\nu_{\tau})$. Choosing in particular 
\begin{equation*}
(\zeta_{\tau},\xi_{\tau}) = (1,w_{\tau}) \in ( \{1\} \times V(\tau,\nu_{\tau})) \cap \co T_{\Graph(\Qpazo)}(\tau,\nu_{\tau}),
\end{equation*}
which is licit under our standing assumption \eqref{eq:ViabCompactHypBis}, one finally gets that
\begin{equation*}
\dot g(\tau) \leq \Big( l(\tau) + L(\tau) \Big) g(\tau)
\end{equation*}
for all times $\tau \in (t,t+\delta) \cap \Tcal \cap \Dcal$. Noting that $g(t) = 0$ and $\Tcal,\Dcal \subset (0,T)$ both have full $\Lcal^1$-measure, it follows from Gr\"onwall's lemma that $g(\tau) = 0$ for all $\tau \in [t,t+\delta)$, which implies that $g : [0,T] \to \R_+$ is identically equal to zero and thus leads to a contradiction. One can then deduce the existence of a viable curve by repeating the argument detailed above in Step 2 of the proof of Theorem \ref{thm:ViabStationary}.
\end{proof}

In the following theorem, we state the natural counterpart of the sufficient implication of the invariance result of Theorem \ref{thm:Invariance} for absolutely continuous time-dependent constraint sets.

\begin{thm}[Sufficient invariance conditions for absolutely continuous constraints tubes]
\label{thm:InvarianceTube}
Suppose that the assumptions of Theorem \ref{thm:ViabTubeAC} hold and that the tube $\Qpazo : [0,T] \tto \Pcal_p(\R^d)$ is such that
\begin{equation}
\label{eq:InvarianceCond}
\big( \{1\} \times V(t,\nu) \big) \subset \co T_{\Graph(\Qpazo)}(t,\nu)
\end{equation}
for $\Lcal^1$-almost every $t \in [0,T]$ and all $\nu \in \Qpazo(t)$. Then $\Qpazo : [0,T] \tto \Pcal_p(\R^d)$ is invariant for \eqref{eq:CauchyViab}.
\end{thm}

\begin{proof}
In what follows, we let $\Tcal \subset (0,T)$ and $R>0$ be given as in the proof of Theorem \ref{thm:ViabTubeAC} above, and assume without loss of generality that $(\tau,\mu_{\tau}) = (0,\mu^0)$ for some $\mu^0 \in \Qpazo(0)$. Given an arbitrary curve $\mu(\cdot) \in \Spazo_{[0,T]}(\mu^0)$, our goal is to show that the distance function, defined by 
\begin{equation*}
g(t) := \dist_{\Pcal_p(\R^d)} \big( \mu(t) \, ; \Qpazo(t) \big) 
\end{equation*}
for all times $t \in [0,T]$, is identically equal to zero. Note that since $\Qpazo : [0,T] \tto \Pcal_p(\R^d)$ is absolutely continuous, it can be easily verified that $g(\cdot) \in \AC([0,T],\R_+)$, and we denote by $\Dcal \subset (0,T)$ the subset of full $\Lcal^1$-measure over which it is differentiable. We posit by contradiction that $g(t) = 0$ for some $t \in [0,T]$ and that there exists $\delta > 0$ such that $g(\tau) > 0$ for all $\tau \in (t,t+\delta)$. Observe now that by the compactness of $\Qpazo_R(\tau)$, there exists an element $\nu_{\tau} \in \Qpazo_R(\tau)$ such that 
\begin{equation*}
g(\tau) = W_p(\mu(\tau),\nu_{\tau}).
\end{equation*}
Besides, owing to the choice of $R>0$ made via \eqref{eq:RViabDef}, it holds for every pair of contingent directions $(\zeta_{\tau},\xi_{\tau}) \in T_{\Graph(\Qpazo)}(\tau,\nu_{\tau})$ that
\begin{equation*}
\dist_{\Pcal_p(\R^d)} \Big( (\Id + h_i \xi_{\tau})_{\sharp} \nu_{\tau} \, ; \Qpazo_R(\tau + h_i \zeta_{\tau}^{\,i}) \Big) = o_{\tau}(h_i) 
\end{equation*}
along two given sequences $h_i \to 0^+$ and $\zeta_{\tau}^{\, i} \to \zeta_{\tau}$. By Theorem \ref{thm:Limsup}, there exists for every $\epsilon > 0$ some $v_{\tau}^{\epsilon} \in V(\tau,\mu(\tau))$ such that
\begin{equation*}
W_p \Big( \mu(\tau + h_{i_k}^{\epsilon} \zeta_{\tau}^{\, i_k}) , (\Id + h_{i_k}^{\epsilon} \zeta_{\tau}^{\, i_k} v_{\tau}^{\epsilon})_{\sharp} \mu(\tau) \Big) \leq \epsilon h_{i_k}^{\epsilon} |\zeta_{\tau}^{\, i_k}| + o_{\tau,\epsilon}(h_{i_k}^{\epsilon})
\end{equation*}
along a subsequence $h_{i_k}^{\epsilon} \to 0^+$. By following the arguments leading to the differential inequality \eqref{eq:ViabTubeDiffEst} on $\tfrac{1}{p} g^p(\cdot)$ in the proof of Theorem \ref{thm:ViabTubeAC}, one can derive the estimate
\begin{equation*}
\zeta_{\tau} \, g^{p-1}(\tau) \dot g(\tau) \leq \INTDom{\big\langle \zeta_{\tau} v_{\tau}^{\epsilon}(x) - \xi_{\tau}(y) , j_p(x-y) \big\rangle}{\R^{2d}}{\gamma_{\tau}(x,y)} + \epsilon
\end{equation*}
for all $\gamma_{\tau} \in \Gamma_o(\mu(\tau),\nu_{\tau})$, up to rescaling $\epsilon >0$. Since the latter expression is linear and continuous with respect to $(\zeta_{\tau},\xi_{\tau}) \in T_{\Graph(\Qpazo)}(\tau,\nu_{\tau})$, it holds more generally for elements of $\co T_{\Graph(\Qpazo)}(\tau,\nu_{\tau})$. There now remains to observe that, by Hypothesis \ref{hyp:CI}-$(iv)$, there exists $w_{\tau}^{\epsilon} \in V(\tau,\nu_{\tau})$ such that 
\begin{equation*}
\dsf_{\sup}(v_{\tau}^{\epsilon},w_{\tau}^{\epsilon}) \leq L(\tau) W_p(\mu(\tau),\nu_{\tau}), 
\end{equation*}
which together with Hypothesis \ref{hyp:CI}-$(iii)$, the definition \eqref{eq:DualityDef} of the duality map $j_p : \Lpazo^p(\R^d,\R^d;\mu) \to \Lpazo^q(\R^d,\R^d;\mu)$ and the fact that $\gamma_{\tau} \in \Gamma_o(\mu(\tau),\nu_{\tau})$ yields the differential estimate 
\begin{equation*}
\zeta_{\tau} \, g^{p-1}(\tau) \dot g(\tau) \leq \INTDom{\big\langle \zeta_{\tau} w_{\tau}^{\epsilon}(y) - \xi_{\tau}(y) , j_p(x-y) \big\rangle}{\R^{2d}}{\gamma_{\tau}(x,y)} + \zeta_{\tau} \Big( l(\tau) + L(\tau) \Big) g^p(\tau) + \epsilon,
\end{equation*}
which is valid whenever $\zeta_{\tau} > 0$. Thence, choosing in particular 
\begin{equation*}
(\zeta_{\tau},\xi_{\tau}) = (1,w_{\tau}^{\epsilon}) \in \big( \{1\} \times V(\tau,\nu_{\tau}) \big) \subset \co T_{\Graph(\Qpazo)}(\tau,\nu_{\tau}), 
\end{equation*}
which is licit under our standing assumption \eqref{eq:InvarianceCond}, one can deduce that
\begin{equation*}
\dot g(\tau) \leq \Big( l(\tau) + L(\tau) \Big) g(\tau) + \epsilon, 
\end{equation*}
which finally yields $g(\tau) = 0$ for each $\tau \in (t,t+\delta)$ by applying Gr\"onwall's lemma while noting that $\epsilon > 0$ is arbitrary. This contradicts our initial choice of $t \in [0,T]$. 
\end{proof}

In the following theorem, we finally present sufficient viability and invariance conditions for constraints tubes which are merely left absolutely continuous. In order to treat this less regular case, we shall see that one must relinquish the convexification of the contingent directions which was available both in Theorem \ref{thm:ViabStationary} and Theorem \ref{thm:ViabTubeAC} above. 

\begin{thm}[Sufficient viability and invariance conditions for left absolutely continuous tubes]
\label{thm:ViabTubeACL}
Suppose that $p \in (1,+\infty)$, let $V : [0,T] \times \Pcal_p(\R^d) \tto C^0(\R^d,\R^d)$ be a set-valued map with convex images satisfying Hypotheses \ref{hyp:CI}, and $\Qpazo : [0,T] \tto \Pcal_p(\R^d)$ be a left absolutely continuous constraints tube with proper images such that 
\begin{equation}
\label{eq:LeftAbsViabCond}
\big( \{1\} \times V(t,\nu) \big) \cap T_{\Graph(\Qpazo)}(t,\nu) \neq \emptyset
\end{equation}
for $\Lcal^1$-almost every $t \in [0,T]$ and each $\nu \in \Qpazo(t)$. Then $\Qpazo : [0,T] \tto \Pcal_p(\R^d)$ is viable for \eqref{eq:CauchyViab}. Analogously, if the constraints tube satisfies the stronger condition 
\begin{equation}
\label{eq:LeftAbsInvCond}
\big( \{1\} \times V(t,\nu) \big) \subset T_{\Graph(\Qpazo)}(t,\nu)
\end{equation} 
for $\Lcal^1$-almost every $t \in [0,T]$ and each $\nu \in \Qpazo(t)$, then $\Qpazo : [0,T] \tto \Pcal_p(\R^d)$ is invariant for \eqref{eq:CauchyViab}.
\end{thm}

\begin{proof}
In what follows, we only prove the viability of $\Qpazo : [0,T] \tto \Pcal_p(\R^d)$ under the sufficient condition \eqref{eq:LeftAbsViabCond}, as its invariance under \eqref{eq:LeftAbsInvCond} follows from a straightforward transposition of the method detailed previously in the proof of Theorem \ref{thm:InvarianceTube}. 

Without loss of generality, we assume that $(\tau,\mu_{\tau}) := (0,\mu^0)$ for some $\mu^0 \in \Qpazo(0)$ and let $\Tcal \subset (0,T)$ be the set of full $\Lcal^1$-measure such that the statements of Theorem \ref{thm:Liminf} as well as Hypotheses \ref{hyp:CI}-$(ii)$, $(iii)$ and $(iv)$ hold. Since the constraints tube $\Qpazo : [0,T] \tto \Pcal_p(\R^d)$ is merely left absolutely continuous, the distance function
\begin{equation*}
g : t \in [0,T] \mapsto \dist_{\Pcal_p(\R^d)} \Big( \Rpazo_{(0,t)}(\mu^0) \, ; \Qpazo(t) \Big)
\end{equation*}
is not absolutely continuous in general. To estimate its local variations, we cannot resort to the usual Gr\"onwall approach detailed above, and need to carry out a subtler viability analysis on its epigraph. By contradiction, assume that there exist $t \in [0,T]$ and $\delta > 0$ such that $g(t) = 0$ and $g(\tau) > 0$ for $\tau \in (t,t+\delta)$. By combining the arguments detailed in the proof of Theorem \ref{thm:ViabTubeAC} with \eqref{eq:LeftAbsViabCond}, one can check that for each $\tau \in (t,t+\delta) \cap \Tcal$, it holds that
\begin{equation}
\label{eq:DirectionalDerivativeIneq}
\liminf_{h \to 0^+} \frac{g(\tau+h) - g(\tau)}{h} \leq \Big( l(\tau) + L(\tau) \Big) g(\tau).
\end{equation}
Consider now the auxiliary real-valued tube defined by 
\begin{equation*}
\Ecal(\tau) := \Big\{ \alpha \in \R_+ ~\,\text{s.t.}~ \alpha = g(\tau) + r ~~ \text{for some $r \geq 0$} \Big\}, 
\end{equation*}
as well as the 1-dimensional linear vector field
\begin{equation*}
f(\tau,\alpha) := \Big(l(\tau) + L(\tau) \Big) \alpha
\end{equation*}
given for all times $\tau \in [t,t+\delta]$ and each $\alpha \in \R$, and notice that $\Ecal : [t,t+\delta] \tto \R$ is left absolutely continuous by Proposition \ref{prop:AC}. Given some $\tau \in (t,t+\delta) \cap \Tcal$, remark that if $\alpha \in \Ecal(\tau)$ is such that $\alpha = g(\tau)$, it can then be deduced from \eqref{eq:DirectionalDerivativeIneq} and \cite[Proposition 5.1.4]{Aubin1990} that 
\begin{equation}
\label{eq:CotingentDer1}
f(\tau,g(\tau)) = \Big( l(\tau) + L(\tau) \Big) g(\tau) \in \Big\{ \xi \in \R ~\, \textnormal{s.t.}~ (1,\xi) \in T_{\Graph(\Ecal)}(\tau,g(\tau)) \Big\}. 
\end{equation}
On the other hand, if $\alpha \in \Ecal(\tau)$ is such that $\alpha > g(\tau)$, it is then clear that
\begin{equation}
\label{eq:CotingentDer2}
f(\tau,\alpha)  \in \Big\{ \xi \in \R ~\, \textnormal{s.t.}~ (1,\xi) \in T_{\Graph(\Ecal)}(\tau,\alpha) \Big\} = \R. 
\end{equation}
Whence, by combining \eqref{eq:CotingentDer1} and \eqref{eq:CotingentDer2}, one then obtains 
\begin{equation*}
f(\tau,\alpha) \in \Big\{ \xi \in \R ~\, \textnormal{s.t.}~ (1,\xi) \in T_{\Graph(\Ecal)}(\tau,\alpha) \Big\}
\end{equation*}
for $\Lcal^1$-almost every $\tau \in (t,t+\delta)$ and each $\alpha \in \Ecal(\tau)$. Thus, observing that $g(t) = 0 \in \Ecal(t)$, the classical measurable viability theorem of \cite[Theorem 
4.2]{Frankowska1995} yields the existence of a curve $\alpha(\cdot) \in \AC([t,t+\delta],\R)$ solution of the Cauchy problem
\begin{equation}
\label{eq:CauchyPbViab}
\left\{
\begin{aligned}
\dot \alpha(\tau) & = \Big( l(\tau) + L(\tau) \Big) \alpha(\tau), \\
\alpha(t) & = 0, 
\end{aligned}
\right.
\end{equation}
such that $\alpha(\tau) \in \Ecal(\tau)$ for all times $\tau \in [t,t+\delta]$. Noting that the unique solution of \eqref{eq:CauchyPbViab} is identically equal to zero, we conclude that $g(\tau) = 0$ on that same interval, which contradicts our initial choice of $t \in [0,T]$. From there, the existence of a viable measure can be obtained by repeating the argument in Step 2 of the proof of Theorem \ref{thm:ViabStationary} above.
\end{proof}


\section{Examples of constraints sets and computations of tangents}
\label{section:Cone}

In this section, we provide two examples of proper constraint sets $\Qpazo \subset \Pcal_p(\R^d)$ which frequently arise in applications, and compute in each case some relevant subsets of tangent directions.

\paragraph*{Constraints sets defined by support inclusions.}

In this first example, we start by considering the prototypical case in which the constraints set is given as
\begin{equation*}
\Qpazo_K := \Big\{ \mu \in \Pcal_p(\R^d) ~\,\text{s.t.} ~ \supp(\mu) \subset K \Big\}, 
\end{equation*}
for some compact set $K \subset \R^d$, where $\supp(\mu) \subset \R^d$ denotes the \textit{support} of $\mu \in \Pcal(\R^d)$, defined by 
\begin{equation*}
\supp(\mu) := \bigg\{ x \in \R^d ~\, \textnormal{s.t.}~ \mu(\Npazo_x) > 0 ~\textnormal{for each neighbourhood $\Npazo_x$ of $x$ in $\R^d$} \bigg\}.
\end{equation*}
By Proposition \ref{prop:Wass}, it can easily be checked that $\Qpazo_K \subset \Pcal_p(\R^d)$ is compact and thus proper. In what ensues, we fully characterise a nice subset of the \textit{adjacent cone} (see e.g. \cite[Definition 4.1.5]{Aubin1990} for a general definition) to $\Qpazo_K$ at some $\mu \in \Qpazo$, defined in our context by 
\begin{equation}
\label{eq:AdjacentCone}
T^{\flat}_{\Qpazo_K}(\mu) := \bigg\{ \xi \in \Lpazo^p(\R^d,\R^d;\mu) ~\, \text{s.t.}~ \lim_{h \to 0^+} \tfrac{1}{h} \dist_{\Pcal_p(\R^d)} \Big( (\Id + h \xi)_{\sharp} \mu \, ; \Qpazo_K \Big) = 0 \bigg\}. 
\end{equation}
Notice in particular that $T^{\flat}_{\Qpazo_K}(\mu) \subset T_{\Qpazo_K}(\mu)$ by construction. 

\begin{prop}[Computation of adjacent directions to $\Qpazo_K$]
\label{prop:AdjacentCone}
For every $\mu \in \Qpazo_K$, it holds that 
\begin{equation*}
\bigg\{ \xi \in \Lpazo^p(\R^d,\R^d;\mu) ~\, \textnormal{s.t.}~ \xi(x) \in T^{\flat}_K(x) ~\text{for $\mu$-almost every $x \in K$} \bigg\} \subset T^{\flat}_{\Qpazo_K}(\mu) 
\end{equation*}
where $T^{\flat}_K(x)$ denotes the standard adjacent cone to $K \subset \R^d$ at $x \in K$.
\end{prop}

\begin{proof}
Let $\xi \in \Lpazo^p(\R^d,\R^d;\mu)$ be such that $\xi(x) \in T^{\flat}_K(x)$ for $\mu$-almost every $x \in K$ and $h>0$ be given. In addition, denote by $\hat{\mu}$ any complete extension of the Borel measure $\mu$ given e.g. by \cite[Theorem 1.36]{Rudin1987}, and observe that the map $x \mapsto x + h\xi(x)$ is Borel and thus $\hat{\mu}$-measurable, see for instance \cite[Definition 1.12]{AmbrosioFuscoPallara}. Then, it follows from \cite[Theorem 8.2.11]{Aubin1990} that the set-valued mapping
\begin{equation*}
\Dpazo_K : x \in K \tto \underset{y \in K}{\textnormal{argmin}} \, |x + h \xi(x) - y| \subset K 
\end{equation*}
is $\hat{\mu}$-measurable as well. Because the latter has closed images and since the $\sigma$-algebra of $\hat{\mu}$-measurable sets is complete, it follows from \cite[Theorem 8.1.3]{Aubin1990} that there exists a $\hat{\mu}$-measurable selection $x \in K \mapsto \hat{d}_K(x) \in \Dpazo_K(x) \subset K$ which satisfies
\begin{equation*}
|x + h \xi(x) - \hat{d}_K(x)| = \dist_{\R^d} \big( x + h \xi(x) \, ; K \big) = o_x(h)
\end{equation*}
for $\hat{\mu}$-almost every $x \in K$ as $h \to 0^+$ , where $|o_x(h)| \leq h |\xi(x)|$. Moreover, by \cite[Proposition 2.1.11]{Bogachev}, the latter coincides with a Borel map  $d_K : K \to K$ outside of a Borel set of zero $\hat{\mu}$-measure. Therefore, noting in turn that 
\begin{equation*}
\supp(d_{K \, \sharp} \mu) \subset \overline{ \Big\{ d_K(x) ~\, \textnormal{s.t.}~ x \in \supp(\mu) \Big\}}^{\, \R^d} \subset K
\end{equation*}
since $d_K^{-1}(B(y,\epsilon)) = \emptyset$ for every $y \in \R^d \setminus K$ and $\epsilon >0$ for which $B(y,\epsilon) \cap K = \emptyset$, because $K \subset \R^d$ is closed, it further holds that
\begin{equation*}
\begin{aligned}
\dist_{\Pcal_p(\R^d)} \Big( (\Id + h\xi)_{\sharp} \mu \, ; \Qpazo_K \Big) & \leq W_p \Big( (\Id + h \xi)_ {\sharp} \mu , d_{K \, \sharp} \mu \Big)  \\
& \leq \; \NormLp{\Id + h \xi - d_K}{p}{\R^d,\R^d; \,\mu} ~=~ o(h)
\end{aligned}
\end{equation*}
as $h \to 0^+$ by Lebesgue's dominated convergence theorem. This concludes the proof by definition \eqref{eq:AdjacentCone} of the adjacent cone  $T_{\Qpazo_K}^{\flat}(\mu)$.  
\end{proof}

It is possible to generalise this example to time-dependent tubes $K : [0,T] \tto \R^d$. In this context, we define the constraints sets by 
\begin{equation*}
\Qpazo_K(t) := \Big\{ \mu \in \Pcal_p(\R^d) ~\, \textnormal{s.t.}~ \supp(\mu) \subset K(t) \Big\}
\end{equation*}
for all times $t \in [0,T]$. In what follows, we treat the case in which $K : [0,T] \tto \R^d$ is left absolutely continuous with nonempty compact images. 

\begin{prop}[Regularity in time of the constraints]
Under our assumptions on $K : [0,T] \tto \R^d$, the constraints tube $\Qpazo_K : [0,T] \tto \Pcal_p(\R^d)$ is left absolutely continuous. 
\end{prop}

\begin{proof}
Fix some $\tau \in [0,T]$ along with $\mu \in \Pcal_p(\R^d)$ and $R>0$ for which $\Qpazo_K(\tau) \cap \B_{\Pcal_p(\R^d)}(\mu,R) \neq \emptyset$, and let $\mu_{\tau} \in \Qpazo_K(\tau) \cap \B_{\Pcal_p(\R^d)}(\mu,R)$. Since $K : [0,T] \tto \R^d$ is left absolutely continuous, there exists for each $x_{\tau} \in \supp(\mu_{\tau})$ and every $t \in [0,T]$ such that $\tau \leq t$ some point $x_t \in K(t)$ for which
\begin{equation*}
|x_t - x_{\tau}| \leq \INTSeg{m_K(s)}{s}{\tau}{t}
\end{equation*}
where $m_K(\cdot) \in L^1([0,T],\R_+)$ only depends on $K : [0,T] \tto \R^d$. Whence, up to a trivial extension argument outside $\supp(\mu_{\tau})$, it is possible to construct a Borel map $\phi_{(\tau,t)} : K(\tau) \to K(t)$ which satisfies
\begin{equation*}
|\phi_{(\tau,t)}(x_{\tau}) - x_{\tau}| \leq \INTSeg{m_K(s)}{s}{\tau}{t}
\end{equation*}
for each $x_{\tau} \in \supp(\mu_{\tau})$. Then, one can easily show that the measure $\mu_t := \phi_{(\tau,t) \, \sharp} \mu_{\tau}$ satisfies $\supp(\mu_t) \subset K(t)$, and by \eqref{eq:WassEst} the latter also complies with the estimate
\begin{equation*}
W_p(\mu_{\tau},\mu_t) \leq \INTSeg{m_K(s)}{s}{\tau}{t}. 
\end{equation*}
Repeating the arguments supporting the regularity statement of Theorem \ref{thm:NecessaryConds} then closes the proof. 
\end{proof}

\begin{prop}[Computation of adjacent directions to $\Graph(\Qpazo_K)$]
For all times $t \in [0,T]$ and each $\mu \in \Qpazo_K(t)$, it holds that 
\begin{equation*}
\begin{aligned}
& \bigg\{ (\zeta,\xi) \in \R \times \Lpazo^p(\R^d,\R^d;\mu) ~\, \textnormal{s.t.}~ (\zeta,\xi(x)) \in T^{\flat}_{\Graph(K)}(t,x) ~ \textnormal{for $\mu$-almost every $x \in K(t)$} \bigg\} \\
& \hspace{13.25cm} \subset T^{\flat}_{\Graph(\Qpazo_K)}(t,\mu). 
\end{aligned}
\end{equation*}
\end{prop}

\begin{proof}
Given some $t \in [0,T]$, an element $\mu \in \Qpazo(t)$ and some $(\zeta,\xi) \in \R \times \Lpazo^p(\R^d,\R^d;\mu)$ such that $(\zeta,\xi(x)) \in T^{\flat}_{\Graph(K)}(t,x)$ for $\mu$-almost every $x \in \R^d$, one can show by repeating the arguments in the proof of Proposition \ref{prop:AdjacentCone} that the set-valued map 
\begin{equation*}
\Dpazo_{K(t)} : K(t) \tto \underset{y \in K(t+h \zeta)}{\textnormal{argmin}} |x + h \xi(x) - y| \subset K(t+h \zeta)
\end{equation*}
admits a Borel selection $x \in K(t) \mapsto d_{K(t)}(x) \in K(t+ h \zeta)$ which satisfies 
\begin{equation*}
|x + h \xi(x) - d_{K(t)}(x)| = o_x(h)  
\end{equation*}
for $\mu$-almost every $x \in \R^d$ as $h \to 0^+$, where $|o_x(h)| \leq h |\xi(x)|$. Then, there simply remains to note that $d_{K(t) \, \sharp} \mu \in \Qpazo_K(t+h)$ by construction, while 
\begin{equation*}
\begin{aligned}
\dist_{\Pcal_p(\R^d)} \Big( (\Id + h\xi)_{\sharp} \mu \, ; \Qpazo_K(t+h) \Big) & \leq W_p \Big( (\Id + h \xi)_ {\sharp} \mu , d_{K(t) \, \sharp} \mu  \Big)  \\
& \leq \; \NormLp{\Id + h \xi - d_{K(t)}}{p}{\R^d,\R^d; \,\mu} ~=~ o(h)
\end{aligned}
\end{equation*}
by Lebesgue's dominated convergence theorem, which concludes the proof. 
\end{proof}


\paragraph*{Constraints sets defined as lifted epigraphs.} 

In this second example, which is discussed in our earlier work \cite{ViabCDC}, we consider an extended real-valued map $\Wpazo : \Pcal_p(\R^d) \to \R \cup \{+\infty\}$ with compact sublevels that is continuous over its domain $\dom(\Wpazo) \subset \Pcal_p(\R^d)$, and define its \textit{lifted epigraph} by 
\begin{equation*}
\Qpazo_{\Wpazo} := \bigg\{ \Bmu \in \Pcal_p(\R^{d+1}) ~\, \text{s.t.}~ \Bmu = \mu \times \delta_{\alpha} ~\text{with}~ \mu \in \Pcal_p(\R^d)~\text{and}~ \Wpazo(\mu) \leq \alpha \bigg\}. 
\end{equation*}
In what follows, we discuss the topological properties of this constraint set and provide a full characterisation of a relevant subset of its contingent cone. 

\begin{prop}[Topological properties of $\Qpazo_{\Wpazo}$]
Under the assumptions listed hereinabove on the function $\Wpazo : \Pcal_p(\R^d) \to \R_+ \cup \{+\infty\}$, the set $\Qpazo_{\Wpazo} \subset \Pcal_p(\R^d)$ is proper. 
\end{prop}

\begin{proof}
Showing that $\Qpazo_{\Wpazo}$ is closed under our assumptions is a matter of routine computations. Given $\Bmu = \mu \times \delta_{\alpha} \in \Qpazo_{\Wpazo}$ and some $\Bnu \in \Pcal_p(\R^{d+1})$, one has that 
\begin{equation*}
\begin{aligned}
W_p(\Bmu,\Bnu) & = \bigg( \INTDom{|(x,\alpha) - (y,\beta)|^p}{\R^{2(d+1)}}{\Bgamma(x,\alpha,y,\beta)} \bigg)^{1/p} \\
& \geq \bigg( \INTDom{|\alpha - \beta|^p}{\R^{2(d+1)}}{\Bgamma(x,\alpha,y,\beta)} \bigg)^{1/p} \\
& \geq |\alpha| - \mathsmaller{\INTDom{|\beta|}{\R^{d+1}}{\Bnu(y,\beta)}} 
\end{aligned}
\end{equation*}
for all $\Bgamma \in \Gamma_o(\Bmu,\Bnu)$, where we used Jensen's inequality. In particular given $R > 0$, it then holds for every $\Bmu \in \Qpazo_{\Wpazo} \cap \B_{\Pcal_p(\R^{d+1})}(\Bnu,R)$ that 
\begin{equation*}
\Wpazo(\mu) \leq \alpha \leq R + \mathsmaller{\INTDom{|\beta|}{\R^{d+1}}{\Bnu(y,\beta)}}. 
\end{equation*}
Since $\Wpazo : \Pcal_p(\R^d) \to \R \cup \{+\infty\}$ has compact sublevels, we deduce that $\Qpazo_{\Wpazo} \cap \B_{\Pcal_p(\R^{d+1})}(\Bnu,R)$ is compact for any $\Bnu \in \Pcal_p(\R^{d+1})$ and $R >0$, and thus that $\Qpazo_{\Wpazo}$ is proper.
\end{proof}

Similarly to what is known in nonsmooth analysis, the contingent cone to $\Qpazo_{\Wpazo}$ can be expressed in terms of \textit{lower directional derivatives} of the functional. Following \cite{Badreddine2022Bis}, we define these latter as
\begin{equation}
\label{eq:LowerDerivativeDef}
\D_{\uparrow} \Wpazo(\mu)(\xi) := \liminf_{\substack{h \to 0^+ \hspace{-0.065cm} , \; \mu_{h} \in \dom(\Wpazo) \\ W_p((\Id + h \xi)_{\sharp}\mu,\mu_h) = o(h)}} \hspace{-0.35cm} \frac{\Wpazo(\mu_h) - \Wpazo(\mu)}{h}
\end{equation}
for each $\mu \in \dom(\Wpazo)$ and every $\xi \in T_{\dom(\Wpazo)}(\mu) \subset \Lpazo^p(\R^d,\R^d;\mu)$. 

\begin{prop}[Characterisation of contingent directions to $\Qpazo_{\Wpazo}$]
For every $\Bmu \in \Qpazo_{\Wpazo}$ and each $(\xi,\rho) \in \Lpazo^p(\R^d,\R^d;\mu) \times \R$, it holds that
\begin{equation*}
(\xi,\rho) \in T_{\Qpazo_{\Wpazo}}(\Bmu) \qquad \text{if and only if} \qquad \left\{
\begin{aligned}
& \xi \in T_{\dom(\Wpazo)}(\mu)  ~~\textnormal{and}~~ \D_{\uparrow} \Wpazo(\mu)(\xi) \leq \rho ~~ & \text{when $\alpha = \Wpazo(\mu)$}, \\
& \xi \in T_{\dom(\Wpazo)}(\mu)  ~~ & \text{when $\alpha > \Wpazo(\mu)$}.
\end{aligned}
\right. 
\end{equation*}
\end{prop}

\begin{proof}
The ensuing arguments are largely inspired by those of \cite[Proposition 6.1.4]{Aubin1990}. We only treat the case $\alpha = \Wpazo(\mu)$, the other being similar. Assume first that $(\xi,\rho) \in T_{\Qpazo_{\Wpazo}}(\Bmu)$ so that one may find a sequence $h_i \to 0^+$ and some $\Bmu_{h_i} := \mu_{h_i} \times \delta_{\alpha_{h_i}} \in \Qpazo_{\Wpazo}$ which satisfy
\begin{equation}
\label{eq:ContingentSmallo}
W_p \Big( (\Id + h_i(\xi,\rho))_{\sharp} \Bmu , \Bmu_{h_i} \Big) \leq o(h_i).
\end{equation}
Note also that, as a consequence of the disintegration theorem (see e.g. \cite[Theorem 5.2.1]{AGS}) and up to a permutation of coordinates, each optimal plan $\Bgamma_{h_i} \in \Gamma_o((\Id + h_i(\xi,\rho))_{\sharp} \Bmu , \Bmu_{h_i})$ is of the form 
\begin{equation*}
\Bgamma_{h_i} = \gamma_{h_i} \times \delta_{(\Wpazo(\mu) + h_i \rho,\alpha_{h_i})}
\end{equation*}
for some $\gamma_{h_i} \in \Gamma_o((\Id + h_i\xi)_{\sharp} \mu , \mu_{h_i})$. Thence, it necessarily holds that
\begin{equation*}
\begin{aligned}
W_p \Big( (\Id + h_i\xi)_{\sharp} \mu , \mu_{h_i} \Big) & = \bigg( \INTDom{|x-y|^p}{\R^{2d}}{\gamma_{h_i}(x,y)} \bigg)^{1/p} \\
& \leq \bigg( \INTDom{|(x,\alpha)-(y,\beta)|^p}{\R^{2(d+1)}}{\Bgamma_{h_i}(x,\alpha,y,\beta)} \bigg)^{1/p} \hspace{-0.15cm} \\
& = W_p \Big( (\Id + h_i(\xi,\rho))_{\sharp} \Bmu , \Bmu_{h_i} \Big).
\end{aligned}
\end{equation*}
Owing to \eqref{eq:ContingentSmallo}, this implies in particular that $\xi \in T_{\dom(\Wpazo)}(\mu)$. Similarly, one can show that 
\begin{equation*}
|\alpha_{h_i} - \Wpazo(\mu) - h_i \rho| \leq  W_p \Big( (\Id + h_i(\xi,\rho))_{\sharp} \Bmu , \Bmu_{h_i} \Big) = o(h_i)
\end{equation*}
which together with \eqref{eq:LowerDerivativeDef} and \eqref{eq:ContingentSmallo} finally yields
\begin{equation*}
\D_{\uparrow} \Wpazo(\mu)(\rho) \leq \liminf_{h_i \to 0^+} \frac{\Wpazo(\mu_{hi})-\Wpazo(\mu)}{h_i} \leq \rho. 
\end{equation*}
Conversely, let $(\xi,\rho) \in T_{\dom(\Wpazo)}(\mu) \times \R$ be such that $\D_{\uparrow} \Wpazo(\mu)(\xi) \leq \rho$, and observe then that there exist sequences $h_i \to 0^+$ and $(\mu_{h_i}) \subset \dom(\Wpazo)$ satisfying $W_p(\mu_{h_i},(\Id + h_i \xi)_{\sharp}\mu) = o(h_i)$, for which 
\begin{equation*}
\Wpazo(\mu_{h_i}) \leq \Wpazo(\mu) + h_i \rho + o(h_i)
\end{equation*}
when $h_i > 0$ is small enough. Hence, there exists $\rho_i \to \rho$ such that $\Bmu_{h_i} := \mu_{h_i} \times \delta_{\Wpazo(\mu) + h_i\rho_i} \in \Qpazo_{\Wpazo}$ and 
\begin{equation*}
W_p \Big( (\Id + h_i(\xi,\rho))_{\sharp} \big(\mu \times \delta_{\Wpazo(\mu)} \big) ,  \Bmu_{h_i} \Big) \leq W_p \big((\Id + h_i \xi)_{\sharp} \mu, \mu_{h_i} \big) + h_i(\rho-\rho_i) = o(h_i)
\end{equation*}
as $h_i \to 0^+$, which equivalently means that $(\xi,\rho) \in T_{\Qpazo_{\Wpazo}}(\Bmu)$. 
\end{proof}

\addcontentsline{toc}{section}{Appendices}
\section*{Appendices}


\setcounter{section}{0} 
\renewcommand{\thesection}{A} 
\renewcommand{\thesubsection}{A} 

\subsection{Proof of Proposition \ref{prop:SuperdiffWass}}
\label{section:AppendixSuperdiff}

\setcounter{Def}{0} \renewcommand{\thethm}{A.\arabic{Def}} 
\setcounter{equation}{0} \renewcommand{\theequation}{A.\arabic{equation}}

In this appendix, we detail the proof of Proposition \ref{prop:SuperdiffWass}. For the sake of self-containedness, we recall first the following technical result taken from \cite[Lemma 10.2.1]{AGS}. 

\begin{lem}[Quantitative superdifferentiability estimates on powers of the euclidean norm]
\label{lem:PnormEst}
Given $x,y \in \R^d$, one has for $p \in (1,2]$ that 
\begin{equation*}
\tfrac{1}{p} |y|^p - \tfrac{1}{p} |x|^p - \langle y-x , j_p(x) \rangle \leq \tfrac{2^{2-p}}{p-1} |x-y|^p, 
\end{equation*}
whereas for $p \in [2,+\infty)$, it holds that
\begin{equation*}
\tfrac{1}{p} |y|^p - \tfrac{1}{p} |x|^p - \langle y-x , j_p(x) \rangle \leq \tfrac{p-1}{2} |x-y|^2 \max\{ |x| , |y| \}^{p-2}.
\end{equation*}
Therein, $j_p : \R^d \to \R^d$ is the usual duality map defined by 
\begin{equation*}
j_p(x) := 
\left\{
\begin{aligned}
& 0 ~~ & \text{if $x=0$}, \\
& |x|^{p-2} x ~~  & \text{otherwise}.
\end{aligned}
\right.
\end{equation*}
\end{lem}

\begin{proof}[Proof of Proposition \ref{prop:SuperdiffWass}]
Given an element $\gamma \in \Gamma_o(\mu,\nu)$ and some $h \in \R$, define the transport plan 
\begin{equation*}
\gamma_h := \Big( (\Id + h \zeta) \circ \pi^1 , (\Id + h \xi) \circ \pi^2 \Big)_{\raisebox{4pt}{$\scriptstyle{\sharp}$}} \gamma \in \Gamma \Big( (\Id + h \zeta)_{\sharp} \mu , (\Id + h \xi)_{\sharp} \nu \Big), 
\end{equation*}
and note that by construction, one has that
\begin{equation}
\label{eq:SuperdiffIneq0}
\begin{aligned}
\tfrac{1}{p} W_p^p \Big( (\Id + h \zeta)_{\sharp} \mu, (\Id + h \xi)_{\sharp}\nu \Big) - \tfrac{1}{p} W_p^p(\mu,\nu) & \leq \INTDom{\tfrac{1}{p}|x_h-y_h|^p}{\R^{2d}}{\gamma_h(x_h,y_h)} - \INTDom{\tfrac{1}{p} |x-y|^p}{\R^{2d}}{\gamma(x,y)} \\
& \leq \INTDom{\Big( \tfrac{1}{p} |x-y + h (\zeta(x) - \xi(y))|^p - \tfrac{1}{p} |x-y|^p \Big)}{\R^{2d}}{\gamma(x,y)}. 
\end{aligned}
\end{equation}
By leveraging the identities of Lemma \ref{lem:PnormEst} above, it can be checked that for $p \in (1,2]$, one has that
\begin{equation}
\label{eq:SuperdiffIneq1}
\begin{aligned}
\INTDom{\Big( \tfrac{1}{p} |x-y + h (\zeta(x) - \xi(y)) |^p - \tfrac{1}{p} |x-y|^p & - h \,  \big\langle \zeta(x) - \xi(y) , j_p(x-y) \big\rangle \Big)}{\R^{2d}}{\gamma(x,y)} \\
& \leq  \tfrac{2^{2-p}}{p-1} \INTDom{|h(\zeta(x) - \xi(y))|^p}{\R^{2d}}{\gamma(x,y)}  \\
& \leq  \tfrac{2}{p-1} |h|^p \Big( \hspace{-0.1cm} \NormLp{\zeta}{p}{\R^d,\R^d;\, \mu}^p + \NormLp{\xi}{p}{\R^d,\R^d; \, \nu}^p \hspace{-0.1cm} \Big),
\end{aligned}
\end{equation}
whereas for $p \in [2,+\infty)$, it holds 
\begin{equation}
\label{eq:SuperdiffIneq2}
\begin{aligned}
& \INTDom{\Big( \tfrac{1}{p} |x-y + h (\zeta(x) - \xi(y)) |^p - \tfrac{1}{p} |x-y|^p - h \,  \big\langle \zeta(x) - \xi(y) , j_p(x-y) \big\rangle \Big)}{\R^{2d}}{\gamma(x,y)}  \\
& \leq \tfrac{p-1}{2} \INTDom{|h(\zeta(x) - \xi(y))|^2 \max \Big\{ |x-y| , |x-y + h(\zeta(x) - \xi(y))| \Big\}^{p-2}}{\R^{2d}}{\gamma(x,y)} \\
& \leq \tfrac{(p-1)}{2} |h|^2 \INTDom{|\zeta(x) - \xi(y)|^2 \Big( |x-y| + |x-y + h(\zeta(x) - \xi(y))| \Big)^{p-2}}{\R^{2d}}{\gamma(x,y)} \\
& \leq (p-1) |h|^2 \bigg( W_p(\mu,\nu) + |h| \Big( \hspace{-0.1cm} \NormLp{\zeta}{p}{\R^d,\R^d; \,\mu} + \NormLp{\xi}{p}{\R^d,\R^d; \,\nu} \hspace{-0.1cm} \Big) \bigg)^{p-2} \Big( \hspace{-0.1cm} \NormLp{\zeta}{p}{\R^d,\R^d; \, \mu}^2 + \NormLp{\xi}{p}{\R^d,\R^d; \,\nu}^2 \hspace{-0.1cm} \Big),
\end{aligned}
\end{equation}
where we used elementary H\"older and convexity inequalities to derive both estimates. Thence, upon combining \eqref{eq:SuperdiffIneq1} and \eqref{eq:SuperdiffIneq2} with \eqref{eq:SuperdiffIneq0} depending on the value of $p \in (1,+\infty)$, one finally obtains that 
\begin{equation*}
\tfrac{1}{p} W_p^p \Big( (\Id + h\xi)_{\sharp} \mu, (\Id + h\zeta)_{\sharp}\nu \Big) - \tfrac{1}{p} W_p^p(\mu,\nu) \leq h \INTDom{\big\langle \zeta(x) - \xi(y), j_p(x-y) \big\rangle}{\R^{2d}}{\gamma(x,y)} + r_p(h,\xi,\zeta)
\end{equation*}
with $r_p(h,\zeta,\xi)$ being defined as in \eqref{eq:Remainder1} or \eqref{eq:Remainder2} depending on the value of $p \in (1,+\infty)$. 
\end{proof}


\setcounter{section}{0} 
\renewcommand{\thesection}{B} 
\renewcommand{\thesubsection}{B} 

\subsection{Proof of Proposition \ref{prop:AC}}
\label{section:AppendixAC}

\setcounter{Def}{0} \renewcommand{\thethm}{B.\arabic{Def}} 
\setcounter{equation}{0} \renewcommand{\theequation}{B.\arabic{equation}}

In this appendix section, we detail the proof of Proposition \ref{prop:AC}.

\begin{proof}[Proof of Proposition \ref{prop:AC}]
In what follows, we assume without loss of generality that $I := [0,T]$ for some $T>0$, and start by showing that when $\Qpazo : [0,T] \tto X$ is absolutley continuous, the map 
\begin{equation*}
t \in [0,T] \mapsto g(t) := \dist_X(\Kpazo(t) \, ; \Qpazo(t))
\end{equation*}
is absolutely continuous as well. To do so, we first need to establish some preliminary facts. Observe that since $\Kpazo(0)$ is compact, there exist $x_0 \in X$ and some $r_0 > 0$ such that
\begin{equation}
\label{eq:KpazoZero}
\Kpazo(0) \subset \B_X(x_0,r_0)
\end{equation}
for all times $t \in [0,T]$. In what follows, we show that the map $t \in [0,T] \mapsto \dist_X(x_0 \, ; \Qpazo(t))$ is continuous. Indeed, fixing $\tau \in [0,T]$, setting $R_{\tau} := \dist_X(x_0 \, ; \Qpazo(\tau))$ and recalling that $\Qpazo : [0,T] \tto X$ is absolutely continuous, for each $\epsilon > 0$ there exists some $\delta > 0$ such that
\begin{equation*}
\Qpazo(\tau) \cap \B_X(x_0,R_{\tau}+\epsilon) \subset \B_X \big( \Qpazo(t),\epsilon \big) \qquad \text{and} \qquad  \Qpazo(t) \cap \B_X(x_0,R_{\tau}+\epsilon) \subset \B_X \big( \Qpazo(\tau),\epsilon \big) 
\end{equation*}
whenever $|t-\tau| \leq \delta$. Noticing in turn that $\Qpazo(\tau) \cap \B_X(x_0,R_{\tau}+\epsilon) \neq \emptyset$ and $\Qpazo(t) \cap \B_X(x_0,R_{\tau}+\epsilon) \neq \emptyset$ by construction, the first of these inclusions implies that
\begin{equation*}
\begin{aligned}
\dist_X(x_0 \, ; \Qpazo(t)) & \leq \dist_X(x_0 \, ; \Qpazo(\tau)) + \dist_X \Big( \Qpazo(\tau) \cap \B_X(x_0,R_{\tau}+\epsilon) \, ; \Qpazo(t) \Big) \\
& \leq \dist_X(x_0 \, ; \Qpazo(\tau)) + \epsilon,  
\end{aligned}
\end{equation*}
while the second one analogously yields 
\begin{equation*}
\dist_X(x_0 \, ; \Qpazo(\tau)) \leq \dist_X(x_0 \, ; \Qpazo(t)) + \epsilon, 
\end{equation*}
from whence we can deduce that $t \mapsto \dist_X(x_0 \, ; \Qpazo(t))$ is continuous at $\tau \in [0,T]$. Recalling that $\Kpazo : [0,T] \tto X$ has compact images and that it satisfies
\begin{equation}
\label{eq:HausdorffEstimate}
\dsf_{\Hpazo}(\Kpazo(\tau) \, ; \Kpazo(t)) \leq \INTSeg{m_{\Kpazo}(s)}{s}{\tau}{t}
\end{equation}
for all times $0 \leq \tau \leq t \leq T$ and some $m_{\Kpazo}(\cdot) \in L^1([0,T],\R_+)$, it follows from \eqref{eq:KpazoZero} that $\Kpazo(t) \subset \B_X ( x_0, r_{\Kpazo})$ for all times $t \in [0,T]$ with $r_{\Kpazo} := r_0 \, + \Norm{m_{\Kpazo}(\cdot)}_1$. Notice then that 
\begin{equation*}
\dist_X(x\, ; \Qpazo(t)) \leq r_{\Kpazo} + \max_{t \in [0,T]} \dist_X(x_0 \, ; \Qpazo(t)) 
\end{equation*}
for all $t \in [0,T]$ and every $x \in \B_X(x_0,r_{\Kpazo})$, so that the quantity
\begin{equation}
\label{eq:RTDef}
R_T := (r_{\Kpazo}+1) + \sup \Big\{ \dist_X(x \, ; \Qpazo(t)) ~\, \textnormal{s.t.}~ (t,x) \in [0,T] \times \B_X(x_0,r_{\Kpazo}) \Big\}
\end{equation}
is well-defined and such that
\begin{equation*}
\dist_X(\Kpazo(t) \, ; \partial \B_X(x_0,R_T)) \geq \sup_{\tau,s \in [0,T]}\dist_X (\Kpazo(\tau) \, ; \Qpazo(s)) +1 
\end{equation*}
for all times $t \in [0,T]$.

We now prove that $g : [0,T] \to \R_+$ is absolutely continuous. For all times $\tau,t \in [0,T]$ satisfying $0 \leq \tau \leq t \leq T$, it holds that
\begin{equation}
\label{eq:gEst1}
\begin{aligned}
|g(t)-g(\tau)| & \leq \big| \dist_X(\Kpazo(t) \, ; \Qpazo(t)) - \dist_X(\Kpazo(\tau) \, ; \Qpazo(t)) \big| \\
& \hspace{0.45cm} + \big| \dist_X(\Kpazo(\tau) \, ; \Qpazo(t)) - \dist_X(\Kpazo(\tau) \, ; \Qpazo(\tau)) \big|.  
\end{aligned}
\end{equation}
In order to estimate the first term in \eqref{eq:gEst1}, note that for each $\epsilon > 0$, there exist $y_t^{\epsilon} \in \Qpazo(t)$ and $x_{\tau}^{\epsilon} \in \Kpazo(\tau)$ such that
\begin{equation*}
\dsf_X(x_{\tau}^{\epsilon},y_t^{\epsilon}) \leq \dist_X(\Kpazo(\tau) \, ; \Qpazo(t)) + \epsilon.
\end{equation*}
Furthermore, it stems from \eqref{eq:HausdorffEstimate} that there exists an element $x_t^{\epsilon} \in \Kpazo(t)$ for which 
\begin{equation*}
\dsf_X(x_{\tau}^{\epsilon},x_t^{\epsilon}) \leq \INTSeg{m_{\Kpazo}(s)}{s}{\tau}{t}. 
\end{equation*}
Merging both estimates, it then follows that 
\begin{equation*}
\begin{aligned}
\dist_X(\Kpazo(t) \, ; \Qpazo(t)) - \dist_X(\Kpazo(\tau) \, ; \Qpazo(t)) & \leq \dsf_X(x_t^{\epsilon},y_t^{\epsilon}) - \dsf_X(x_{\tau}^{\epsilon},y_t^{\epsilon}) + \epsilon \\
& \leq \INTSeg{m_{\Kpazo}(s)}{s}{\tau}{t} + \epsilon
\end{aligned}
\end{equation*}
and repeating the same argument while exchanging the roles of $\tau$ and $t$ further yields
\begin{equation}
\label{eq:gEst2}
\big| \dist_X(\Kpazo(t) \, ; \Qpazo(t)) - \dist_X(\Kpazo(\tau) \, ; \Qpazo(t)) \big| \leq \INTSeg{m_{\Kpazo}(s)}{s}{\tau}{t}
\end{equation}
since $\epsilon > 0$ was arbitrary. Concerning the second term in \eqref{eq:gEst1}, it stems from our choice of $R_T > 0$ in \eqref{eq:RTDef} that
\begin{equation*}
\dist_X(\Kpazo(\tau) \, ; \Qpazo(\tau)) = \dist_X \Big( \Kpazo(\tau) \, ; \Qpazo(\tau) \cap \B_X(x_0,R_T) \Big)
\end{equation*}
and
\begin{equation*}
\dist_X (\Kpazo(\tau) \, ; \Qpazo(t)) = \dist_X \Big( \Kpazo(\tau) \, ; \Qpazo(t) \cap \B_X(x_0,R_T) \Big)
\end{equation*}
Moreover, since $\Qpazo : [0,T] \tto X$ is absolutely continuous in the sense of Definition \ref{def:AC}, there exists a map $m_{x_0,R_T}(\cdot) \in L^1([0,T],\R_+)$ for which 
\begin{equation*}
\Qpazo(\tau) \cap \B(x_0,R_T) \subset \B_X \Big( \Qpazo(t) \hspace{0.015cm} , \mathsmaller{\INTSeg{m_{x_0,R_T}(s)}{s}{\tau}{t}} \Big) \quad \text{and} \quad \Qpazo(t) \cap \B(x_0,R_T) \subset \B_X \Big( \Qpazo(\tau) \hspace{0.015cm} , \mathsmaller{\INTSeg{m_{x_0,R_T}(s)}{s}{\tau}{t}} \Big).
\end{equation*}
Combining these few latter facts together, we further obtain
\begin{equation}
\label{eq:gEst3}
\big| \dist_X(\Kpazo(\tau) \, ; \Qpazo(t)) - \dist_X(\Kpazo(\tau) \, ; \Qpazo(\tau)) \big| \leq \INTSeg{m_{x_0,R_T}(s)}{s}{\tau}{t},
\end{equation}
which along with \eqref{eq:gEst1} and \eqref{eq:gEst2} finally yields that
\begin{equation*}
|g(t) - g(\tau)| \leq \INTSeg{\Big( m_{\Kpazo}(s) + m_{x_0,R_T}(s) \Big)}{s}{\tau}{t}, 
\end{equation*}
for all times $\tau,t \in [0,T]$ satisfying $0 \leq \tau \leq t \leq T$, which equivalently means that $g (\cdot) \in \AC([0,T],\R_+)$. 

We finally conclude by showing that whenever $\Qpazo : [0,T] \tto X$ is left absolutely continuous, then the set-valued map 
\begin{equation*}
\Ecal : t \in [0,T] \tto \Big\{ \alpha \in \R_+ ~\, \text{s.t.}~ \alpha = g(t) + r ~~ \text{for some $r \geq 0$} \Big\}
\end{equation*}
is left absolutely continuous as well. To do so, let $x_0 \in \Kpazo(0)$ and $R_T >0$ be as above, fix an element $\alpha_{\tau} \in \Ecal(\tau)$, and observe that 
\begin{equation*}
g(\tau) = \dist_X \Big( \Kpazo(\tau) \, ; \Qpazo(\tau) \cap \B_X(x_0,R_T) \Big) \leq \alpha_{\tau}
\end{equation*}
by construction. It then follows from elementary applications of the triangle inequality that
\begin{equation*}
\begin{aligned}
g(t) & = \dist_X(\Kpazo(t) \, ; \Qpazo(t)) \\
& \leq \dist_X(\Kpazo(t) \, ; \Kpazo(\tau)) + \dist_X \Big( \Kpazo(\tau) \, ; \Qpazo(\tau) \cap \B_X(x_0,R_T) \Big) + \dist_X \Big( \Qpazo(\tau) \cap \B_X(x_0,R_T) \, ; \Qpazo(t) \Big) \\
& \leq \alpha_{\tau} + \INTSeg{\Big( m_{\Kpazo}(s) + m_{x_0,R_T}(s) \Big)}{s}{\tau}{t}, 
\end{aligned}
\end{equation*}
for all times $t \in [0,T]$ such that $\tau \leq t$. In particular, we have shown that 
\begin{equation*}
\Delta_{\alpha,R}(\Ecal(\tau) \, ; \Ecal(t)) \leq \INTSeg{\Big( m_{\Kpazo}(s) + m_{x_0,R_T}(s) \Big)}{s}{\tau}{t}
\end{equation*}
for all times $0 \leq \tau \leq t \leq T$, every $\alpha \in \R_+$ and each $R>0$, which yields the desired claim 
\end{proof}


\bibliographystyle{plain}
{\footnotesize
\bibliography{../../ControlWassersteinBib}
}

\end{document}